\documentclass[11pt]{article}
\usepackage{amssymb,amsmath,fullpage}
\usepackage{color}
\title{\boldmath\bf A determinant formula associated with the elliptic 
hypergeometric integrals of type $BC_n$}
\author{
Masahiko ITO\footnote{
Department of Mathematical Sciences,
University of the Ryukyus, 
Okinawa 903-0213, Japan}
\  and 
Masatoshi NOUMI\footnote{
Department of Mathematics, Kobe University, 
Rokko, Kobe 657-8501, Japan
}
}
\date{
}
\newtheorem{thm}{Theorem}[section]

\newtheorem{cor}[thm]{Corollary}
\newtheorem{lem}[thm]{Lemma}
\newtheorem{rem}[thm]{Remark}

\numberwithin{equation}{section}
\newcommand{\bN}{\mathbb{N}}

\newcommand{\bR}{\mathbb{R}}
\newcommand{\bC}{\mathbb{C}}
\newcommand{\bT}{\mathbb{T}}
\newcommand{\cA}{\mathcal{A}}
\newcommand{\cH}{\mathcal{H}}
\newcommand{\cI}{\mathcal{I}}

\newcommand{\cO}{\mathcal{O}}
\newcommand{\hf}{\frac{1}{2}}

\newcommand{\ds}{\displaystyle}
\newcommand{\ts}{\textstyle}
\newcommand{\ep}{\epsilon}

\newcommand{\la}{\langle}
\newcommand{\ra}{\rangle}

\newcommand{\frS}{\mathfrak{S}}
\newcommand{\pr}[1]{\left\{#1\right\}}
\newcommand{\prm}[2]{\left\{\,#1\ \big|\ #2\,\right\}}
\newcommand{\prmts}[2]{\{#1\ |\ #2\}}

\newcommand{\isom}{\stackrel{\sim}{\to}}
\newcommand{\bTC}{(\bC^{\ast})}
\newcommand{\bTR}{\bT}
\newenvironment{proof}[1]{
\par\smallskip\noindent{\sl {#1}\,$:$\ }
}{\hfill $\square$
\par\smallskip}
\begin{document}
\maketitle
\allowdisplaybreaks
\begin{abstract}
We establish a determinant formula for the bilinear form associated with the elliptic hypergeometric integrals of type $BC_n$
by studying the structure of $q$-difference equations to be satisfied by them.  
The determinant formula is proved by combining the $q$-difference equations of the determinant and its asymptotic analysis along the singularities.  
The elliptic interpolation functions of type $BC_n$ are essentially used in the study of  
the $q$-difference equations. 
\end{abstract}
\tableofcontents
\section{Introduction} 
\label{section:1}
The purpose of this paper is to provide a foundation for the study of $q$-difference equations satisfied by
the elliptic hypergeometric integrals of type $BC_n$, and to establish a determinant 
formula for the bilinear form associated with them. 
We summarize below the main results of this paper. 
\par
Throughout this paper we fix two bases $p,q\in\bC^\ast$ with $|p|<1$, $|q|<1$, and 
use the notation of the {\em multiplicative theta function} $\theta(u;p)$ 
and the {\em elliptic gamma function} $\Gamma(u;p,q)$ specified by 
\begin{equation*}
\theta(u;p)=(u;p)_\infty(p/u;p)_\infty, \quad 
\Gamma(u;p,q)=\frac{(pq/u;p,q)_\infty}{(u;p,q)_\infty}\quad(u\in\bC^\ast)
\end{equation*}
where $(u;p)_\infty=\prod_{i=0}^{\infty}(1-p^i u)$ and 
$(u;p,q)_\infty=\prod_{i,j=0}^{\infty}(1-p^iq^ju)$.
They satisfy
\begin{equation*}
\theta(pu;p)=-u^{-1}\theta(u;p)
\ \ \mbox{and}\ \
\Gamma(qu;p,q)=\theta(u;p)\Gamma(u;p,q),\ 
\Gamma(pu;p,q)=\theta(u;q)\Gamma(u;p,q),
\end{equation*}
respectively.
In the $BC_n$ context of this paper, 
we define the function $e(u,v;p)$ of 
two variables by 
\begin{equation*}
e(u,v;p)=u^{-1}\theta(uv;p)\theta(uv^{-1};p). 
\end{equation*}
We also use the notation of $t$-shifted factorials
\begin{equation*}
\begin{split}
\theta(u;p)_{t,k}&=
\theta(u;p)\theta(tu;p)\cdots
\theta(t^{k-1}u;p),
\\
e(u,v;p)_{t,k}&=
e(u,v;p)e(tu,v;p)\cdots
e(t^{k-1}u,v;p)
\end{split}
\end{equation*}
for $k=0,1,2,\ldots$.  
\subsection{Elliptic hypergeometric integrals of type $BC_n$}
\label{subsection:1.1}
Let $z=(z_1,\ldots,z_n)$ be the canonical coordinates of the $n$-dimensional algebraic 
torus $\bTC^{n}$.  We denote by $W_n=\pr{\pm1}^n\rtimes\frS_n$ the hyperoctahedral group 
of degree $n$ (the Weyl group of type $BC_n$), acting on $\bTC^n$ through permutations and inversions of the coordinates 
$z_i$ ($i=1,\ldots,n$). Fixing a constant $t\in \bC^*$ with $|t|<1$, 
for each $m=1,2,\ldots$ we define a $W_n$-invariant 
meromorphic function $\Phi(z)$ 
on $\bTC^n$ with parameters $a=(a_1,\ldots,a_m)\in\bTC^m$ by 
\begin{equation}\label{eq:Phi}
\Phi(z)=\Phi(z;a;p,q)=
\prod_{i=1}^{n}
\frac{\prod_{k=1}^{m}\Gamma(a_kz_i^{\pm1};p,q)}
{\Gamma(z_i^{\pm 2};p,q)}
\prod_{1\le i<j\le n}
\frac{\Gamma(tz_i^{\pm1}z_j^{\pm1};p,q)}
{\Gamma(z_i^{\pm1}z_j^{\pm 1};p,q)}. 
\end{equation}
The double sign $\pm$ indicates, here and hereafter, 
the product of factors with possible combinations of signs as 
\begin{equation*}
f(u^{\pm1})=f(u)f(u^{-1}),\quad
f(u^{\pm1}v^{\pm1})=f(uv)f(uv^{-1})f(u^{-1}v)f(u^{-1}v^{-1}). 
\end{equation*}
We now suppose that the parameters satisfy the conditions  
$|a_k|<1$ $(k=1,\ldots,m)$.  
Noting that $\Phi(z)$ is holomorphic in a neighborhood of the 
real torus
\begin{equation*}
\bTR^n=\prm{z=(z_1,\ldots,z_n)\in\bTC^n}{|z_i|=1
\quad(i=1,\ldots,n)}, 
\end{equation*}
for each holomorphic function $f(z)$ on $\bTC^n$ 
we consider the {\em elliptic hypergeometric integral}
\begin{equation}
\label{eq:EHI}
\la f \ra_{\Phi}=\int_{\bTR^n}
f(z)\Phi(z)\omega_n(z),\qquad
\omega_n(z)=\frac{1}{(2\pi\sqrt{-1})^n}\frac{dz_1\cdots dz_n}
{z_1\cdots z_n}. 
\end{equation}
\par
When $m=6$ and $f=1$, the integral \eqref{eq:EHI} 
can be evaluated as
\begin{equation}
\label{eq:detK(axy);r=1*}
\begin{split}
\la\,1\ra_{\Phi}=
\int_{\bTR^n}\Phi(z)\omega_n(z)
=
\frac{2^nn!}{(p;p)_\infty^n(q;q)_\infty^n}
\prod_{i=0}^{n-1}
\bigg(\frac{\Gamma(t^{i+1};p,q)
}
{\Gamma(t;p,q)
}
\prod_{1\le k<l\le 6}
\Gamma(t^ia_ka_l;p,q)
\bigg),
\end{split}
\end{equation}
provided that 
$a_1\cdots a_6t^{2n-2}=pq$.
This is known as the evaluation formula of the {\em elliptic Selberg integral of type $BC_n$} due to van Diejen and Spiridonov \cite{vDS}. 
If we denote by $I(a_1,\ldots,a_6)$ the left -hand side of \eqref{eq:detK(axy);r=1*}, then it satisfies the $q$-difference equations 
\begin{equation}\label{eq:qDE-0}
I(a_1,\ldots,qa_k,\ldots,a_5,q^{-1}a_6)=I(a_1,\ldots,a_5,a_6)
\prod_{i=0}^{n-1}\prod_{\substack{1\le l\le 5\\[1pt] l\ne k}}
\frac{\theta(a_la_kt^{i};p)}{\theta(q^{-1}a_la_6t^{i};p)}
\end{equation}
for $k=1,\ldots,5$.  
In \cite{INSumBC,INIntBC}, we gave a proof of the formula \eqref{eq:detK(axy);r=1*} based on the $q$-difference equations \eqref{eq:qDE-0}
and singularity analysis of the integral. 
\par
When we study the general case where $m\ge 6$, we need to consider an appropriate class of functions $f(z)$ in the integral \eqref{eq:EHI}. 
In the following, we introduce two vector spaces $\cH^{(p)}_{r-1,n}$, $\cH^{(q)}_{r-1,n}$ 
of $W_n$-invariant quasi-periodic functions with respect to $p$ and $q$, 
respectively, and investigate 
the $\bC$-bilinear form 
\begin{equation*}
\la\,,\,\ra_{\Phi}:\ \cH^{(p)}_{r-1,n}\times \cH^{(q)}_{r-1,n}\to\bC
\end{equation*}
defined by 
\begin{equation}\label{eq:bilin}
\begin{split}
&\la\,f,g\ra_{\Phi}=\la fg\ra_{\Phi}
=\int_{\bTR^n}f(z)g(z)\Phi(z)\omega_n(z)
\qquad
(f\in\cH^{(p)}_{r-1,n}, g\in \cH^{(q)}_{r-1,n}). 
\end{split}
\end{equation}
This bilinear form is a $BC_n$ elliptic extension of the {\em hypergeometric pairing} of 
Tarasov and Varchenko \cite{TV97}, studied in the context of the $q$-KZ equations (of type $A_n$). 
From the viewpoint of $q$-difference de Rham theory, 
$\cH_{r-1,n}^{(p)}$ plays the role of the space of {\em $n$-cocycles}, and 
$\cH_{r-1,n}^{(q)}$ that of {\em $n$-cycles}, respectively.
One of the main goals of this paper is to provide an explicit formula 
for the determinant 
of this bilinear form with respect to a pair of 
certain interpolation bases 
for 
$\cH^{(p)}_{r-1,n}$ and 
$\cH^{(q)}_{r-1,n}$. 
\subsection{Interpolation basis for a space of 
$W_n$-invariant quasi-periodic functions.}
In what follows we denote by $T_{p,z_i}$ the $p$-shift operator 
with respect to the variable $z_i$: 
\begin{equation*}
T_{p,z_i}f(z_1,\ldots,z_n)=f(z_1,\ldots,pz_i,\ldots,z_n)\quad(i=1,\ldots,n). 
\end{equation*}
As in \cite{INSlaterBC}, 
for each $s=1,2,\ldots$ we introduce 
the $\bC$-vector space 
\begin{equation*}
\begin{split}
\cH_{s-1,n}^{(p)}=\prm{f(z)\in\cO(\bTC^n)^{W_n}}
{T_{p,z_i}f(z)=f(z)(pz_i^2)^{-s+1}\quad(i=1,\ldots,n)}
\end{split}
\end{equation*}
of all $W_n$-invariant holomorphic functions on $\bTC^n$ 
with quasi-periodicity {\em of degree $s-1$}. 
This vector space has dimension 
$\binom{n+s-1}{s-1}$, which coincides with the cardinality 
of the set 
\begin{equation*}
Z_{s,n}=\prm{\mu=(\mu_1,\ldots,\mu_s)\in\bN^s}
{|\mu|=\mu_1+\cdots+\mu_s=n}
\end{equation*}
of multiindices, where $\bN=\pr{0,1,2,\ldots}$. 
\par
Fixing a set $c=(c_1,c_2,\ldots,c_s)\in\bTC^s$ of generic parameters, 
for each $\mu\in Z_{s,n}$ we consider the reference point
\begin{equation*}
(c)_{t,\mu}
=(c_1,tc_1,\ldots,t^{\mu_1-1}c_1,c_2,tc_2,\ldots,t^{\mu_2-1}c_2,\ldots, 
c_s,tc_s,\ldots,t^{\mu_s-1}c_s)\in\bTC^n
\end{equation*}
in $\bTC^n$, where the indexing set $\pr{1,\ldots,n}$ is divided into $s$ blocks of size $\mu_1,\mu_2,\ldots,\mu_s$. 
Then, 
for the set 
$\prmts{(c)_{t,\mu}}{\mu\in Z_{s,n}}$ 
of reference points, 
it is known by \cite{INSlaterBC} that there exists 
a unique $\bC$-basis 
$\prmts{E_{\mu}(c;z;p)}{\mu\in Z_{s,n}}$ of $\cH_{s-1,n}^{(p)}$ 
satisfying the interpolation condition 
\begin{equation*}
E_{\mu}(c;(c)_{t,\nu};p)=\delta_{\mu,\nu}\qquad(\mu,\nu\in Z_{s,n})
\end{equation*}
where $\delta_{\mu,\nu}$ denotes the Kronecker delta; 
we call this $\prmts{E_{\mu}(c;z;p)}{\mu\in Z_{s,n}}$ 
the {\em interpolation basis} 
of $\cH_{s-1,n}^{(p)}$ with 
respect to $c\in\bTC^s$. 
Note that each $f(z)\in\cH^{(p)}_{s-1,n}$ is expanded as 
\begin{equation*}
f(z)=\sum_{\mu\in Z_{s,m}} f((c)_{t,\mu})E_{\mu}(c;z;p)
\end{equation*}
in terms of the interpolation basis.  
Fundamental properties of this 
interpolation basis are summarized below in Section \ref{section:2}. 
\subsection{Bilinear form associated with the elliptic hypergeometric integral}
Returning to the meromorphic function $\Phi(z)$ of \eqref{eq:Phi}, 
we assume that $m=2r+4$ $(r=1,2,\ldots)$.  
With respect to the two bases $p$, $q$, 
we consider the two vector spaces $\cH^{(p)}_{r-1,n}$, $\cH^{(q)}_{r-1,n}$ 
of $W_n$-invariant quasi-periodic functions of degree $r-1$, 
respectively, and define 
the $\bC$-bilinear form 
$
\la\,,\,\ra_{\Phi}:\cH^{(p)}_{r-1,n}\times \cH^{(q)}_{r-1,n}\to\bC
$
by \eqref{eq:bilin}. 
We propose 
an explicit formula for the determinant 
of this bilinear form with respect to a pair of interpolation bases 
for 
$\cH^{(p)}_{r-1,n}$ and 
$\cH^{(q)}_{r-1,n}$. 
\par
Fixing generic 
parameters $x=(x_1,\ldots,x_r)$ and $y=(y_1,\ldots,y_r)$, 
we take the interpolation bases 
for these two vector spaces with respect to $x$ and $y$ 
respectively: 
\begin{equation*}
\begin{split}
\cH^{(p)}_{r-1,n}=\bigoplus_{\mu\in Z_{r,n}}\bC\,E_{\mu}(x;z;p),
\qquad
\cH^{(q)}_{r-1,n}=\bigoplus_{\mu\in Z_{r,n}}\bC\,E_{\mu}(y;z;q). 
\end{split}
\end{equation*}
For each pair $(\mu,\nu)\in Z_{r,n}\times Z_{r,n}$, 
we introduce 
the elliptic hypergeometric integral
\begin{equation*}
\begin{split}
&K_{\mu,\nu}(a;x,y)
=
\la E_{\mu}(x;z;p),E_{\nu}(y;z;q)\ra_{\Phi}
\\
&=\int_{\bTR^n}
E_{\mu}(x;z;p)\,E_{\nu}(y;z;q)\Phi(z;a;p,q)\omega_{n}(z)
\qquad(\mu,\nu\in Z_{r,n}),  
\end{split}
\end{equation*}
which is a holomorphic function on the domain $|a_k|<1$ $(k=1,\ldots,m)$ of $(\bC^*)^m$. 
We consider the 
$\binom{n+r-1}{r-1}\times\binom{n+r-1}{r-1}$ matrix 
\begin{equation*}
K(a;x,y)
=\big(K_{\mu,\nu}(a;x,y)\big)_{\mu,\nu\in Z_{r,n}}
\end{equation*}
as a matrix representation of the bilinear form 
\eqref{eq:bilin} with respect to the interpolation bases. 
\begin{thm}\label{thm:1A}
Set $m=2r+4$ $(r=1,2,\ldots)$.  
Under the balancing condition 
$a_1\cdots a_mt^{2n-2}=pq$ for the parameters, 
the determinant of the 
$\binom{n+r-1}{r-1}\times \binom{n+r-1}{r-1}$ matrix 
$K(a;x,y)=(K_{\mu,\nu}(a;x,y))_{\mu,\nu\in Z_{r,n}}$ 
is given explicitly by 
\begin{equation}
\label{eq:detK(axy)}
\begin{split}
&
\det K(a;x,y)
\\
&
=
\bigg(
\frac{2^nn!}{(p;p)_\infty^n(q;q)_\infty^n}
\bigg)^{\binom{n+r-1}{r-1}}
\prod_{i=0}^{n-1}
\bigg(\frac{\Gamma(t^{i+1};p,q)^r
}
{\Gamma(t;p,q)^r
}
\prod_{1\le k<l\le m}
\Gamma(t^ia_ka_l;p,q)
\bigg)^{\binom{n-i+r-2}{r-1}}
\\
&\qquad
\cdot
\Bigg(
\prod_{0\le i+j<n}
\prod_{1\le k<l\le r}
\big(
e(t^ix_k,t^jx_l;p)
e(t^iy_k,t^jy_l;q)
\big)^{\binom{n-i-j+r-3}{r-2}}
\Bigg)^{\!\!-1}.
\end{split}
\end{equation}
\end{thm}
\begin{rem}{\rm
We comment on two special cases of Theorem \ref{thm:1A}. 
When $r=1$ (i.e.,~$m=6$), the matrix size of $K(a;x,y)$ reduces to 1 and Theorem \ref{thm:1A} gives 
the formula \eqref{eq:detK(axy);r=1*}. 
On the other hand, when $n=1$, the matrix size of $K(a;x,y)$ reduces to $r$,  
and Theorem \ref{thm:1A} means that, 
for $x,y\in (\mathbb{C}^*)^r$
\begin{equation*}
\begin{split}
\det\Big(\la E_{\epsilon_i}(x;z;p),E_{\epsilon_j}(y;z;q)\ra_\Phi\Big)_{i,j=1}^r
=\frac{2^r}{(p;p)_\infty^r(q;q)_\infty^r}
\frac{
\prod_{1\le k<l\le 2r+4}
\Gamma(a_ka_l;p,q)
}{
\prod_{1\le k<l\le r}
e(x_k,x_l;p)
e(y_k,y_l;q)
}
\end{split}
\end{equation*}
under the condition $a_1\cdots a_{2r+4}=pq$ with $|a_k|<1$ $(k=1,\ldots,2r+4)$, 
where $E_{\epsilon_i}(x;z;p)=\prod_{1\le k\le r\atop k\ne i}e(z,x_k;p)/e(x_i,x_k;p)$ $(z\in \mathbb{C}^*)$ for $i=1,\ldots, r$, 
as mentioned in \eqref{eq:E(c;u);n=1}. 
This is also a special case of the determinant formula {\it of type I}\, found by Rains and Spiridonov \cite{RS}.
}
\end{rem}
\par
\medskip
Under the balancing condition, it turns out that  
the integrals $K_{\mu,\nu}(a;x,y)$ are continued to meromorphic functions on the whole hypersurface 
$a_1\cdots a_mt^{2n-2}=pq$ of $(\bC^*)^m$, 
provided that $|p|$ or $|q|$ is sufficiently small.
\begin{thm}\label{thm:1B}
Suppose that $|p|<|t|^{2n-2}$. Under the condition 
$a_1\cdots a_mt^{2n-2}=pq$ $(m=2r+4),$ 
the matrix $K(a;x,y)$ satisfies a system of $q$-difference equations with respect of $a=(a_1,\ldots,a_m)$ 
in the form 
\begin{equation}
\label{eq:1Bq}
T_{q,a_k}T_{q,a_l}^{-1}K(a;x,y)=M^{k,l}(a;x;p,q)K(a;x,y)\quad (1\le k<l\le m),
\end{equation}
where 
\begin{equation*}
M^{k,l}(a;x;p,q)=\big(M^{k,l}_{\mu,\nu}(a;x;p,q)\big)_{\mu,\nu\in Z_{r,n}} \quad (1\le k<l\le m)
\end{equation*}
are $\binom{n+r-1}{r-1}\times\binom{n+r-1}{r-1}$ matrices whose entries are meromorphic functions in $a$. 
\end{thm}
\par
Theorem \ref{thm:1A} indicates that $K(a;x,y)$ is a fundamental matrix of solutions of the $q$-difference system \eqref{eq:1Bq}. 
\begin{rem}{\rm 
If $|q|<|t|^{2n-2}$, the matrix $K(a;x,y)$ also satisfies the system of $p$-difference equations
\begin{equation}
T_{p,a_k}T_{p,a_l}^{-1}K(a;x,y)=K(a;x,y)M^{k,l}(a;y;q,p)^\mathrm{t}\quad (1\le k<l\le m), 
\label{eq:1Bp}
\end{equation}
by symmetry with respect to $p$ and $q$.
Although we imposed the condition $|p|<|t|^{2n-2}$ in Theorem \ref{thm:1B}, it may be possible to relax this restriction on $|p|$. 
}
\end{rem}
\subsection{Plan of this paper}
This paper is organized as follows. 
In Section \ref{section:2}, we give a review of the elliptic interpolation functions of type $BC_n$ based on 
our previous work \cite{INSlaterBC}.  
We also propose explicit formulas for 
the special values and for the transition coefficients between interpolation 
bases with different parameters.  
After this preparation, in Section \ref{section:3} we formulate 
the method of $q$-difference de Rham theory in terms of the 
$q$-difference coboundary operator 
$\nabla_{\rm sym}^{\Psi}:\mathcal{H}_{s-1,n-1}^{(p)}\to \mathcal{H}_{s-1,n}^{(p)}$
$(m=2s+2)$.  The cokernel of this operator, denoted by $H^{\Psi}_{\mathrm{sym}}$, 
plays the role of the ``symmetrized $n$th $q$-difference de Rham cohomology 
group''.  
This method enables us to describe linear relations among the hypergeometric integrals 
on the algebraic level
(Theorem \ref{thm:E cong RE}).  We prove in particular that 
$\dim_{\bC}H_{\mathrm{sym}}^{\Psi}=\binom{n+s-2}{s-2}$
giving a $\bC$-basis consisting of interpolation functions
(Theorem \ref{thm:modnabla}).
In Section \ref{section:4}, we apply the results of Section \ref{section:3} to 
derive the system of $q$-difference equations for the 
elliptic hypergeometric integrals with respect to an interpolation 
basis of $H_{\mathrm{sym}}^{\Psi}$ (Theorems \ref{thm:detcA}).  
Taking the determinant of the coefficient matrices, we obtain 
the system of $q$-difference equations for the determinant 
of the bilinear form $\la f, g\ra_{\Phi}$. 
In particular, we see that the determinant is expressed 
as a product of elliptic gamma functions, up to an unknown 
constant $c_{r,n}$ (Theorem \ref{thm:detcI}). 
We determine in Section \ref{section:5} the explicit value of the unknown constant 
through the recursive structure of asymptotic behavior of the integrals along the singularities. 
In the final section, 
we investigate the limiting cases of our main theorems as $p\to 0$,  
and derive determinant formulas for two types of $q$-hypergeometric 
integrals of type $BC_n$. 
\par\medskip
We expect that the results of this paper will be used as a foundation 
for further analysis of elliptic hypergeometric integrals. 
\section{\boldmath Elliptic interpolation functions of type $BC_n$}
\label{section:2}
As in Section \ref{subsection:1.1}, we consider the $\bC$-vector 
space $\cH_{s-1,n}^{(p)}$ ($s=1,2,\ldots$)
of all quasi-periodic $W_n$-invariant holomorphic functions 
on $\bTC^n$ of degree $s-1$ with respect to $p$.  
In this section we recall from \cite{INSlaterBC} 
basic properties of 
the interpolation functions 
$E_{\mu}(c;z;p)$ $(\mu\in Z_{s,n})$. 
We remark that our elliptic interpolation functions 
for $s=2$ are essentially the special cases of interpolation theta functions of Coskun--Gustafson \cite{CG} and Rains \cite{R} {\em attached to single columns} of partitions. In fact, $E_{\mu}(c;z;p)$ $(\mu\in Z_{s,n})$ for $s=2$ are compared explicitly with the functions of \cite{CG} and \cite{R}, respectively, in 
\cite[Introduction]{INSumBC}. 
Since our functions $E_{\mu}(c;z;p)$
for general $s$ 
are defined by an interpolation property of Lagrange type, 
they are different in nature 
from those of Coskun--Gustafson and Rains which are based on the 
triangularity with respect to partitions. 
It would be an interesting problem to clarify how these two approaches are 
related to each other. 

We also propose explicit formulas for 
the special values $E_{\mu}(c;(u)_{t,n};p)$ 
and for the transition coefficients between interpolation 
bases with different $c$ parameters.  
Throughout this section we use the base $p$ only, 
and simply set $E_{\mu}(c;z)=E_{\mu}(c;z;p)$,
$\theta(u)=\theta(u;p)$ and $e(u,v)=e(u,v;p)$. 
\subsection{Recursion formula}
When $n=1$, the interpolation functions are 
parametrized by the 
canonical basis $\ep_1,\ldots,\ep_s$ of $\bN^s$, and 
given explicitly as 
\begin{equation}
\label{eq:E(c;u);n=1}
E_{\ep_k}(c;u)=
\prod_{\substack{1\le l\le s\\[1pt] l\ne k}}
\frac{e(u,c_l)}{e(c_k,c_l)}
=
\prod_{\substack{1\le l\le s\\[1pt] l\ne k}}
\frac{\theta(c_l/u)\theta(c_lu)}{\theta(c_l/c_k)\theta(c_lc_k)}.
\end{equation}
When $n\ge 2$, they satisfy the recursion formula 
\begin{equation}
\label{eq:recursion}
\begin{split}
E_{\mu}(c;z)=\sum_{\substack{1\le k\le s\\[1pt] \mu_k>0}}
E_{\mu-\ep_k}(c;z_1,\ldots,z_{n-1})
E_{\ep_k}(t^{\mu-\ep_k}c;z_n)
\quad(\mu\in Z_{s,n}),
\end{split}
\end{equation}
where $t^{\nu}c=(t^{\nu_1}c_1,\ldots,t^{\nu_s}c_s)$.  
This formula apparently depends on the ordering 
of the variables $z_1,\ldots,z_n$, while $E_{\mu}(c;z)$ are 
$W_n$-invariant. 
\subsection{Explicit formula}
\begin{equation}\label{eq:explicit}
\begin{split}
E_{\mu}(c;z)&=
\sum_{k_1,\ldots,k_n\in\pr{1,\ldots,s}}
E_{\ep_{k_1}}(c;z_1)
E_{\ep_{k_2}}(t^{\ep_{k_1}}c;z_2)
\cdots
E_{\ep_{k_n}}(t^{\ep_{k_1}+\cdots+\ep_{k_{n-1}}}c;z_n)
\\
&=
\sum_{k_1,\ldots,k_n\in\pr{1,\ldots,s}}
\prod_{i=1}^{n}
\prod_{\substack{1\le l\le s\\[1pt] l\ne k_i}}
\frac{e(z_i,t^{\mu^{(i)}}c_l)}
{e(t^{\mu^{(i-1)}_{k_i}}c_{k_i},t^{\mu^{(i-1)}_l}c_l)},
\end{split}
\end{equation}
where 
$\mu^{(i)}=\ep_{k_1}+\cdots+\ep_{k_i}$ 
($i=1,\ldots,n$). 
\subsection{Interpolation functions on the vertices and the faces} 
The indexing set 
$Z_{s,n}$ for the interpolation functions are the lattice points 
on the $(s-1)$-simplex 
\begin{equation*}
\prm{\mu=(\mu_1,\ldots,\mu_s)\in\bR^s}{
\mu_1+\cdots+\mu_s=n;\quad
\mu_k\ge 0\ \ (k=1,\ldots,s)
}
\end{equation*}
in $\bR^{s}$. 
On the vertices $\mu=n\ep_k$ ($k=1,\ldots,s$), 
the interpolation functions are factorized as 
\begin{equation*}
E_{n\ep_k}(c;z)=\prod_{\substack{1\le l\le s\\[1pt] l\ne k}}
\frac{\prod_{i=1}^{n}e(z_i,c_l)} {e(c_k,c_l)_{n}}
\quad(k=1,\ldots,s),
\end{equation*}
where $e(u,v)_k=e(u,v)e(tu,v)\cdots e(t^{k-1}u,v)$ 
denotes the $t$-shifted factorial with respect to the first argument
($k=0,1,2,\ldots$). 
When $\mu\in Z_{s,n}$ is on the $s$th face ($\mu_s=0$), they 
are expressed as 
\begin{equation}\label{eq:Emusthface}
E_{(\mu_1,\ldots,\mu_{s-1},0)}(c_1,\ldots,c_s;z)
=
E_{(\mu_1,\ldots,\mu_{s-1})}(c_1,\ldots,c_{s-1};z)
\frac{\prod_{i=1}^{n}e(z_i,c_s)}
{\prod_{l=1}^{s-1}e(c_l,c_s)_{\mu_l}}
\end{equation}
in terms of the interpolation functions with $s-1$ parameters
$(c_1,\ldots,c_{s-1})$. 
\subsection{Dual Cauchy formula}
For each 
$\mu\in\bN^s$ we define a holomorphic function 
$F_{\mu}(c;w)$ in $s-1$ variables 
$w=(w_1,\ldots,w_{s-1})$ by 
\begin{equation*}
F_{\mu}(c;w)=\prod_{k=1}^{s}\prod_{l=1}^{s-1}e(c_k,w_l)_{\mu_k}. 
\end{equation*}
Then we have the {\it dual Cauchy formula}
\begin{equation}\label{eq:dualCauchy}
\prod_{j=1}^{n}\prod_{l=1}^{s-1}e(z_j,w_l)
=\sum_{\mu\in Z_{s,n}}E_{\mu}(c;z)F_{\mu}(c;w). 
\end{equation}
\subsection{Partition of the variables}
For the variables $z=(z',z'')\in (\mathbb{C}^*)^n$ divided into two parts 
$z'=(z_1,\ldots,z_{l})$ and $z''=(z_{l+1},\ldots,z_{n})$, the interpolation function $E_{\lambda}(c;z)$ is expressed as 
\begin{equation*}
E_{\lambda}(c;z)=\sum_{
\substack{|\mu|=l,\,|\nu|=n-l \\ \mu+\nu=\lambda}
 }\,E_{\mu}(c;z')E_{\nu}(t^{\mu}c;z''). 
\end{equation*}
\subsection{Special value}
\begin{thm}
\label{thm: special value}
 For each $\mu\in Z_{s,n}$ the special value of 
 $E_{\mu}(c;z)$ at 
$z=(u)_{t,n}=(u,tu,\ldots,t^{n-1}u) $
is given explicitly by 
\begin{equation}\label{eq:evalutn}
E_{\mu}(c;(u)_{t,n})
=\frac{
[t^{-n}]_{n}
\prod_{i=1}^{s}[u/c_i]_{n-\mu_i}
[t^{\mu_i}uc_i]_{n-\mu_i}
}{
\prod_{i,j=1}^{s}
[t^{-\mu_j}c_i/c_j]_{\mu_i}
\prod_{1\le i<j\le s}[c_ic_j]_{\mu_i+\mu_j}
},
\end{equation}
where 
$[u]=u^{-\hf}\theta(u)$ and $[u]_k=[u][tu]\cdots [t^{k-1}u]$. 
\end{thm}
\begin{rem} {\rm
The above special value $E_{\mu}(c;(u)_{t,n})$ is expressed as 
\begin{equation*}
E_{\mu}(c;(u)_{t,n})
=\Big(t^{-\binom{n}{2}}u^{-n}
\prod_{i=1}^{s}
t^{\binom{\mu_i}{2}}c_i^{\mu_i}\Big)^{s-1}
\frac{
\theta(t^{-n})_{n}
\prod_{i=1}^{s}\theta(u/c_i)_{n-\mu_i}
\theta(t^{\mu_i}uc_i)_{n-\mu_i}
}{
\prod_{i,j=1}^{s}
\theta(t^{-\mu_j}c_i/c_j)_{\mu_i}
\prod_{1\le i<j\le s}\theta(c_ic_j)_{\mu_i+\mu_j}
}
\end{equation*}
in terms of the $t$-shifted factorials 
$\theta(u)_k=\theta(u)\theta(tu)\cdots\theta(t^{k-1}u)$ 
$(k=0,1,2,\ldots)$ of the theta function $\theta(u)$.
}\end{rem}
\begin{proof}{Proof of Theorem \ref{thm: special value}}
By \eqref{eq:recursion}
the special value $E_{\mu}(c;(u)_{t,n})$ satisfies 
the recurrence formula 
\begin{equation*}
E_{\mu}(c;(u)_{t,n})=\sum_{\substack{1\le k\le s\\[1pt]\mu_k>0}}
E_{\mu-\ep_k}(c;(u)_{t,n-1})
\prod_{\substack{1\le l\le s\\[1pt]l\ne k}}\frac{e(t^{n-1}u,t^{\mu_l}c_l)}
{e(t^{\mu_k-1}c_k,t^{\mu_l}c_l)}. 
\end{equation*}
Denoting by $C_\mu$ $(n=|\mu|)$ the right-hand side of \eqref{eq:evalutn},
we verify that $C_{\mu}$ satisfies the same recurrence formula. 
Since we have 
\begin{equation*}
C_{\mu-\ep_k}=
C_{\mu}
\frac{
[t^{-\mu_k}][t^{\mu_k-1}uc_k]
}{
[t^{-n}][t^{n-1}uc_k]
}
\prod_{\substack{1\le j\le s\\[1pt] j\ne k}}
\frac
{
[t^{\mu_k-\mu_j-1}c_k/c_j]
[c_j/t^{\mu_k}c_k]
[t^{\mu_k+\mu_j-1}c_kc_j]
}
{
[t^{\mu_j}c_j/t^{\mu_k}c_k]
[t^{n-\mu_j-1}u/c_j]
[t^{n-1}uc_j]
}
\end{equation*}
for $\mu_k>0$, 
the above recurrence formula is equivalent to 
\begin{equation*}
\begin{split}
&\sum_{k=1}^{s}
\frac{
[t^{-\mu_k}][t^{\mu_k-1}uc_k]
}{
[t^{-n}][t^{n-1}uc_k]
}
\prod_{\substack{1\le j\le s\\[1pt] j\ne k}}
\frac
{
[t^{\mu_k-\mu_j-1}c_k/c_j]
[c_j/t^{\mu_k}c_k]
[t^{\mu_k+\mu_j-1}c_kc_j]
}
{
[t^{\mu_j}c_j/t^{\mu_k}c_k]
[t^{n-\mu_j-1}u/c_j]
[t^{n-1}uc_j]
}
\\
&\qquad\qquad\qquad\cdot
\prod_{\substack{1\le j\le s\\[1pt] j\ne k}}
\frac{
[t^{\mu_j}c_j/t^{n-1}u]
}{
[t^{\mu_j}c_j/t^{\mu_k-1}c_k]
}
\frac{
[t^{\mu_j}c_jt^{n-1}u]
}{
[t^{\mu_j}c_jt^{\mu_k-1}c_k]
}=1,
\end{split}
\end{equation*} 
namely, 
\begin{equation*}
\sum_{k=1}^{s}
\frac{
[t^{-\mu_k}]
[t^{\mu_k-1}uc_k]
}{
[t^{-n}]
[t^{n+\mu_k-1}uc_k]
}
\prod_{\substack{1\le j\le s\\[1pt] j\ne k}}
\frac
{
[c_j/t^{\mu_k}c_k]
}
{
[t^{\mu_j}c_j/t^{\mu_k}c_k]
}
=
\prod_{j=1}^{s}
\frac{
[t^{n-1}uc_j]
}{
[t^{n+\mu_j-1}uc_j]
}.
\end{equation*}
This formula follows from the partial fraction decomposition 
\begin{equation*}
\prod_{j=1}^{s}
\frac{[zc_j]}
{[zt^{\mu_j}c_j]}
=\sum_{k=1}^{s}
\frac{
[zt^{\mu_k-n}c_k][t^{-\mu_k}]
}{[zt^{\mu_k}c_k][t^{-n}]}
\prod_{\substack{1\le j\le s\\[1pt] j\ne k}}
\frac{[c_j/t^{\mu_k}c_k]
}{[t^{\mu_j}c_j/t^{\mu_k}c_k]}
\end{equation*}
as the special case where $z=t^{n-1}u$.
\end{proof}
\subsection{Change of parameters}
We consider to expand the interpolation function 
$E_{\mu}(c;z)=E_{\mu}(c_1,\ldots,c_s;z)$ in terms of another interpolation basis 
$E_{\nu}(c_1,\ldots,c_{s-1},d_s;z)$ $(\nu\in Z_{s,n})$ 
by replacing $c_s$ with $d_s$:
\begin{equation}\label{eq:transC}
E_{\mu}(c',c_s;z)=\sum_{\nu\in Z_{s,n}}
C_{\mu,\nu}(c';c_s,d_s)E_{\nu}(c',d_s;z),
\end{equation}
where 
$c'=(c_1,\ldots,c_{s-1})$. 
The transition coefficients $C_{\mu,\nu}(c';c_s,d_s)$ in this case are computed as follows:
\begin{equation*}
\begin{split}
C_{\mu,\nu}(c';c_s,d_s)&=
E_{\mu}(c; (c')_{t,\nu'},(d_s)_{t,\nu_s})
\\
&=
\sum_{
\substack{
|\lambda|=n-\nu_s,\,|\rho|=\nu_s
\\\lambda+\rho=\mu}
}
E_{\lambda}(c;(c')_{t,\nu'})
E_{\rho}(t^{\lambda}c; (d_s)_{t,\nu_s}),
\end{split}
\end{equation*}
where $\nu'=(\nu_1,\ldots,\nu_{s-1})$. 
Since 
\begin{equation*}
E_{\lambda}(c;(c')_{t,\nu'})
=
E_{\lambda}(c;(c)_{t,\nu-\nu_s\ep_s})
=\delta_{\lambda,\nu-\nu_s\ep_s},  
\end{equation*}
we have  
\begin{equation*}
C_{\mu,\nu}(c';c_s,d_s)
=
E_{\mu-\nu+\nu_s\ep_s}(t^{\nu-\nu_s\ep_s}c; (d_s)_{t,\nu_s})
=
E_{(\mu'-\nu',\mu_s)}(t^{\nu'}c',c_s; (d_s)_{t,\nu_s}),
\end{equation*}
which can be computed by Theorem \ref{thm: special value}.
\begin{thm}\label{thm:transC}
When we change the parameters  
$c=
(c',c_s)$ to 
$(c',d_s)$, we have 
\begin{equation*}
E_{\mu}(c',c_s;z)=
\sum_{\substack{\nu\in Z_{s,n}\\ \mu'\ge\nu'}}
C_{\mu,\nu}(c';c_s,d_s)E_{\nu}(c',d_s;z)
\quad (\mu\in Z_{s,n}),
\end{equation*}
where $\mu'\ge \nu'$ means that $\mu_k\ge \nu_k$ $(k=1,\ldots,s-1)$. 
The coefficients $C_{\mu,\nu}(c';c_s,d_s)$ are given explicitly by 
\begin{equation*}
\begin{split}
C_{\mu,\nu}(c';c_s,d_s)
&
=
\frac{
[t^{-\nu_s}]_{\nu_s}
}{
[t^{-\mu_s}]_{\mu_s}
}
\frac{
[d_s/t^{\mu_s}c_s]_{\nu_s}
}{
[d_s/t^{\mu_s}c_s]_{\mu_s}
}
\frac{
[c_sd_s]_{\nu_s}
}{
[c_sd_s]_{\mu_s}
}
\\
&\quad\cdot
\prod_{i=1}^{s-1}
\frac{[d_s/t^{\mu_i}c_i]_{\nu_s}}
{[c_s/t^{\mu_i}c_i]_{\mu_s}}
\frac{[c_i/t^{\mu_s}c_s]_{\nu_i}
}{[c_i/t^{\mu_s}c_s]_{\mu_i}}
\frac{
[d_s/t^{\nu_i}c_i]_{\nu_i}
}{
[d_s/t^{\mu_i}c_i]_{\mu_i}
}
\frac{
[c_ic_s]_{\nu_i}[c_id_s]_{\nu_i+\nu_s}
}{
[c_ic_s]_{\mu_i+\mu_s}[c_id_s]_{\mu_i}
}
\\
&\quad\cdot
\prod_{i,j=1}^{s-1}
\frac{
[c_i/t^{\mu_j}c_j]_{\nu_i}
}{[c_i/t^{\mu_j}c_j]_{\mu_i}}
\prod_{1\le i<j\le s-1}
\frac{
[c_ic_j]_{\nu_i+\nu_j}
}{
[c_ic_j]_{\mu_i+\mu_j}
}.
\end{split}
\end{equation*}
\hfill $\square$
\end{thm}
\begin{rem}{\rm
In terms of the ordinary 
$\theta(u)$ notation, we have 
\begin{equation*}
\begin{split}
C_{\mu,\nu}(c';c_s,d_s)
&
=
\Big(
t^{\binom{\mu_s}{2}-\binom{\nu_s}{2}}
c_s^{\mu_s}d_s^{-\nu_s}
\prod_{i=1}^{s-1}t^{\binom{\mu_i}{2}-\binom{\nu_i}{2}}
c_i^{\mu_i-\nu_i}
\Big)^{s-1}
\\
&\quad\cdot
\frac{
\theta(t^{-\nu_s})_{\nu_s}
}{
\theta(t^{-\mu_s})_{\mu_s}
}
\frac{
\theta(d_s/t^{\mu_s}c_s)_{\nu_s}
}{
\theta(d_s/t^{\mu_s}c_s)_{\mu_s}
}
\frac{
\theta(c_sd_s)_{\nu_s}
}{
\theta(c_sd_s)_{\mu_s}
}
\\
&\quad\cdot
\prod_{i=1}^{s-1}
\frac{\theta(d_s/t^{\mu_i}c_i)_{\nu_s}}
{\theta(c_s/t^{\mu_i}c_i)_{\mu_s}}
\frac{\theta(c_i/t^{\mu_s}c_s)_{\nu_i}
}{\theta(c_i/t^{\mu_s}c_s)_{\mu_i}}
\frac{
\theta(d_s/t^{\nu_i}c_i)_{\nu_i}
}{
\theta(d_s/t^{\mu_i}c_i)_{\mu_i}
}
\frac{
\theta(c_ic_s)_{\nu_i}\theta(c_id_s)_{\nu_i+\nu_s}
}{
\theta(c_ic_s)_{\mu_i+\mu_s}\theta(c_id_s)_{\mu_i}
}
\\
&\quad\cdot
\prod_{i,j=1}^{s-1}
\frac{
\theta(c_i/t^{\mu_j}c_j)_{\nu_i}
}{\theta(c_i/t^{\mu_j}c_j)_{\mu_i}}
\prod_{1\le i<j\le s-1}
\frac{
\theta(c_ic_j)_{\nu_i+\nu_j}
}{
\theta(c_ic_j)_{\mu_i+\mu_j}
}.
\end{split}
\end{equation*}
}\end{rem}
\par
In the succeeding section, Theorem \ref{thm:transC} will be applied to 
interpolation functions with respect to subsets of the parameters 
$a=(a_1,\ldots,a_m)$ of $\Phi(z)$ where $m=2s+2$.  For each 
subset $K\subseteq\pr{1,\ldots,m}$ of the indexing set 
with $|K|=s$, we consider the 
interpolation basis 
\begin{equation*}
E_{\mu}(a_{K};z)=E_{\mu}(a_{K};z;p)\qquad(\mu\in Z_{K,n})
\end{equation*}
of $\cH_{s-1,n}^{(p)}$ with respect to the parameters 
$a_{K}=(a_k)_{k\in K}$, where 
\begin{equation*}
Z_{K,n}=\prm{\mu=(\mu_k)_{k\in K}\in\bN^K}
{\ts |\mu|=\sum_{k\in K}\mu_k=n}.  
\end{equation*}
Let $I\subseteq\pr{1,\ldots,m}$ be a subset with $|I|=s-1$,
and choose two indices $k,l\in\pr{1,\ldots,m}\backslash I$.  Then 
Theorem \ref{thm:transC} is reformulated in terms of the transition 
between the two 
bases $E_{\mu}(a_{I\cup\pr{k}};z)$ and 
$E_{\nu}(a_{I\cup\pr{l}};z)$: 
\begin{equation*}
E_{\mu}(a_{I\cup\pr{k}};z)=\sum_{\nu\in Z_{I\cup\pr{l},n}}
C^{I;k,l}_{\mu,\nu}\,E_{\nu}(a_{I\cup\pr{l}};z)
\qquad (\mu\in Z_{I\cup\pr{k},n}). 
\end{equation*}
The transition coefficients $C^{I;k,l}_{\mu,\nu}$ are nonzero only if 
$\mu_{\widehat{k}}\ge \nu_{\widehat{l}}$, namely 
$\mu_i\ge \nu_i$ ($i\in I$);  
they are given explicitly by 
\begin{equation*}
\begin{split}
C_{\mu,\nu}^{I;k,l}
&=
\frac{
[t^{-\nu_l}]_{\nu_l}
}{
[t^{-\mu_k}]_{\mu_k}
}
\frac{
[a_l/t^{\mu_k}a_k]_{\nu_l}
}{
[a_l/t^{\mu_k}a_k]_{\mu_k}
}
\frac{
[a_ka_l]_{\nu_l}
}{
[a_ka_l]_{\mu_k}
}
\\
&\quad\cdot
\prod_{i\in I}
\frac{[a_l/t^{\mu_i}a_i]_{\nu_l}}
{[a_k/t^{\mu_i}a_i]_{\mu_k}}
\frac{[a_i/t^{\mu_k}a_k]_{\nu_i}
}{[a_i/t^{\mu_k}a_k]_{\mu_i}}
\frac{
[a_l/t^{\nu_i}a_i]_{\nu_i}
}{
[a_l/t^{\mu_i}a_i]_{\mu_i}
}
\frac{
[a_ia_k]_{\nu_i}[a_ia_l]_{\nu_i+\nu_l}
}{
[a_ia_k]_{\mu_i+\mu_k}[a_ia_l]_{\mu_i}
}
\\
&\quad\cdot
\prod_{i,j\in I}
\frac{
[a_i/t^{\mu_j}a_j]_{\nu_i}
}{[a_i/t^{\mu_j}a_j]_{\mu_i}}
\prod_{i,j\in I;\,i<j}
\frac{
[a_ia_j]_{\nu_i+\nu_j}
}{
[a_ia_j]_{\mu_i+\mu_j}
}.
\end{split}
\end{equation*}
We denote by $C^{I;k,l}=\big(C^{I;k,l}_{\mu,\nu}\big)_{\mu,\nu}$ the 
square matrix with row indices $\mu\in Z_{I\cup\pr{k},n}$ 
and column indices $\nu\in Z_{I\cup\pr{l},n}$ arranged by the 
partial ordering of $\mu_{\widehat{k}},\nu_{\widehat{l}}\in \bN^I$. 
Then the diagonal entries of $C^{I;k,l}$ are given by 
\begin{equation}\label{eq:transCmunu}
\begin{split}
C_{\mu,\nu}^{I;k,l}
&=
\prod_{i\in I}
\frac{[a_l/t^{\mu_i}a_i]_{\mu_k}}
{[a_k/t^{\mu_i}a_i]_{\mu_k}}
\frac{
[t^{\mu_i}a_ia_l]_{\mu_k}
}{
[t^{\mu_i}a_ia_k]_{\mu_k}
}
=
\prod_{i\in I}
\frac{e(a_l,t^{\mu_i}a_i)_{\mu_k}}
{e(a_k,t^{\mu_i}a_i)_{\mu_k}}. 
\end{split}
\end{equation}
for each pair $(\mu,\nu)$ such that 
$\mu_i=\nu_i$ $(i\in I)$ and $\mu_k=\nu_l$. 
Hence the determinant of 
the transition matrix 
$C^{I;k,l}=\big(C_{\mu,\nu}^{I;k,l}\big)_{\mu,\nu}$ 
is computed as 
\begin{equation}\label{eq:Cdet}
\begin{split}
\det C^{I;k,l}
&=
\prod_{0\le u+v\le n}
\prod_{i\in I}
\left(
\frac{e(a_l,t^{u}a_i)_{v}}
{e(a_k,t^{u}a_i)_{v}}
\right)^{\binom{n-u-v+s-3}{s-3}}
\\
&
=
\prod_{0\le u+v<n}
\prod_{i\in I}
\left(
\frac{e(t^{u}a_l,t^{v}a_i)}
{e(t^{u}a_k,t^{v}a_i)}
\right)^{\binom{n-u-v+s-3}{s-2}}. 
\end{split}
\end{equation}
\section{\boldmath $q$-Difference de Rham theory}
\label{section:3}
\subsection{Definition of $H_{\mathrm{sym}}^{\Psi}$}
In this section we suppose that $m=2r+4$ ($r=1,2,\ldots$) 
and set $s=r+1$. 
Fixing a nonzero $W_n$-invariant 
holomorphic function 
$g(z)\in\cH^{(q)}_{s-1,n}$ with quasi-periodicity 
of degree $s-1$, we set 
\begin{equation*}
\Psi(z)=\Phi(z)g(z). 
\end{equation*}
For each $i=1,\ldots,n$, 
the function 
\begin{equation*}
b_i^{\Psi}(z)=\frac{T_{q,z_i}\Psi(z)}{\Psi(z)}\qquad(i=1,\ldots,n)
\end{equation*}
is computed as follows: 
\begin{equation*}
\begin{split}
b_i^\Psi(z)
&=
\frac{\theta(q^{-1}z_i^{-2};p)\theta(q^{-2}z_i^{-2};p)
}{q^{s-1}z_i^{2s-2}\theta(z_i^2;p)\theta(qz_i^{2};p)}
\prod_{k=1}^{m}
\frac{\theta(a_kz_i;p)}{\theta(a_kq^{-1}z_i^{-1};p)}
\prod_{\substack{1\le j\le n\\[1pt] j\ne i}}
\frac{\theta(tz_iz_j^{\pm1};p)}
{\theta(z_iz_j^{\pm1};p)}
\frac{\theta(q^{-1}z_i^{-1}z_j^{\pm1};p)}
{\theta(tq^{-1}z_i^{-1}z_j^{\pm1};p)}
\\
&=
-
\frac{q^{-s}z_i^{-s}\theta(q^{-2}z_i^{-2};p)
}{z_i^{s}\theta(z_i^2;p)}
\prod_{k=1}^{m}
\frac{\theta(a_kz_i;p)}{\theta(a_kq^{-1}z_i^{-1};p)}
\prod_{\substack{1\le j\le n\\[1pt] j\ne i}}
\frac{\theta(tz_iz_j^{\pm1};p)}
{\theta(z_iz_j^{\pm1};p)}
\frac{\theta(q^{-1}z_i^{-1}z_j^{\pm1};p)}
{\theta(tq^{-1}z_i^{-1}z_j^{\pm1};p)}. 
\end{split}
\end{equation*}
Hence, $b_i^{\Psi}(z)$ is expressed as 
\begin{equation*}
b_i^{\Psi}(z)=-\frac{f_i^{+}(z)}{T_{q,z_i}f_i^{-}(z)},
\end{equation*}
where 
\begin{equation*}
f_i^{+}(z)=\frac{\prod_{k=1}^{m}\theta(a_kz_i;p)
}{z_i^{s}\theta(z_i^2;p)}\prod_{\substack{1\le j\le n\\[1pt] j\ne i}}
\frac{\theta(tz_iz_j^{\pm1};p)}{\theta(z_iz_j^{\pm1};p)},\qquad 
f_i^{-}(z)=f_i^{+}(z^{-1}).
\end{equation*}
This implies that 
\begin{equation*}
(1-T_{q,z_i})\left(\Psi(z)f_i^{-}(z)\right)
=
\Psi(z)
\left(f_i^{-}(z)-b_i^{\Psi}(z)T_{q,z_i}f_i^{-}(z)\right)
=\Psi(z)(f_i^{-}(z)+f_i^{+}(z)). 
\end{equation*}
Note that $f_i^{+}(z)$ and $f_i^{-}(z)$ have the following 
quasi-periodicity with respect to the $p$-shifts: 
\begin{equation*}
\begin{array}{lll}
&
T_{p,z_i}f_i^{+}(z)=f_i^{+}(z)\,(pz_i^2)^{-s+1}
(t^{2n-2}a_1\cdots a_m)^{-1},\quad
&T_{p,z_j}f_i^{+}(z)=f_i^{+}(z)\quad(j\ne i),
\\[4pt]
&
T_{p,z_i}f_i^{-}(z)=f_i^{-}(z)(pz_i^2)^{-s+1}
(t^{2n-2}a_1\cdots a_m),\quad
&T_{p,z_j}f_i^{-}(z)=f_i^{-}(z)\quad(j\ne i). 
\end{array}
\end{equation*}
If the parameters satisfy the 
balancing condition
$t^{2n-2}a_1\cdots a_m=1$, then we have 
\begin{equation*}
T_{p,z_i}f_i^{+}(z)=f_i^{+}(z)\,(pz_i^2)^{-s+1},\quad
T_{p,z_i}f_i^{-}(z)=f_i^{-}(z)(pz_i^2)^{-s+1}, 
\end{equation*}
so that 
$f_i^{+}(z)$ and $f_i^{-}(z)$ have the same quasi-periodicity. 
\begin{lem}
Under the balancing condition 
$t^{2n-2}a_1\cdots a_m=1$ $(m=2s+2)$, 
for each $\varphi\in \cH^{(p)}_{s-1,n-1}$,  
\begin{equation}\label{eq:defpsi}
\begin{split}
\psi(z)=\sum_{i=1}^{n}(f_i^{+}(z)+f_i^{-}(z))\varphi(z_{\,\widehat{i}}), 
\quad z_{\,\widehat{i}}=(z_1,\ldots,z_{i-1},z_{i+1},\ldots,z_{n}), 
\end{split}
\end{equation}
belongs to $\cH_{s-1,n}^{(p)}$.
\end{lem}
From the definition \eqref{eq:defpsi}, it directly follows that 
$\psi(z)$ is a 
$W_n$-invariant meromorphic function with quasi-periodicity 
of degree $s-1$; $\psi(z)$ is in fact holomorphic on $(\bC^\ast)^n$, 
since $\psi(z)$ is $W_n$-invariant and 
$\Delta_{C}(z;p)\psi(z)$ is holomorphic, where 
\begin{equation*}
\Delta_{C}(z;p)=
\prod_{i=1}^nz_i\theta(z_i^{-2};p)
\prod_{1\le i<j\le n}z_i\theta(z_i^{-1}z_j^{\pm1};p)
\end{equation*}
denotes the elliptic version of the Weyl denominator of type $C_n$.   
\par\medskip
In view of this lemma, we define the $\bC$-linear mapping 
$\nabla_{\mathrm{sym}}^{\Psi}:\ \cH^{(p)}_{s-1,n-1}\to \cH^{(p)}_{s-1,n}$ 
by 
\begin{equation*}
(\nabla_{\mathrm{sym}}^{\Psi}\varphi)(z)
=
\sum_{i=1}^{n}(f_i^{+}(z)+f_i^{-}(z))\varphi(z_{\,\widehat{i}})
\qquad(\varphi\in\cH_{s-1,n-1}^{(p)}).  
\end{equation*}
Since  $\nabla^{\Psi}_{\mathrm{sym}}\varphi$ is rewritten as 
\begin{equation*}
(\nabla_{\mathrm{sym}}^{\Psi}\varphi)(z)
=
\sum_{i=1}^{n}\left(1-b_i^{\Psi}(z)T_{q,z_i}\right)
\left(f_i^{-}(z)\varphi(z_{\,\widehat{i}})\right),
\end{equation*}
it 
can be regarded as 
{\em symmetrization} of the coboudary operator for the 
$q$-difference 
de Rham cohomology \cite{A1990,A1991}. 
The cokernel 
\begin{equation*}
H_{\mathrm{sym}}^{\Psi}
=\mathrm{Coker}
(\nabla_{\mathrm{sym}}^{\Psi}:\,\cH^{(p)}_{s-1,n-1}\to \cH^{(p)}_{s-1,n})=
\cH^{(p)}_{s-1,n}/\nabla_{\mathrm{sym}}^{\Psi}\cH^{(p)}_{s-1,n-1}
\end{equation*}
plays the role of the 
``symmetrized $n$th 
$q$-difference de Rham cohomology group'' 
associated with $\Psi(z)$. 
In this context the elements of $\cH_{s-1,n}^{(p)}$ will be called 
{\em $q$-cocyles}, or simply {\em cocycles}.
When 
$\psi_1,\psi_2\in\cH^{(p)}_{s-1,n}$, we denote by 
\begin{equation*}
\psi_1(z)\equiv_{\Psi} \psi_2(z)
\quad\mbox{or}\quad 
\Psi(z)\psi_1(z)\equiv \Psi(z)\psi_2(z)
\end{equation*}
the congruence 
{\em modulo} $\nabla_{\mathrm{sym}}^{\Psi}\cH^{(p)}_{s-1,n-1}$, i.e. 
$\psi_1(z)-\psi_2(z)\in\nabla_{\mathrm{sym}}^{\Psi}\cH^{(p)}_{s-1,n-1}.$  
If this is the case, 
it turns out that $\la\psi_1\ra_{\Psi}=\la\psi_2\ra_{\Psi}$, 
namely, 
\begin{equation*}
\int_{\bT^n}\Psi(z)\psi_1(z)\omega_n(z)=
\int_{\bT^n}\Psi(z)\psi_2(z)\omega_n(z)
\end{equation*}
by the Cauchy theorem, 
provided that $|a_k|<1$ ($k=1,\ldots,,m$) and $|t|<1$. 
In fact, we have 
\begin{lem}
\label{lem:nabla=0}
Suppose that $|a_k|<1$ $(k=1,\ldots,,m)$ and $|t|<1$. 
If $\psi(z)\equiv_{\Psi} 0$ for $\psi\in \cH^{(p)}_{s-1,n}$, then we have  
$\la\psi\ra_{\Psi}=0$, 
namely, 
\begin{equation*}
\int_{\bT^n}\Psi(z)\psi(z)\omega_n(z)=0.
\end{equation*}
\end{lem}
\begin{proof}{Proof}  
For $\psi\in \cH^{(p)}_{s-1,n}$, 
$\psi(z)\equiv_{\Psi} 0$ 
is equivalent to 
\begin{equation*}
\exists \varphi\in\cH^{(p)}_{s-1,n-1}:\quad
\Psi(z)\psi(z)=
\sum_{i=1}^{n}\left(1-T_{q,z_i}\right)
\left(\Psi(z)f_i^{-}(z)\varphi(z_{\,\widehat{i}})\right). 
\end{equation*}
If $|a_k|<1$ $(k=1,\ldots,,m)$ and $|t|<1$, 
one can verify that 
$\Psi(z)f_i^{-}(z)\varphi(z_{\,\widehat{i}})$ is holomorphic in a neighborhood of the compact set 
\begin{equation}
\label{eq:Tqi}
|q|\le |z_i| \le 1,\quad |z_j|=1\quad(1\le j\le n;\ j\ne i).
\end{equation}
In fact, since $f_i^{-}(z)$ is explicitly written as 
\begin{equation*}
f_i^{-}(z)=f_i^{+}(z^{-1})
=\frac{\prod_{k=1}^{m}\theta(a_kz_i^{-1};p)
}{z_i^{-1}\theta(z_i^{-2};p)}\prod_{\substack{1\le j\le n\\[1pt] j\ne i}}
\frac{\theta(tz_i^{-1}z_j^{\pm1};p)}{\theta(z_i^{-1}z_j^{\pm1};p)}, 
\end{equation*}
in the product $\Psi(z)f_i^{-}(z)$, 
all possible poles of each of the two functions $\Psi(z)$ and  
$f_i^{-}(z)$ 
\begin{equation*}
\begin{split}
&p^{\mu} z_i^{-2}=1, \quad p^{\mu}z_iz_j^{\pm1}, 
\quad
p^{\mu}a_k z_i^{-1},\quad
p^{\mu} tz_i^{-1}z_j^{\pm1}=1
\\
&\quad(1\le j\le n,\ j\ne i;\ k=1,\ldots,m; \ \mu
\in\mathbb{N}
)
\end{split}
\end{equation*}
relevant to 
the
region \eqref{eq:Tqi}
are eliminated 
by zeros of the other.  
Hence, by the Cauchy theorem, we have
\begin{equation*}
\begin{split}
&\int_{\bT^n}T_{q,z_i}
\left(\Psi(z)f_i^{-}(z)\varphi(z_{\,\widehat{i}})\right)
\omega_n(z)\\
&=
\int_{\bT^{n-1}}\bigg(
\frac{1}{2\pi\sqrt{-1}}
\int_{|z_i|=1}
T_{q,z_i}
\left(\Psi(z)f_i^{-}(z)\right)
\frac{dz_i}{z_i}
\bigg)\varphi(z_{\,\widehat{i}})\omega_{n-1}(z)
\\
&=
\int_{\bT^{n-1}}\bigg(
\frac{1}{2\pi\sqrt{-1}}
\int_{|z_i|=|q|}
\Psi(z)f_i^{-}(z)
\frac{dz_i}{z_i}
\bigg)
\varphi(z_{\,\widehat{i}})
\omega_{n-1}(z)
\\
&=
\int_{\bT^{n-1}}\bigg(
\frac{1}{2\pi\sqrt{-1}}
\int_{|z_i|=1}
\Psi(z)f_i^{-}(z)
\frac{dz_i}{z_i}
\bigg)
\varphi(z_{\,\widehat{i}})
\omega_{n-1}(z)
\\
&=\int_{\bT^n}
\Psi(z)f_i^{-}(z)\varphi(z_{\,\widehat{i}})
\omega_n(z)
,
\end{split}
\end{equation*}
so that 
\begin{equation*}
\int_{\bT^n}\left(1-T_{q,z_i}\right)
\left(\Psi(z)f_i^{-}(z)\varphi(z_{\,\widehat{i}})\right)\omega_n(z)=0\quad (i=1,\ldots,n).
\end{equation*}
This implies that 
\begin{equation*}
\int_{\bT^n}\Psi(z)\psi(z)\omega_n(z)=
\sum_{i=1}^{n}\int_{\bT^n}\left(1-T_{q,z_i}\right)
\left(\Psi(z)f_i^{-}(z)\varphi(z_{\,\widehat{i}})\right)\omega_n(z)=0.
\end{equation*}
This completes the proof.
\end{proof}
\subsection{Reduction of cocycles}
Choosing $s$ parameters $a_1,\ldots,a_s$ from 
$a=(a_1,\ldots,a_m)$, 
we consider the interpolation basis 
\begin{equation*}
E_{\mu}(a_{\pr{1,\ldots,s}};z;p)
=
E_{\mu}(a_1,\ldots,a_s;z;p)\qquad(\mu\in Z_{s,n})
\end{equation*}
for $\cH_{s-1,n}^{(p)}$ 
with respect to the parameters 
$a_{\pr{1,\ldots,s}}=(a_1,\ldots,a_s)$. 
In this subsection we use the notations 
$\theta(z)=\theta(z;p)$ and  
$E_{\mu}(a_{\pr{1,\ldots,s}};z)=E_{\mu}(a_{\pr{1,\ldots,s}};z;p)$, 
omitting the base $p$.  
\par\medskip
For each 
$\lambda\in Z_{s,n-1}$,  
we take the interpolation function 
$\varphi_{\lambda}=
E_{\lambda}(a_{\pr{1,\ldots,s}};\cdot)\in\cH_{s-1,n-1}^{(p)}$
and set 
$\psi_{\lambda}=
\nabla_{\mathrm{sym}}^{\Psi}\varphi_{\lambda}\in\cH^{(p)}_{s-1,n}$: 
\begin{equation}\label{eq:psilambda}
\quad
\psi_{\lambda}(z)=\sum_{i=1}^{n}(f_i^{+}(z)+f_i^{-}(z))
E_{\lambda}(a_{\pr{1,\ldots,s}};z_{\,\widehat{i}})\in\cH^{(p)}_{s-1,n}. 
\end{equation}
\begin{lem}\label{lem:psiexp}
For each 
$\lambda\in Z_{s,n-1}$, $\psi_{\lambda}(z)$ defined by 
\eqref{eq:psilambda} is expressed 
in terms of the interpolation basis of $\cH_{s-1,n}^{(p)}$
as 
\begin{equation*}
\psi_{\lambda}(z)=\sum_{k=1}^{s}C_{\lambda,k}
E_{\lambda+\ep_k}(a_{\pr{1,\ldots,s}};z),
\end{equation*}
where 
\begin{equation*}
\begin{split}
C_{\lambda,k}
&=
\frac{\prod_{l=1}^{m}\theta(t^{\lambda_k}a_ka_l)}
{(t^{\lambda_k}a_k)^{s}\theta(t^{2\lambda_k}a_k^2)}
\prod_{l=1}^{s}
\frac{\theta(t^{\lambda_k+\lambda_l}a_ka_l)
\theta(t^{\lambda_k+1}a_k/a_l)}
{\theta(t^{\lambda_k}a_ka_l)
\theta(t^{\lambda_k-\lambda_l+1}a_k/a_l)}\,
\\
&=\frac{1}
{(t^{\lambda_k}a_k)^{s}}
\prod_{l=1}^{s}
\frac{\theta(t^{\lambda_k+1}a_k/a_l)}
{\theta(t^{\lambda_k-\lambda_l+1}a_k/a_l)}
\prod_{\substack{1\le l\le s\\[1pt] l\ne k}}
\theta(t^{\lambda_k+\lambda_l}a_ka_l)
\prod_{l=s+1}^{m}
\theta(t^{\lambda_k}a_ka_l). 
\end{split}
\end{equation*}
\end{lem}
\begin{proof}{Proof}
Noting that 
\begin{equation*}
\psi_{\lambda}(z)=\sum_{\mu\in Z_{s,n}}
\psi_{\lambda}((a_{\pr{1,\ldots,s}})_{t,\mu})E_{\mu}(a_{\pr{1,\ldots,s}};z), 
\end{equation*}
we determine the values 
of $\psi_{\lambda}(z)$ at $z=(a_{\pr{1,\ldots,s}})_{t,\mu}$ 
by means of \eqref{eq:psilambda}. 
\par
Fixing $\mu\in Z_{s,n}$, suppose that 
$z_i=t^{\nu}a_k$ ($k=1,\ldots,s$; $0\le \nu<\mu_k$) 
under the substitution $z=(a_{\pr{1,\ldots,s}})_{t,\mu}$. 
If $\nu>0$, we have 
$f_i^{-}((a_{\pr{1,\ldots,s}})_{t,\lambda})=0$ 
since $i\ge 2$ and 
$f_i^{-}(z)$ has the factor $\theta(tz_{i-1}/z_i)$. 
If $\nu=0$, then $z_i=a_k$, and 
$f_i^{-}((a_{\pr{1,\ldots,s}})_{t,\mu})=0$ since 
$f_i^{-}(z)$ has the factor $\theta(a_k/z_i)$. 
This means that $f^{-}_i((a_{\pr{1,\ldots,s}})_{t,\mu})=0$ for any $\mu$. 
\par
As to 
$f_i^{+}(z)$, 
if $\nu<\mu_k-1$, then $f_i^{+}((a_{\pr{1,\ldots,s}})_{t,\mu})=0$
since $i<n$ and $f_i^{+}(z)$ have the factor 
$\theta(tz_i/z_{i+1})$.  
Assume that $\nu=\mu_{k}-1$. 
Since 
$z_{\,\widehat{i}}=(a_{\pr{1,\ldots,s}})_{t,\mu-\ep_k}$ in this case, we have 
\begin{equation*}
\begin{split}
f_{i}^{+}(z)E_{\lambda}(a_{\pr{1,\ldots,s}};z)
\Big|_{z=(a_{\pr{1,\ldots,s}})_{t,\mu}}
&=f_{i}^{+}((a_{\pr{1,\ldots,s}})_{t,\mu})
E_{\lambda}(a_{\pr{1,\ldots,s}};(a_{\pr{1,\ldots,s}})_{t,\mu-\ep_k})
\\ 
&=\delta_{\lambda,\mu-\ep_k}f_{i}^{+}((a_{\pr{1,\ldots,s}})_{t,\mu}).
\end{split}
\end{equation*}
Hence, 
$\psi_{\lambda}((a_{\pr{1,\ldots,s}}))_{t,\mu})$ is nontrivial only when 
$\mu=\lambda+\ep_k$ for some $k=1,\ldots,m$, 
and 
\begin{equation*}
\begin{split}
\psi_{\lambda}((a_{\pr{1,\ldots,s}})_{t,\mu})&=f_{i}^{+}((a_{\pr{1,\ldots,s}})_{t,\lambda+\ep_k})
\\
&=
\frac{\prod_{l=1}^{m}\theta(t^{\lambda_k}a_ka_l)}
{(t^{\lambda_k}a_k)^{s}\theta(t^{2\lambda_k}a_k^2)}
\prod_{l=1}^{s}
\prod_{\nu=0}^{\lambda_l-1}
\frac{\theta(t^{\lambda_k+1+\nu}a_ka_l)
\theta(t^{\lambda_k+1-\nu}a_k/a_l)}
{\theta(t^{\lambda_k+\nu}a_ka_l)
\theta(t^{\lambda_k-\nu}a_k/a_l)}
\\
&=
\frac{\prod_{l=1}^{m}\theta(t^{\lambda_k}a_ka_l)}
{(t^{\lambda_k}a_k)^{s}\theta(t^{2\lambda_k}a_k^2)}
\prod_{l=1}^{s}
\frac{\theta(t^{\lambda_k+\lambda_l}a_ka_l)
\theta(t^{\lambda_k+1}a_k/a_l)}
{\theta(t^{\lambda_k}a_ka_l)
\theta(t^{\lambda_k-\lambda_l+1}a_k/a_l)}, 
\end{split}
\end{equation*}
which gives the explicit expression of $C_{\lambda,k}$. 
\end{proof}
\begin{rem}\rm 
In the expansion of 
$\psi_{\lambda}(z)$ $(\lambda\in Z_{s,n-1})$
in terms of the interpolation basis of 
$\cH^{(p)}_{s,n}$, 
its leading term is given by 
$E_{\lambda+\ep_1}(a_{\pr{1,\ldots,s}};z)$
with respect to the lexicographic ordering of $Z_{s,n}$. 
Hence, for generic values of $a_1,\ldots,a_s$, 
the functions 
$\psi_{\lambda}(z)$ $(\lambda\in Z_{s,n-1})$ 
are linearly independent over $\bC$.  
This means that the $\bC$-linear mapping 
$\nabla_{\mathrm{sym}}^{\Psi}:\ \cH^{(p)}_{s-1,n-1}\to \cH^{(p)}_{s-1,n}$ 
is {\em injective}.  Hence
\begin{equation*}\ts
\dim_{\bC}H_{\mathrm{sym}}^{\Psi}
=\dim_{\bC}\cH^{(p)}_{s-1,n}/\nabla_{\mathrm{sym}}^{\Psi}\cH^{(p)}_{s-1,n-1}
=\binom{n+s-1}{s-1}-
\binom{n+s-2}{s-1}=\binom{n+s-2}{s-2}.  
\end{equation*}
In particular, we have 
$\dim_{\bC}H_{\mathrm{sym}}^{\Psi}=1$ if $s=2$ ($r=1$), and 
$\dim_{\bC}H_{\mathrm{sym}}^{\Psi}=n+1$ if $s=3$ ($r=2$).
\hfill $\square$
\end{rem}
\par
From Lemma \ref{lem:psiexp}, 
for each $\lambda\in Z_{s,n-1}$ 
we have the congruence 
\begin{equation*}
\sum_{k=1}^{s}
\frac{1}
{(t^{\lambda_k}a_k)^{s}}
\prod_{l=1}^{s}
\frac{\theta(t^{\lambda_k+1}a_k/a_l)}
{\theta(t^{\lambda_k-\lambda_l+1}a_k/a_l)}
\prod_{\substack{1\le l\le s\\[1pt] l\ne k}}
\theta(t^{\lambda_k+\lambda_l}a_ka_l)
\prod_{l=s+1}^{m}
\theta(t^{\lambda_k}a_ka_l)
E_{\lambda+\ep_k}(a_{\pr{1,\ldots,s}};z)
\equiv_{\Psi}0
\end{equation*}
modulo $\nabla_{\mathrm{sym}}^\Psi\cH_{s-1,n-1}^{\Psi}$, 
or equivalently, 
\begin{equation}
\label{eq:E-cong}
\begin{split}
\sum_{k=1}^{s}
\prod_{l=1}^{s}
\frac{[t^{\lambda_k+1}a_k/a_l]}
{[t^{\lambda_k-\lambda_l+1}a_k/a_l]}
\prod_{\substack{1\le l\le s\\[1pt] l\ne k}}
[t^{\lambda_k+\lambda_l}a_ka_l]
\prod_{l=s+1}^{m}
[t^{\lambda_k}a_ka_l]
E_{\lambda+\ep_k}(a_{\pr{1,\ldots,s}};z)
\equiv_{\Psi}0
\end{split}
\end{equation}
in terms of the odd theta function
$[u]=u^{-\hf}\theta(u)$. 
For each $\mu\in\bN^{s}$ we set 
\begin{equation*}
K_{\mu}=
\prod_{i,j=1}^{s}
[ta_i/a_j]_{\mu_i}
\prod_{i=1}^{s}
\prod_{j=s+1}^{m}
[a_ia_j]_{\mu_i}
\prod_{1\le i<j\le s}
[a_ia_j]_{\mu_i+\mu_j},
\end{equation*}
so that 
\begin{equation*}
K_{\lambda+\ep_k}
=K_{\lambda}
\prod_{j=1}^{s}[t^{\lambda_k+1}a_k/a_j]
\prod_{\substack{1\le l\le s\\[1pt] l\ne k}}
[t^{\lambda_k+\lambda_j}a_ka_j]
\prod_{j=s+1}^{m}
[t^{\lambda_k}a_ka_j]
\quad(1\le k\le s).
\end{equation*}
Then, formula \eqref{eq:E-cong} implies that 
the renormalized interpolation functions 
\begin{equation*}
\widetilde{E}_{\mu}(a_{\pr{1,\ldots,s}};z)
=K_{\mu}E_{\mu}(a_{\pr{1,\ldots,s}};z)
\quad (\mu\in Z_{s,n})
\end{equation*}
satisfy 
\begin{equation*}
\sum_{k=1}^{s}
\prod_{l=1}^{s}
\frac{1}
{[t^{\lambda_k-\lambda_l+1}a_k/a_l]}
\widetilde{E}_{\lambda+\ep_k}(a_{\pr{1,\ldots,s}};z)
\equiv_{\Psi}0.
\end{equation*}
Therefore we obtain 
\begin{equation*}
\widetilde{E}_{\lambda+\ep_{s}}(a_{\pr{1,\ldots,s}};z)
\equiv_{\Psi}
\sum_{k=1}^{s-1}
\frac{[t^{\lambda_s-\lambda_{k}+1}a_s/a_{k}]}
{[t^{\lambda_s-\lambda_{k}-1}a_s/a_{k}]}
\prod_{\substack{1\le l\le s-1\\[1pt] l\ne k}}
\frac{[t^{\lambda_{s}-\lambda_l+1}a_{s}/a_l]}
{[t^{\lambda_k-\lambda_l+1}a_k/a_l]}
\widetilde{E}_{\lambda+\ep_k}(a_{\pr{1,\ldots,s}};z), 
\end{equation*}
and by putting $\lambda+\ep_{s}=\mu$, 
\begin{equation}\label{eq:Etilde}
\begin{split}
\widetilde{E}_{\mu}(a_{\pr{1,\ldots,s}};z)
&\equiv_{\Psi}
\sum_{k=1}^{s-1}
\frac{[t^{\mu_s-\mu_k}a_s/a_{k}]}
{[t^{\mu_s-\mu_k-2}a_s/a_{k}]}
\prod_{\substack{1\le l\le s-1\\[1pt] l\ne k}}
\frac{[t^{\mu_s-\mu_l}a_s/a_l]}
{[t^{\mu_k-\mu_l+1}a_k/a_l]}
\widetilde{E}_{\mu+\ep_k-\ep_s}(a_{\pr{1,\ldots,s}};z). 
\end{split}
\end{equation}
Using this formula, we can lower the last component of $\mu$ step by step.
\begin{thm}
\label{thm:E cong RE}
Suppose that $t^{2n-2}a_1\cdots a_m=1$ $(m=2s+2)$.  
For each $\mu\in Z_{s,n}$ and for each integer $l$ satisfying $0\le l\le\mu_{s}$, 
we have 
\begin{equation}\label{eq:Rmunu01}
\begin{split}
&\widetilde{E}_{\mu}(a_{\pr{1,\ldots,s}};z)
\equiv_{\Psi}
\sum_{\substack{\nu\in Z_{s,n}\\ 
\nu_{s}=\mu_{s}-l,\,\mu'\le\nu'}}
R_{\mu,\nu} \widetilde{E}_{\nu}(a_{\pr{1,\ldots,s}};z),
\end{split}
\end{equation}
where
\begin{equation}\label{eq:Rmunu02}
\begin{split}
&R_{\mu,\nu}=
\prod_{j=1}^{s-1}
\frac{
[t^{\mu_s-\mu_j}a_s/a_j]
}{
[t^{\nu_s-\nu_j}a_s/a_j]
}
\prod_{j=1}^{s}
\frac{[t^{\nu_s-\mu_j}a_s/a_j]_{\mu_s-\nu_s}}
{\prod_{i=1}^{s-1}[t^{\mu_i-\mu_j+1}a_i/a_j]_{\nu_i-\mu_i}}. 
\end{split}
\end{equation}
\end{thm}
\begin{rem}
The coefficient $R_{\mu,\nu}$ is also written as 
\begin{equation*}
R_{\mu,\nu}
=
\frac{[t]_{\mu_s-\nu_s}}
{
\prod_{i,j=1}^{s-1}
[t^{\mu_i-\mu_j+1}a_i/a_j]_{\nu_i-\mu_i}
}
\prod_{i=1}^{s-1}
\frac{[t^{\mu_s-\mu_i}a_s/a_i]}
{[t^{\nu_s-\nu_i}a_s/a_i]}
\frac{[t^{\nu_s-\mu_i}a_s/a_i]_{\mu_s-\nu_s}}
{[t^{\mu_s-\nu_i}a_s/a_i]_{\nu_i-\mu_i}}.
\end{equation*}
\end{rem}
\begin{proof}{Proof} We use induction on $l$ to prove that the coefficients $R_{\mu,\nu}$ defined by \eqref{eq:Rmunu02} 
satisfy \eqref{eq:Rmunu01}.
When $l=1$, 
\begin{equation*}
R_{\mu,\mu+\ep_k-\ep_s}
=
\frac{
[t^{\mu_s-\mu_k}a_s/a_k]
}{
[t^{\mu_s-\mu_k-2}a_s/a_k]
}
\prod_{\substack{1\le j\le s-1\\[1pt] j\ne k}}
\frac{[t^{\mu_s-\mu_j}a_s/a_j]}
{[t^{\mu_k-\mu_j+1}a_k/a_j]}
\qquad (1\le k\le s-1).
\end{equation*}
recover the coefficients in \eqref{eq:Etilde}.
Under  \eqref{eq:Rmunu01} for $l-1$ as the assumption of induction, we have 
\begin{equation*}
\begin{split}
\widetilde{E}_{\mu}(a_{\pr{1,\ldots,s}};z)
&\equiv_{\Psi}
\sum_{\substack{\lambda\in Z_{s,n}\\ \lambda_s=\mu_s-l+1}}
R_{\mu,\lambda}\widetilde{E}_{\lambda}(a_{\pr{1,\ldots,s}};z)
\\
&\equiv_{\Psi}
\sum_{\substack{\lambda\in Z_{s,n}\\ \lambda_s=\mu_s-l+1}}
\ \sum_{\substack{\nu\in Z_{s,n}\\  \nu_s=\mu_s-l}}
R_{\mu,\lambda}
R_{\lambda,\nu}
\widetilde{E}_{\nu}(a_{\pr{1,\ldots,s}};z).
\end{split}
\end{equation*}
Hence it suffices to show 
\begin{equation}
\label{eq:RR=R01}
\begin{split}
R_{\mu,\nu}=
\sum_{\substack{\lambda\in Z_{s,n}\\
\lambda_s=\mu_s-l+1,\,
\mu'\le\lambda'\le\nu'}}
R_{\mu,\lambda}R_{\lambda,\nu}.
\end{split}
\end{equation}
for each $\nu\in Z_{s,n}$ with $\nu_s=\mu_s-l$, $\mu'\le\nu'$. 
Under the condition
$\mu'\le\lambda'\le\nu'$, by \eqref{eq:Rmunu02} we have 
\begin{equation}
\label{eq:RR=R02}
\begin{split}
R_{\mu,\lambda}R_{\lambda,\nu}
&=
R_{\mu,\nu}
\prod_{j=1}^{s}
\frac{
[t^{\lambda_s-\mu_j}a_s/a_j]_{\mu_s-\lambda_s}
[t^{\nu_s-\lambda_j}a_s/a_j]_{\lambda_s-\nu_s}
}{[t^{\nu_s-\mu_j}a_s/a_j]_{\mu_s-\nu_s}}
\\
&\quad\cdot
\prod_{j=1}^{s}
\prod_{i=1}^{s-1}
\frac{[t^{\mu_i-\mu_j+1}a_i/a_j]_{\nu_i-\mu_i}}
{[t^{\mu_i-\mu_j+1}a_i/a_j]_{\lambda_i-\mu_i}
[t^{\lambda_i-\lambda_j+1}a_i/a_j]_{\nu_i-\lambda_i}}
\\
&=
R_{\mu,\nu}
\prod_{j=1}^{s}
\frac{
[t^{\nu_s-\lambda_j}a_s/a_j]_{\lambda_s-\nu_s}
}{
[t^{\nu_s-\mu_j}a_s/a_j]_{\lambda_s-\nu_s}}
\prod_{j=1}^{s}
\prod_{i=1}^{s-1}
\frac{[t^{\lambda_i-\mu_j+1}a_i/a_j]_{\nu_i-\lambda_i}}
{[t^{\lambda_i-\lambda_j+1}a_i/a_j]_{\nu_i-\lambda_i}}. 
\end{split}
\end{equation}
In the right-hand side of \eqref{eq:RR=R01}, each $\lambda$ is expressed as 
$\lambda=\nu-\ep_k+\ep_s$ for some $k\in\pr{1,\ldots,s-1}$ such that $\mu_k<\nu_k$.
Then 
\eqref{eq:RR=R02} is rewritten as 
\begin{equation}
\label{eq:RR=R03}
\begin{split}
R_{\mu,\lambda}R_{\lambda,\nu}&=
R_{\mu,\nu}
\prod_{j=1}^{s}
\frac{
[t^{\nu_s-\lambda_j}a_s/a_j]
}{
[t^{\nu_s-\mu_j}a_s/a_j]}
\frac{
[t^{\lambda_k-\mu_j+1}a_k/a_j]}
{[t^{\lambda_k-\lambda_j+1}a_k/a_j]}
\\
&=
R_{\mu,\nu}
\frac{[t^{\nu_k-\mu_k}]}
{[t^{\mu_s-\nu_s}]}
\frac{[t^{\mu_s-\nu_k}a_s/a_k]}
{[t^{\nu_s-\nu_k}a_s/a_k]}
\prod_{\substack{1\le j\le s-1\\[1pt] j\ne k}}
\frac{[t^{\nu_k-\mu_j}a_k/a_j]}
{[t^{\nu_k-\nu_j}a_k/a_j]}
\prod_{j=1}^{s-1}
\frac{[t^{\nu_s-\nu_j}a_s/a_j]}
{[t^{\nu_s-\mu_j}a_s/a_j]}. 
\end{split}
\end{equation}
Note that the partial fraction decomposition
\begin{equation*}
\prod_{j=1}^{s-1}
\frac{[z/t^{\mu_j}a_j]}
{[z/t^{\nu_j}a_j]}
=\sum_{k=1}^{s-1}
\frac{[t^{\mu_s-\nu_s}z/t^{\nu_k}a_k]
[t^{\nu_k-\mu_k}]}
{[z/t^{\nu_k}a_k][t^{\mu_s-\nu_s}]}
\prod_{\substack{1\le j\le s-1\\[1pt] j\ne k}}
\frac{[t^{\nu_k}a_k
/t^{\mu_j}a_j]}
{[t^{\nu_k}a_k/t^{\nu_j}a_j]}
\end{equation*}
implies 
\begin{equation}
\label{eq:pfd}
\begin{split}
\prod_{j=1}^{s-1}
\frac{[t^{\nu_s-\mu_j}a_s/a_j]}
{[t^{\nu_s-\nu_j}a_s/a_j]}
=
\sum_{k=1}^{s-1}
\frac{[t^{\mu_s-\nu_k}a_s/a_k][t^{\nu_k-\mu_k}]}
{[t^{\nu_s-\nu_k}a_s/a_k][t^{\mu_s-\nu_s}]}
\prod_{\substack{1\le j\le s-1\\[1pt] j\ne k}}
\frac{[t^{\nu_k-\mu_j}a_k/a_j]}
{[t^{\nu_k-\nu_j}a_k/a_j]} 
\end{split}
\end{equation}
as the special case $z=t^{\nu_s}a_s$. 
By \eqref{eq:RR=R03} and  \eqref{eq:pfd} the right-hand side of \eqref{eq:RR=R01} is computed as  
\begin{equation*}
\begin{split}
\sum_{\substack{\lambda\in Z_{s,n}\\ \lambda_s=\nu_s+1}}
R_{\mu,\lambda}R_{\lambda,\nu}
&=
R_{\mu,\nu}
\sum_{k=1}^{s-1}
\frac{[t^{\mu_s-\nu_k}a_s/a_k][t^{\nu_k-\mu_k}]}
{[t^{\nu_s-\nu_k}a_s/a_k][t^{\mu_s-\nu_s}]}
\prod_{\substack{1\le j\le s-1\\[1pt] j\ne k}}
\frac{[t^{\nu_k-\mu_j}a_k/a_j]}
{[t^{\nu_k-\nu_j}a_k/a_j]}
\prod_{j=1}^{s-1}
\frac{[t^{\nu_s-\nu_j}a_s/a_j]}
{[t^{\nu_s-\mu_j}a_s/a_j]}
\\
&=
R_{\mu,\nu},
\end{split}
\end{equation*}
which completes the proof of Theorem \ref{thm:E cong RE}. 
\end{proof}
\par\medskip
Theorem \ref{thm:E cong RE} implies 
that each 
$E_{\mu}(a_{\pr{1,\ldots,s}};z)$ 
can be reduced to a linear combination of 
interpolation functions on the $s$th face 
under the congruence in 
$H_{\mathrm{sym}}^{\Psi}
=\cH^{(p)}_{s-1,n}/
\nabla_{\mathrm{sym}}^{\Psi}\cH^{(p)}_{s-1,n-1}$. 
Namely, for each $\mu\in Z_{s,n}$ we have 
\begin{equation}\label{eq:Smunu}
\begin{split}
E_{\mu}(a_{\pr{1,\ldots,s}};z)\equiv_{\Psi}
\sum_{\substack{\nu\in Z_{s,n}\\ \nu_s=0,\,\mu'\le\nu'}}
S_{\mu,\nu}E_{\nu}(a_{\pr{1,\ldots,s}};z), 
\end{split}
\end{equation}
where the coefficients are given by 
\begin{equation}\label{eq:Smunu2}
\begin{split}
S_{\mu,\nu}
&=K_{\mu}^{-1}R_{\mu,\nu}K_{\nu}
\\
&=
\frac{[t]_{\mu_s-\nu_s}}
{
\prod_{i,j=1}^{s-1}
[t^{\mu_i-\mu_j+1}a_i/a_j]_{\nu_i-\mu_i}
}
\prod_{i=1}^{s-1}
\frac{[t^{\mu_s-\mu_i}a_s/a_i]}
{[t^{\nu_s-\nu_i}a_s/a_i]}
\frac{[t^{\nu_s-\mu_i}a_s/a_i]_{\mu_s-\nu_s}}
{[t^{\mu_s-\nu_i}a_s/a_i]_{\nu_i-\mu_i}}
\\
&\quad\cdot
\prod_{i,j=1}^{s}
\frac{[ta_i/a_j]_{\nu_i}}{[ta_i/a_j]_{\mu_i}}
\prod_{i=1}^{s}
\prod_{j=s+1}^{m}
\frac{[a_ia_j]_{\nu_i}}{[a_ia_j]_{\mu_i}}
\prod_{1\le i<j\le s}
\frac{[a_ia_j]_{\nu_i+\nu_j}}
{[a_ia_j]_{\mu_i+\mu_j}}. 
\end{split}
\end{equation}
Recall that 
\begin{equation*}
E_{(\nu',0)}(a_{\pr{1,\ldots,s}};z)
=E_{\nu'}(a_{\pr{1,\ldots,s-1}};z)
\frac{\prod_{i=1}^{n} e(z_i,a_s)}
{\prod_{k=1}^{s-1}e(a_k,a_s)_{\nu_k}}
\end{equation*}
by \eqref{eq:Emusthface},
and that 
$E_{\nu'}(a_{\pr{1,\ldots,s-1}};z)$ 
$(\nu'\in Z_{s-1,n})$ form a basis of 
$\cH^{(p)}_{s-2,n}$.  
The above argument implies that the composition 
\begin{equation*}
\cH^{(p)}_{s-2,n}\prod_{i=1}^{n}e(z_i,a_s)
\hookrightarrow
\cH^{(p)}_{s-1,n}
\twoheadrightarrow
H_{\mathrm{sym}}^{\Psi}=
\cH^{(p)}_{s-1,n}/\nabla^{\Psi}\cH^{(p)}_{s-1,n-1}
\end{equation*}
is surjective.  Since 
$\dim_{\bC}\cH^{(p)}_{s-2,n}
=\dim_{\bC}H_{\mathrm{sym}}^{\Psi}=
\binom{n+s-2}{s-2}$, 
we obtain a natural $\bC$-isomorphism 
\begin{equation*}
\cH^{(p)}_{s-2,n}\prod_{i=1}^{n}e(z_i;a_s)\isom
H_{\mathrm{sym}}^{\Psi}
=\cH^{(p)}_{s-1,n}/
\nabla_{\mathrm{sym}}^{\Psi}\cH^{(p)}_{s-1,n-1}. 
\end{equation*}
Hence we have 
\begin{thm}\label{thm:modnabla}
Under the assumption of Theorem 
\ref{thm:E cong RE}, the classes
\begin{equation*}
E_{\nu'}(a_{\pr{1,\ldots,s-1}};z)\prod_{i=1}^{n} e(z_i,a_s)
\quad(\nu'\in Z_{s-1,n})
\end{equation*}
modulo 
$\nabla_{\mathrm{sym}}^{\Psi}\cH^{(p)}_{s-1,n-1}$ 
form a $\bC$-basis of 
$H_{\mathrm{sym}}^{\Psi}=
\cH^{(p)}_{s-1,n}/
\nabla_{\mathrm{sym}}^{\Psi}\cH^{(p)}_{s-1,n-1}$. 
\hfill $\square$
\end{thm}
\subsection{Base change in $H_{\mathrm{sym}}^{\Psi}$}
In what follows we set $s=r+1$, so that $m=2r+4$ 
and 
$H_{\mathrm{sym}}^{\Psi}=
\cH^{(p)}_{r,n}/\nabla_{\mathrm{sym}}^{\Psi}\cH^{(p)}_{r,n-1}$. 
Let $K$ be a subset 
of the indexing set $\pr{1,\ldots,m}$ with 
$|K|=r+1$, and consider the 
the interpolation basis 
\begin{equation*}
\cH^{(p)}_{r,n}=\bigoplus_{\mu\in Z_{K,n}}
\bC\,
E_{\mu}(a_K;z)
\end{equation*}
with respect to the parameters 
$a_K=(a_k)_{k\in K}\in(\bC^\ast)^{K}$, 
where 
\begin{equation*}
Z_{K,n}=\prm{\mu=(\mu_k)_{k\in K}\in \bN^K}
{\ts |\mu|=\sum_{k\in I}\mu_k=n}.
\end{equation*}
Then, 
for each subset $J\subset K$ with $|J|=r$,  
$K=J\cup\pr{k}$, 
the classes 
\begin{equation}\label{eq:EmuaJ}
E_{\mu}(a_J;z)\prod_{i=1}^{n}e(z_i,e_k)\quad (\mu\in Z_{J,n})
\end{equation}
modulo $\nabla_{\mathrm{sym}}^{\Psi}\cH^{(p)}_{r,n-1}$
form a $\bC$-basis of 
$H_{\mathrm{sym}}^{\Psi}$. 
\par\medskip
We now fix a subset $I$ of $\pr{1,\ldots,m}$ with $|I|=r-1$, 
and choose two indices $k,l\in\pr{1,\ldots,m}\backslash I$.  
In this setting, we consider the transition 
between the 
two bases of the form \eqref{eq:EmuaJ} 
of $H_{\mathrm{sym}}^{\Psi}$ specified by 
$I\cup\pr{k}$ and $I\cup\pr{l}$.  
We define the transition coefficients 
$B_{\mu,\nu}^{I;k,l}$  
through the relation 
\begin{equation}\label{eq:transB}
E_{\mu}(a_{I\cup\pr{k}};z)\prod_{i=1}^{n}e(a_l;z_i)
\equiv_{\Psi}
\sum_{\substack{
\nu\in Z_{I\cup\pr{l},n}\\
\mu_{\widehat{k}}\le\nu_{\,\widehat{l}}
}}
B_{\mu,\nu}^{I;k,l}
E_{\nu}(a_{I\cup\pr{l}};z)\prod_{i=1}^{n}e(a_k;z_i)
\quad
(\mu\in Z_{I\cup\pr{k}}),
\end{equation}
where 
$\mu_{\widehat{k}}=(\mu_i)_{i\in I}$ 
and $\nu_{\widehat{l}}=(\nu_i)_{i\in I}$.  
These coefficients are directly computed 
by 
\eqref{eq:Smunu} and \eqref{eq:Smunu2} 
as follows: 
\begin{equation*}
\begin{split}
B_{\mu,\nu}^{I;k,l}
&=
(-1)^{\mu_k-\nu_l}
[a_k/a_l]_{\mu_k-\nu_l}
\prod_{i\in I}
\frac
{[t^{\mu_i}a_i/t^{\mu_k}a_k]}
{[t^{\nu_i}a_i/t^{\mu_k}a_k]}
\frac
{[a_i/t^{\nu_l}a_l]_{\mu_i}}
{[a_i/t^{\mu_k}a_k]_{\nu_i}}
\frac
{[a_ia_l]_{\nu_i+\nu_l}}
{[a_ia_k]_{\mu_i+\mu_k}}
\prod_{\substack{j\notin I\\[1pt] j\ne k,l}}
\frac{[a_la_j]_{\nu_l}}{[a_ka_j]_{\mu_k}}
\\
&\quad\cdot
\prod_{i,j\in I}
\frac
{[t^{\mu_i+1}a_i/a_j]_{\nu_i-\mu_i}}
{[t^{\mu_i-\mu_j+1}a_i/a_j]_{\nu_i-\mu_i}}
\prod_{\substack{i,j\in I\\[1pt] i<j}}
\frac
{[a_ia_j]_{\nu_i+\nu_j}}
{[a_ia_j]_{\mu_i+\mu_j}}
\prod_{i\in I}
\prod_{\substack{j\notin I\\[1pt] j\ne k,l}}
\frac{[a_ia_j]_{\nu_i}}{[a_ia_j]_{\mu_i}}. 
\end{split}
\end{equation*}
We remark 
that the matrix 
$B^{I;k,l}=\big(B_{\mu,\nu}^{I;k,l}\big)_{\mu,\nu}$ 
is upper triangular with respect to the partial 
ordering of $\mu_{\widehat{k}}, \nu_{\widehat{l}}\in\bN^I$. 
The diagonal components 
(with $\mu_{\widehat{k}}=\nu_{\widehat{l}}$\,,
$\mu_k=\nu_l$) 
are given by 
\begin{equation*}
\begin{split}
B_{\mu,\nu}^{I;k,l}
&=
\prod_{i\in I}
\frac
{[t^{-\mu_k}a_i/a_l]_{\mu_i}}
{[t^{-\mu_k}a_i/a_k]_{\mu_i}}
\frac
{[t^{\mu_k}a_ia_l]_{\mu_i}}
{[t^{\mu_k}a_ia_k]_{\mu_i}}
\prod_{\substack{1\le j\le m\\[1pt] j\ne k,l}}
\frac
{[a_la_j]_{\mu_k}}
{[a_ka_j]_{\mu_k}}
\\
&=
\left(\frac{a_k}{a_l}\right)^{(r+1)\mu_k}
\prod_{i\in I}
\frac
{e(a_i,t^{\mu_k}a_l)_{\mu_i}}
{e(a_i,t^{\mu_k}a_k)_{\mu_i}}
\prod_{\substack{1\le j\le m\\[1pt] j\ne k,l}}
\frac
{\theta(a_la_j)_{\mu_k}}
{\theta(a_ka_j)_{\mu_k}}.
\end{split}
\end{equation*}
Hence the determinant of $B^{I;k,l}$ is computed as 
\begin{equation}\label{eq:Bdet}
\begin{split}
\det B^{I;k,l}
&=
\left(\frac{a_k}{a_l}\right)^{(r+1)\binom{n+r-1}{r}}
\prod_{0\le u+v\le n}
\prod_{i\in I}
\left(
\frac
{e(a_i,t^{u}a_l)_{v}}
{e(a_i,t^{u}a_k)_{v}}
\right)^{\binom{n-u-v+r-3}{r-3}}
\\
&\quad\cdot
\prod_{u=0}^{n-1}\, 
\prod_{\substack{1\le j\le m\\[1pt] j\ne k,l}}
\left(
\frac
{\theta(a_la_j)_{u}}
{\theta(a_ka_j)_{u}}
\right)^{\binom{n-u+r-2}{r-2}}
\\
&=
\left(\frac{a_k}{a_l}\right)^{(r+1)\binom{n+r-1}{r}}
\prod_{0\le u+v<n}
\prod_{i\in I}
\left(
\frac
{e(t^{u}a_l,t^{v}a_i)}
{e(t^{u}a_k,t^{v}a_i)}
\right)^{\binom{n-u-v+r-3}{r-2}}
\\
&\quad\cdot
\prod_{u=0}^{n-1}\,
\prod_{\substack{1\le j\le m\\[1pt] j\ne k,l}}
\left(
\frac
{\theta(t^{u}a_la_j)}
{\theta(t^{u}a_ka_j)}
\right)^{\binom{n-u+r-2}{r-1}}. 
\end{split}
\end{equation}
\section{\boldmath System of $q$-difference equations}
\label{section:4}
\subsection{System of $q$-difference equations 
associated with a basis of $\cH^{(p)}_{r-1,n}$}
In this section, 
in view of the parameter dependence of 
$\Phi(z)=\Phi(z;a)$ 
we investigate $q$-difference equations 
to be satisfied by the integrals
\begin{equation*}
\la f, g\ra_{\Phi}=
\int_{\bTR^n}
f(z)g(z)\Phi(z;a)\,\omega_n(z)\qquad
(f\in\cH_{r-1,n}^{(p)},\ 
g\in\cH_{r-1,n}^{(q)})
\end{equation*}
with respect to $a$.   
Here we assume that $m=2r+4$ and that 
the parameters 
$a=(a_1,\ldots,a_m)\in\bTC^m$ 
satisfy the conditions $|a_k|<1$ ($k=1,\ldots,m$) and $|t|<1$ as in Lemma \ref{lem:nabla=0}. 
Fixing a $\bC$-basis $\prmts{f_{\mu}(z)}{\mu\in Z_{r,n}}$ 
of 
$ \cH_{r-1,n}^{(p)}$ and 
a holomorphic function $g(z)$ in $\cH^{(q)}_{r-1,n}$, 
we consider the integrals 
\begin{equation}\label{eq:defImu}
\cI_{\mu}=
\la f_{\mu},g\ra_{\Phi}=
\int_{\bTR^n}f_{\mu}(z)g(z)\Phi(z)\,\omega_n(z)
\quad(\mu\in Z_{r,n}). 
\end{equation}
As we will see below, 
under the balancing condition $t^{2n-2}a_1\cdots a_m=pq$, 
the column vector $\cI=(\cI_{\mu})_{\mu\in Z_{r,n}}$ 
satisfies a system of $q$-difference equations of the form 
\begin{equation}\label{eq:qDEI}
T_{q,a_k}T_{q,a_l}^{-1}\,\cI=\cA^{k,l}(a)\,\cI\qquad(1\le k<l\le m)
\end{equation}
of rank $\binom{n+r-1}{n}$, 
where the coefficient matrices 
$\cA^{k,l}(a)=(\cA_{\mu\nu}^{k,l}(a))_{\mu,\nu\in Z_{r,n}}$
are determined independently of the choice of $g(z)$.  
In this setting, 
$f_{\mu}(z)$ and $g(z)$ may depend meromorphically 
on $a\in (\mathbb{C}^*)^m$, 
while we assume that $g(z)$ satisfies the condition
\begin{equation*}
T_{q,a_k}T_{q,a_l}^{-1}g(z)=g(z)\qquad(1\le k<l\le m).  
\end{equation*}
We say that a meromorphic function 
$g(z)=g(z;a)$ 
on $(\mathbb{C}^*)^n\times(\mathbb{C}^*)^m$ {\em depends meromorphically on} $a$, 
if there exists a holomorphic function $h(a)$ on $(\mathbb{C}^*)^m$ such that 
$h(a)g(z;a)$ is holomorphic on $(\mathbb{C}^*)^n\times(\mathbb{C}^*)^m$.
\subsection{Derivation of $q$-difference equations}
We explain how one can derive the $q$-difference 
equations \eqref{eq:qDEI} in the case $l=m$. 
Shifting $a_m$ by $pq$, we have 
\begin{equation*}
\begin{split}
T_{pq,a_m}\Phi(z)&=\Phi(z;a_1,\ldots,a_{m-1},pqa_{m})
\\
&=\Phi(z)
\prod_{i=1}^{n}a_m^{-2}\theta(a_mz_i^{\pm};p)\,\theta(a_mz_i^{\pm1};q)
\\
&=\Phi(z)
\prod_{i=1}^{n}e(a_m,z_i;p)e(a_m,z_i;q),
\end{split}
\end{equation*}
and hence
\begin{equation*}
T_{pq,a_m}\big(f_{\mu}(z)g(z)\Phi(z)\big)
=
\widetilde{f}_{\mu}(z)
\prod_{i=1}^{n}e(a_m,z_i;p)\,
\widetilde{g}(z)\prod_{i=1}^{n}e(a_m,z_i;q)\,\Phi(z),
\end{equation*}
where 
$\widetilde{f}_{\mu}(z)=T_{pq,a_m}f_{\mu}(z)$ and 
$\widetilde{g}(z)=T_{pq,a_m}g(z)$.
Setting
\begin{equation}
\label{eq:deffmug}
\varphi_{\mu}(z)=\widetilde{f}_{\mu}(z)\prod_{i=1}^{n}e(a_m,z_i;p)\quad(\mu\in Z_{r,n}),
\quad 
\psi(z)=\widetilde{g}(z)\prod_{i=1}^{n}e(a_m,z_i;q), 
\end{equation}
we rewrite the formula above as 
\begin{equation*}
T_{pq,a_m}\big(f_{\mu}(z)g(z)\Phi(z)\big)=\varphi_{\mu}(z)\psi(z)\Phi(z).  
\end{equation*}
Note that
$\varphi_{\mu}(z)\in\cH_{r,n}^{(p)}$ $(\mu\in Z_{r,n})$ 
and 
$\psi(z)\in\cH_{r,n}^{(q)}$, 
since $f_{\mu}(z)\in\cH_{r-1,n}^{(p)}$ and $g(z)\in\cH_{r-1,n}^{(q)}$. 
Then we have 
\begin{equation*}
\begin{split}
T_{pq,a_m}\cI_{\mu}&=
\int_{\bTR^n}T_{pq,a_m}\big(f_{\mu}(z)g(z)\Phi(z))
\big)\omega_n(z)
\\
&=\int_{\bTR^n}
\varphi_{\mu}(z)\psi(z)\Phi(z)\,\omega_n(z)
\\
&=\la 
\varphi_{\mu},\psi\ra_{\Phi}
=\la \varphi_{\mu}\ra_{\Psi},
\end{split}
\end{equation*}
where $\Psi(z)=\psi(z)\Phi(z)$. 
We first prove that the integrals 
\begin{equation*}
\widetilde{\cI}_{\mu}=T_{pq,a_m}\cI_{\mu}
=\la \varphi_{\mu}\ra_{\Psi}\quad
(\mu\in Z_{r,n})
\end{equation*}
satisfy a system of $q$-difference equations of the form
\begin{equation*}
T_{q,a_k}T_{q,a_m}^{-1}
\widetilde{\cI}_{\mu}
=\sum_{\nu\in Z_{r,n}}\widetilde{\cA}_{\mu,\nu}^{k,m}(a)\widetilde{\cI}_{\nu}
\quad(\mu\in Z_{r,n})
\end{equation*}
for $k=1,\ldots,m-1$, 
under the balancing condition $t^{2n-2}a_1\cdots a_m=1$. 
Then, applying $T_{pq,a_m}^{-1}$ 
we see that 
$\cI_{\mu}$ $(\mu\in Z_{r,n})$ 
satisfy the system of $q$-difference equations 
\begin{equation*}
T_{q,a_k}T_{q,a_m}^{-1}
\cI_{\mu}
=\sum_{\nu\in Z_{r,n}}{\cA}_{\mu,\nu}^{k,m}(a)\cI_{\nu}
\quad(\mu\in Z_{r,n})
\end{equation*}
under the balancing condition $t^{2n-2}a_1\cdots a_m=pq$, 
where $\cA_{\mu,\nu}^{k,m}(a)
=T_{pq,a_m}^{-1}\widetilde{\cA}_{\mu,\nu}^{k,m}(a)$. 
\par
We now assume that $t^{2n-2}a_1\cdots a_m=1$.  Note that
\begin{equation*}
T_{q,a_k}\Phi(z)=a_{k}^n\prod_{i=1}^{n}e(a_k,z_i;p)\,\Phi(z)\quad
(k=1,\ldots,m).
\end{equation*}
Also, by \eqref{eq:deffmug} we have
\begin{equation*}
T_{q,a_k}T_{q,a_m}^{-1}\psi(z)=\psi(z)(qa_m^2)^{-n}
\qquad(k=1,\ldots,m-1). 
\end{equation*}
Hence obtain
\begin{equation*}
\begin{split}
T_{q,a_k}T_{q,a_m}^{-1}
\big(
\varphi_{\mu}(z)
\Psi(z)\big)&=
T_{q,a_k}T_{q,a_m}^{-1}
\big(
\widetilde{f}_{\mu}(z)\prod_{i=1}^{n}e(a_m,z_i;p)
\Psi(z)\big)
\\
&=
(a_ka_m)^{n}
T_{q,a_k}T_{q,a_m}^{-1}(\widetilde{f}_{\mu}(z))\prod_{i=1}^{n}e(a_k,z_i;p)
\Psi(z). 
\end{split}
\end{equation*}
Since $f_{\nu}(z)$ ($\nu\in Z_{r,n}$) form a $\bC$-basis 
of $\cH_{r-1,n}^{(p)}$, 
by Theorem \ref{thm:modnabla} 
the congruence classes of 
\begin{equation*}
\varphi_{\nu}(z)=
\widetilde{f}_{\nu}(z)\prod_{i=1}^{n}e(a_m,z_i;p)\quad(\nu\in Z_{r,n})
\end{equation*}
form a $\bC$-basis of 
$H_{\mathrm{sym}}^{\Psi}
=\cH_{r,n}^{(p)}/\nabla^{\Psi}_{\mathrm{sym}}\cH_{r,n-1}^{(p)}$. 
This implies that 
\begin{equation*}
(a_ka_m)^{n}\,T_{q,a_k}T_{q,a_m}^{-1}(\widetilde{f}_{\mu}(z))
\prod_{i=1}^{n}e(a_k,z_i;p)
\equiv_{\Psi}
\sum_{\nu\in Z_{r,n}} \widetilde{\cA}_{\mu,\nu}^{k,m}(a)\,\varphi_{\nu}(z)
\end{equation*}
for some 
$\widetilde{\cA}_{\mu,\nu}^{k,m}(a)$.  
Hence we have 
\begin{equation*}
\begin{split}
T_{q,a_k}T_{q,a_m}^{-1}\widetilde{\cI}_{\mu}
&=
T_{q,a_k}T_{q,a_m}^{-1}\la \varphi_{\mu}\ra_{\Psi}
=
\la 
T_{q,a_k}T_{q,a_m}^{-1}(\widetilde{f}_{\mu}(z))
\prod_{i=1}^{n}e(a_k,z_i;p)
\ra_{\Psi}
\\
&=
\sum_{\nu\in Z_{r,n}} \widetilde{\cA}_{\mu,\nu}^{k,m}(a) 
\la \varphi_{\nu}\ra_{\Psi}
=
\sum_{\nu\in Z_{r,n}} \widetilde{\cA}_{\mu,\nu}^{k,m}(a)
\widetilde{\cI}_{\nu}.  
\end{split}
\end{equation*}
\subsection{System of $q$-difference equations 
associated with an interpolation basis}
We now consider the case of the interpolation basis 
\begin{equation*}
f_{\mu}(z)=E_{\mu}(a_{\pr{1,\ldots,r}};z;p)\qquad(\mu\in Z_{r,n}), 
\end{equation*}
and investigate the system of $q$-difference equations
\eqref{eq:qDEI} to be satisfied by the integrals 
\begin{equation}\label{eq:cImuE}
\cI_{\mu}=\cI_{\mu}(a)=\int_{\bT^n} E_{\mu}(a_{\pr{1,\ldots,r}};z;p)g(z)\Phi(z)\omega_n(z)
\quad(\mu\in Z_{r.n})
\end{equation}
for a fixed $g(z)\in\cH^{(q)}_{r-1,n}$. 
We assume that $g(z)$ depends meromorphically on $a\in (\mathbb{C}^*)^m$
and satisfies the $q$-shift invariance 
$T_{q,a_k}T_{q,a_l}^{-1}g(z)=g(z)$ ($1\le k<l\le m$). 
\par\medskip
We suppose below that $a_1\cdots a_mt^{2n-2}=pq$, 
and regard $a_m=pq/a_1\cdots a_{m-1}t^{2n-2}$ 
as a function of $(a_1,\ldots,a_{m-1})$.  
Then the integral $\cI_{\mu}(a)$, regarded as 
a function of $(a_1,\ldots,a_{m-1})$, 
is meromorphic on the open subset 
\begin{equation*}
U_0=\big\{\ 
(a_1,\ldots,a_{m-1})\in(\mathbb{C}^\ast)^{m-1}
\big\vert\ 
|a_1|<1,\ldots,|a_{m-1}|<1, |a_1\cdots a_{m-1}|>\frac{|p||q|}{|t|^{2n-2}}  
\big\}
\end{equation*}
of $(\mathbb{C}^\ast)^{m-1}$; we need to assume   
$|p||q|<|t|^{2n-2}$ in order to ensure that $U_0$ 
is not empty.  
If we assume further that $|p|<|t|^{2n-2}$, the integrals 
$T_{q,a_k}T_{q,a_m}^{-1}\cI_{\mu}(a)$ $(k=1,\ldots,m-1)$ as well as $\cI_{\mu}(a)$
are meromorphic on the nonempty open subset 
\begin{equation*}
V_0=\big\{\ 
(a_1,\ldots,a_{m-1})\in(\mathbb{C}^\ast)^{m-1}
\big\vert\ 
|a_1|<1,\ldots,|a_{m-1}|<1, |a_1\cdots a_{m-1}|>\frac{|p|}{|t|^{2n-2}}  
\big\} 
\end{equation*}
of $U_0$.
\begin{thm}\label{thm:detcA}
Suppose that 
$|p|<|t|^{2n-2}$. 
Under the balancing condition $t^{2n-2}a_1\cdots a_m=pq$, 
the integrals $\cI_{\mu}$ of \eqref{eq:cImuE} 
satisfy a system of $q$-difference equations of the form
\begin{equation}
\label{eq:TTI=SAI}
T_{q,a_k}T_{q,a_m}^{-1}\cI_{\mu}=\sum_{\nu\in Z_{r,n}} 
\cA^{k,m}_{\mu,\nu}(a)\cI_{\nu}\qquad(\mu\in Z_{r,n})
\end{equation}
for each $k\in\pr{1,\ldots,m-1}$, on the nonempty open set $V_0\subset U_0$. 
Here the coefficients $\cA^{k,m}_{\mu,\nu}(a)$ are meromorphic 
functions in $a$, and do not depend on the choice of $g(z)$. 
Furthermore, the determinant 
of the coefficient matrices $\cA^{k,m}(a)$ are given as follows\,$:$
For $k\in\pr{1,\ldots,r}$
\begin{equation}\label{eq:detA1}
\begin{split}
&\det \cA^{k,m}(a)\\
&=\prod_{\substack{i,j\ge0\\ \, i+j<n}}
\prod_{\substack{1\le l\le r\\[1pt] l\ne k}}
\left(\frac{e(t^ia_k,t^ja_l;p)}
{e(t^iqa_k,t^ja_l;p)}\right)^{\binom{n-i-j+r-3}{r-2}}
\prod_{i=0}^{n-1}
\prod_{\substack{1\le l\le m\\[1pt] l\ne k}}
\left(
\frac{\theta(t^ia_ka_l;p)}{\theta(t^iq^{-1}a_ma_l;p)}
\right)^{\binom{n-i+r-2}{r-1}},
\end{split}
\end{equation}
and for $k\in\pr{r+1,\ldots, m-1}$
\begin{equation}\label{eq:detA2}
\begin{split}
\det \cA^{k,m}(a)
=
\prod_{i=0}^{n-1}
\prod_{\substack{1\le l\le m\\[1pt] l\ne k}}
\left(
\frac{\theta(t^ia_ka_l;p)}{\theta(t^iq^{-1}a_ma_l;p)}
\right)^{\binom{n-i+r-2}{r-1}}.  
\end{split}
\end{equation}
\end{thm}
\begin{proof}{Proof}
Note that 
$\widetilde{f}_{\mu}(z)=f_{\mu}(z)$ in this case,
and that 
\begin{equation*}
T_{q,a_k}T_{q,a_m}^{-1}(\widetilde{f}_{\mu}(z))
\prod_{i=1}^{n}e(a_k,z_i;p)
=
T_{q,a_k}(E_{\mu}(a_{\pr{1,\ldots,r}};z;p))
\prod_{i=1}^{n}e(a_k,z_i;p)
\end{equation*}
for $k=1,\ldots,m-1$. 
We investigate the two cases 
$k\in\pr{1,\ldots,r}$ and 
$k\in\pr{r+1,\ldots,m-1}$ separately.  
\par
When $k\in\pr{1,\ldots,r}$, 
by the change of parameters $a_k\to a_m$,  
we have 
\begin{equation*}
E_{\mu}(a_{\pr{1,\ldots,r}};z;p)=
E_{\mu}(a_{I_{k}\cup\pr{k}};z;p)
=\sum_{\alpha\in Z_{I_k\cup\pr{m},n}}
C_{\mu,\alpha}^{I_k;k,m}
E_{\alpha}(a_{I_k\cup\pr{m}};z;p),
\end{equation*}
where $I_{k}=\pr{1,\ldots,r}\backslash\pr{k}$ 
and the coefficients $C_{\mu,\alpha}^{I_k;k,m}$ 
are specified by \eqref{eq:transCmunu}. 
Hence, 
\begin{equation*}
\begin{split}
&
T_{q,a_k}T_{q,a_m}^{-1}\big(
E_{\mu}(a_{\pr{1,\ldots,r}};z;p)\prod_{i=1}^{n}e(a_m,z_i;p)\Psi(z)
\big)
\\
&
=
(a_ka_m)^n
\sum_{\alpha\in Z_{I_k\cup\pr{m},n}}
T_{q,a_k}(C_{\mu,\alpha}^{I_k;k,m})
E_{\alpha}(a_{I_k\cup\pr{m}};z;p)
\prod_{i=1}^{n}e(a_k,z_i;p)\Psi(z). 
\end{split}
\end{equation*}
We now apply \eqref{eq:transB} for the reduction 
in $H^{\Psi}_{\mathrm{sym}}$ from the $k$th face 
to the $m$th face: 
\begin{equation*}
\begin{split}
&
T_{q,a_k}T_{q,a_m}^{-1}\big(
E_{\mu}(a_{\pr{1,\ldots,r}};z;p)\prod_{i=1}^{n}e(a_m,z_i;p)\Psi(z)
\big)
\\
&
\equiv
(a_ka_m)^n
\sum_{\alpha\in Z_{I_k\cup\pr{m},n}}
\sum_{\mu\in Z_{r,n}}
T_{q,a_k}(C_{\mu,\alpha}^{I_k;k,m})
B_{\alpha,\mu}^{I_k;m,k}
E_{\nu}(a_{\pr{1,\ldots,r}};z;p)
\prod_{i=1}^{n}e(a_m,z_i;p)
\Psi(z). 
\end{split}
\end{equation*}
Hence we have 
\begin{equation*}
\widetilde{\cA}^{k,m}_{\mu,\nu}(a)=(a_ka_m)^n
\sum_{\alpha\in Z_{I_k\cup\pr{m},n}} 
T_{q,a_k}(C_{\mu,\alpha}^{I_k,k,m})
B_{\alpha,\nu}^{I_k;m,k}. 
\end{equation*}
\par
When $k\in\pr{r+1,\ldots,m-1}$, choosing an 
index $l\in\pr{1,\ldots,r}$ arbitrarily, we apply the 
change of parameters $a_l\to a_m$ in advance, 
and then perform the reduction from the $k$th 
face to the $m$th face: 
\begin{equation*}
\begin{split}
&
T_{q,a_k}T_{q,a_m}^{-1}\big(
E_{\mu}(a_{\pr{1,\ldots,r}};z;p)\prod_{i=1}^{n}e(a_m,z_i;p)\Psi(z)
\big)
\\
&=
(a_ka_m)^n
E_{\mu}(a_{\pr{1,\ldots,r}};z;p)
\prod_{i=1}^{n}e(a_k,z_i;p)\Psi(z)
\\
&=
(a_ka_m)^n
\sum_{\alpha,\beta,\nu}
C_{\mu,\alpha}^{I_l;l,m}
E_{\alpha}(a_{I_l\cup\pr{m}};z;p)
\prod_{i=1}^{n}e(a_k,z_i;p)\Psi(z)
\\
&\equiv
(a_ka_m)^n
\sum_{\alpha,\beta}
C_{\mu,\alpha}^{I_l;l,m}
B_{\alpha,\beta}^{I_l;m,k}
E_{\beta}(a_{I_l\cup\pr{k}};z;p)
\prod_{i=1}^{n}e(a_m,z_i;p)\Psi(z)
\\
&=
(a_ka_m)^n
\sum_{\alpha,\beta}
C_{\mu,\alpha}^{I_l;l,m}
B_{\alpha,\beta}^{I_l;m,k}
C_{\beta,\nu}^{I_l;k,l}
E_{\nu}(a_{\pr{1,\ldots,r}};z;p)
\prod_{i=1}^{n}e(a_m,z_i;p)\Psi(z).
\end{split}
\end{equation*}
Hence we have 
\begin{equation*}
\widetilde{\cA}_{\mu,\nu}^{k,m}(a)
=(a_ka_m)^n
\sum_{\alpha\in Z_{I_l\cup\pr{m},n}}
\sum_{\beta\in Z_{I_l\cup\pr{k},n}}
C_{\mu,\alpha}^{I_l;l,m}
B_{\alpha,\beta}^{I_l;m,k}
C_{\beta,\nu}^{I_l;k,l}. 
\end{equation*}
As we explained before, the coefficient matrices $\cA^{k,m}(a)$ 
are determined from $\widetilde{\cA}^{k,m}(a)$ 
by $\cA^{k,m}(a)=T_{pq,a_m}^{-1}(\widetilde{\cA}^{k,m}(a))$. 
\par
Note that the matrices 
$C^{l;k,l}=\big(C_{\mu,\nu}^{I;k,l}\big)_{\mu,\nu}$ 
and 
$B^{I;k,l}=\big(B_{\mu,\nu}^{I;k,l}\big)_{\mu,\nu}$ 
are 
lower triangular and upper triangular respectively,  
with respect to the partial ordering of $\bN^I$. 
Hence, 
the determinant of 
$\widetilde{\cA}^{k,m}(a)=
(\widetilde{\cA}_{\mu,\nu}^{k,m}(a))_{\mu,\nu\in Z_{r,n}}$
is computed by \eqref{eq:Cdet} and \eqref{eq:Bdet} 
as follows: 
For 
$k\in\pr{1,\ldots,r}$, 
\begin{equation}\label{eq:detwtA1}
\begin{split}
\det \widetilde{\cA}^{k,m}(a)&=
(a_k^{-1}a_m^{2r+1})^{\binom{n+r-1}{r}}
\prod_{0\le i+j<n}
\prod_{\substack{1\le l\le r\\[1pt] l\ne k}}
\left(
\frac{e(t^ia_k,t^ja_l;p)}{e(t^iqa_k,t^ja_l;p)}
\right)^{\binom{n-i-j+r-3}{r-2}}
\\
&\qquad\cdot
\prod_{i=0}^{n-1}
\prod_{\substack{1\le l\le m-1\\[1pt] l\ne k}}
\left(
\frac
{\theta(t^ia_ka_l;p)}
{\theta(t^ia_ma_l;p)}
\right)^{\binom{n-i+r-2}{r-1}}, 
\end{split}
\end{equation}
and for $k\in\pr{r+1,\ldots,m-1}$
\begin{equation}\label{eq:detwtA2}
\begin{split}
\det \widetilde{\cA}^{k,m}(a)
&=
(a_k^{-1}a_m^{2r+1})^{\binom{n+r-1}{r}}
\prod_{i=0}^{n-1}
\prod_{\substack{1\le l\le m-1\\[1pt] l\ne k}}
\left(
\frac{\theta(t^ia_ka_l;p)}{\theta(t^ia_ma_l;p)}
\right)^{\binom{n-i+r-2}{r-1}}. 
\end{split}
\end{equation}
The determinants 
$\det  \cA^{k,m}(a)$ are obtained from these by 
applying $T_{pq,a_m}^{-1}$. 
\end{proof}
\par
\medskip
\begin{rem}{\rm
Under the assumption of Theorem \ref{thm:detcA}, by the $q$-difference equations \eqref{eq:TTI=SAI} 
the integrals $\cI_\mu(a)$ $(\mu\in Z_{r,n})$, regarded as functions on $(a_1,\ldots,a_{m-1})$, 
are continued meromorphically to the whole algebraic torus $(\mathbb{C}^*)^{m-1}$, 
and hence define meromorphic functions on the hypersurface $a_1\cdots a_m t^{2n-2}=pq$ in $(\mathbb{C}^*)^m$. 
}
\end{rem}
\subsection{Symmetry of the difference system 
with respect to $(p,q)$.}
In Theorem \ref{thm:detcA}, under the condition  $|p|<|t|^{2n-2}$ we derived the system of 
$q$-difference equations for the integrals
\begin{equation*}
\cI_{\mu}(a)=\la E_{\mu}(a_{\pr{1,\ldots,r}};z;p),g(z)\ra_{\Phi}
\quad(\mu\in Z_{r,n})
\end{equation*}
defined by the interpolation basis 
of $\cH^{(p)}_{r-1, n}$ and a holomorphic function 
$g(z)\in\cH^{(q)}_{r-1,n}$.  
In this formulation, 
we imposed on $g(z)$ 
the $q$-shift invariance with respect to the $a$ parameters 
so that 
the coefficient matrices should not depend on the 
choice of $g(z)$. 
We now modify the interpolation basis appropriately 
in order to make the $q$-difference system consistent 
with the $(p,q)$ symmetry of the bilinear form 
\begin{equation*}
\la\ ,\ \ra_{\Phi}:\ \ 
\cH^{(p)}_{r-1,n}\times 
\cH^{(q)}_{r-1,n}\to \bC.
\end{equation*}
Recall 
the dual Cauchy formula \eqref{eq:dualCauchy}
for the interpolation functions 
$E_{\mu}(a_{\pr{1,\ldots,r}};z;p)$:
For two sets of variables $z=(z_1,\ldots,z_n)$ and $w=(w_1,\ldots,w_{r-1})$, we have
\begin{equation}\label{eq:dualCauchyar}
\prod_{j=1}^{n}\prod_{l=1}^{r-1} e(z_j,w_i;p)
=\sum_{\mu\in Z_{r,n}} 
E_{\mu}(a_{\pr{1,\ldots,r}};z;p)
F_{\mu}(a_{\pr{1,\ldots,r}};w;p),
\end{equation}
where 
\begin{equation}\label{eq:defFmu}
F_{\mu}(a_{\pr{1,\ldots,r}};w;p)=
\prod_{k=1}^{r}\prod_{l=1}^{r-1} e(a_k,w_l;p)_{t,\mu_k}
\quad(\mu\in Z_{r,n}).
\end{equation}
In view of this formula, we set 
\begin{equation}\label{eq:deffmu}
f_{\mu}(z;w;p)=E_{\mu}(a_{\pr{1,\ldots,r}};z;p)
F_{\mu}(a_{\pr{1,\ldots,r}};w;p)\quad(\mu\in Z_{r,n}). 
\end{equation}
\begin{lem}
The functions 
$f_{\mu}(z;w;p)$ defined by \eqref{eq:deffmu} are invariant 
with respect to the $p$-shifts in the $a$ parameters,  namely, 
\begin{equation*}
T_{p,a_k}f_{\mu}(z;w;p)=f_{\mu}(z;w;p)\quad (k=1,\ldots,m). 
\end{equation*}
\end{lem}
\begin{proof}{Proof}
Since $f_{\mu}(z;w;p)$ do not depend on $a_k$ ($k=r+1,\ldots,m$), 
we show the invariance of $f_{\mu}(z;w;p)$ with respect to 
$T_{p,a_k}$ ($k=1,\ldots,r$). 
Applying $T_{p,a_k}$ to \eqref{eq:dualCauchyar} we have
\begin{equation*}
\prod_{j=1}^{n}\prod_{l=1}^{r-1} e(z_j,w_i;p)
=\sum_{\mu\in Z_{r,n}} 
T_{p,a_k}(E_{\mu}(a_{\pr{1,\ldots,r}};z;p))
T_{p,a_k}(F_{\mu}(a_{\pr{1,\ldots,r}};w;p)).
\end{equation*}
By \eqref{eq:defFmu} 
it is directly 
checked that
\begin{equation*}
T_{p,a_l}F_{\mu}(a_{\pr{1,\ldots,r}};w;p)=
F_{\mu}(a_{\pr{1,\ldots,r}};w;p)
(t^{2\binom{\mu_k}{2}}p^{\mu_k}a_k^{2\mu_k})^{-r+1}
\quad(k=1,\ldots,r). 
\end{equation*}
Since $F_{\mu}(a_{\pr{1,\ldots,r}};w;p)$ are linearly independent 
as functions in $w$, we see that 
\begin{equation*}
T_{p,a_k}(
E_{\mu}(a_{\pr{1,\ldots,r}};z;p))
=
E_{\mu}(a_{\pr{1,\ldots,r}};z;p)
(t^{2\binom{\mu_k}{2}}p^{\mu_k}a_k^{2\mu_k})^{r-1}. 
\end{equation*}
It also implies that the functions $f_{\mu}(z;w;p)=
E_{\mu}(a_{\pr{1,\ldots,r}};z;p)
F_{\mu}(a_{\pr{1,\ldots,r}};w;p)
$ 
are invariant with respect to $T_{p,a_k}$ ($k=1,\ldots,r$). 
\end{proof}
\par
Introducing a new set of parameters $u=(u_1,\ldots,u_{r-1})$, 
we consider the integrals 
\begin{equation*}
I_{\mu}(a)=
F_{\mu}(a;u;p)\cI_{\mu}(a)
=\la f_{\mu}(z;u;p),g(z)\ra_{\Phi}.  
\end{equation*}
Then the system of 
$q$-difference equations to be satisfied by 
$I_{\mu}=I_{\mu}(a)$ are given by 
\begin{equation*}
T_{q,a_k}T_{q,a_m}^{-1}I_{\mu}
=\sum_{\nu\in Z_{r,n}}A_{\mu,\nu}^{k,m}(a)
I_{\nu}
\qquad(\mu\in Z_{r,n}),
\end{equation*}
where 
\begin{equation*}
A_{\mu,\nu}^{k,m}(a)
=\frac{T_{p,a_k}F_{\mu}(a_{\pr{1,\ldots,r}};u;p)}
{F_{\nu}(a_{\pr{1,\ldots,r}};u;p)}
\cA_{\mu,\nu}^{k,m}(a)\quad(k=1,\ldots,m-1).
\end{equation*}
To be more precise, we have 
\begin{equation*}
A_{\mu,\nu}^{k,m}(a)
=\begin{cases}
\ds\ 
\prod_{j=1}^{r-1}
\frac{e(qa_k,u_j;p)_{t,\mu_k}}{e(a_k,u_j;p)_{t,\nu_k}}
\prod_{\substack{1\le l\le r\\[1pt] l\ne k}}
\prod_{j=1}^{r-1}
\frac{e(a_l,u_j;p)_{t,\mu_l}}{e(a_l,u_j;p)_{t,\nu_l}}
\cA_{\mu,\nu}^{k,m}(a)\ \ &(k=1,\dots,r),
\\
\ds\ 
\prod_{l=1}^{r}
\prod_{j=1}^{r-1}
\frac{e(a_l,u_j;p)_{t,\mu_l}}{e(a_l,u_j;p)_{t,\nu_l}}
\,
\cA_{\mu,\nu}^{k,m}(a)\ \ &(k=r+1,\ldots,m-1). 
\end{cases}
\end{equation*}
Hence the determinant 
of the matrix 
$A^{k,m}(a)=(A^{k,m}_{\mu,\nu}(a))_{\mu,\nu\in Z_{r,n}}$
is computed by \eqref{eq:detA1}, \eqref{eq:detA2}
as 
\begin{equation}\label{eq:detcA1}
\begin{split}
\det A^{k,m}(a)
&=
\prod_{i=1}^n
\prod_{l=1}^{r}
\left(
\frac{e(qa_k,u_l;p)_{t,i}}{e(a_k,u_l;p)_{t,i}}
\right)^{\binom{n-i+r-2}{r-2}}
\\
&\quad\cdot
\prod_{\substack{i,j\ge0\\\, i+j<n}}
\prod_{\substack{1\le l\le r\\[1pt] l\ne k}}
\left(\frac{e(t^ia_k,t^ja_l;p)}
{e(t^iqa_k,t^ja_l;p)}\right)^{\binom{n-i-j+r-3}{r-2}}
\\
&\quad\cdot
\prod_{i=0}^{n-1}
\prod_{\substack{1\le l\le m\\[1pt] l\ne k}}
\left(
\frac{\theta(t^ia_ka_l;p)}{\theta(t^iq^{-1}a_ma_l;p)}
\right)^{\binom{n-i+r-2}{r-1}}
\end{split}
\end{equation}
for $k\in\pr{1,\ldots,r}$, 
and 
\begin{equation}\label{eq:detcA2}
\begin{split}
\det A^{k,m}(a)
=
\prod_{i=0}^{n-1}
\prod_{\substack{1\le l\le m\\[1pt] l\ne k}}
\left(
\frac{\theta(t^ia_ka_l;p)}{\theta(t^iq^{-1}a_ma_l;p)}
\right)^{\binom{n-i+r-2}{r-1}}  
\end{split}
\end{equation}
for $k\in\pr{r+1,\ldots, m-1}$.
When we need to make the bases 
$p,q$ and the parameters $u$ explicit 
we use the notation 
$A^{k,m}(a;u;p,q)$ for 
$A^{k,m}(a)$.
\par\medskip
In order to deal with the two bases 
$(p,q)$ on an equal footing, 
we introduce two sets of parameters 
$u=(u_1,\ldots,u_{r-1})$, $v=(v_1,\ldots,v_{r-1})$, 
and define
\begin{equation}\label{eq:defcImunu}
\begin{split}
I_{\mu,\nu}(a;u,v)
&=\la f_{\mu}(z;u;p),f_{\nu}(z;v;q)\ra_{\Phi}
\\
&=\int_{\bT^n}
f_{\mu}(z;u;p)
f_{\nu}(z;v;q)
\Phi(z;a)\omega_n(z)
\end{split}
\end{equation}
for $\mu,\nu\in Z_{r,n}$. 
We suppose that $|p|<|t|^{2n-2}$ and $|q|<|t|^{2n-2}$.
Then, by the symmetry with respect to $(p,q)$
the square matrix 
\begin{equation}\label{eq:defcI}
I(a;u,v)=(I_{\mu,\nu}(a;u,v))_{\mu,\nu\in 
Z_{r,n}}
\end{equation}
satisfies the following system of 
$q$- and $p$-difference equations with respect to the 
$a$ parameters: For each $k=1,\ldots,m-1$, 
\begin{equation}\label{eq:pqDiffEq}
\begin{split}
T_{q,a_k}T_{q,a_m}^{-1}I_{\mu,\lambda}(a;u,v)
&=\sum_{\nu\in Z_{r,n}}A_{\mu,\nu}^{k,m}(a;u;p,q)I_{\nu,\lambda}(a;u,v), 
\\
T_{p,a_k}T_{p,a_m}^{-1}I_{\lambda,\mu}(a;u,v)
&=
\sum_{\nu\in Z_{r,n}}
A_{\mu,\nu}^{k,m}(a;v;q,p)
I_{\lambda,\nu}(a;u,v)
\qquad(\lambda,\mu\in Z_{r,n}),
\end{split}
\end{equation}
or equivalently 
\begin{equation}
\label{eq:TTI=AI}
\begin{split}
T_{q,a_k}T_{q,a_m}^{-1}I(a;u,v)
&=A^{k,m}(a;u;p,q)I(a;u,v),
\\
T_{p,a_k}T_{p,a_m}^{-1}I(a;u,v)
&=
I(a;u,v)
A^{k,m}(a;v;q,p)^{\mathrm{t}}
\end{split}
\end{equation}
in the matrix notation. 
Hence the determinant $J=\det I(a;u,v)$ satisfies 
the $q$- and $p$-difference equations 
\begin{equation}\label{eq:pqDEdet}
\begin{split}
T_{q,a_k}T_{q,a_m}^{-1}J
&=\det A^{k,m}(a;u;p,q)\,J,
\\
T_{p,a_k}T_{p,a_m}^{-1} J
&=
\det A^{k,m}(a;v;q,p)\,
J\quad(k=1,\ldots,m-1)
\end{split}
\end{equation}
of rank one.  
By inspecting 
the explicit formulas \eqref{eq:detcA1}, \eqref{eq:detcA2}
for the determinants of the coefficient matrices, 
it is directly verified that the function 
\begin{equation}\label{eq:partsol}
\begin{split}
J_0(a;u,v)&=
\prod_{i=0}^{n-1}\prod_{1\le k<l\le m}
\Gamma(t^ia_ka_l;p,q)^{\binom{n-i+r-2}{r-1}}
\\
&
\quad\cdot
\frac{
\prod_{i=1}^n
\prod_{k=1}^{r}
\prod_{l=1}^{r-1}
\left(
e(a_k,u_l;p)_{t,i}e(a_k,v_l;q)_{t,i}
\right)^{\binom{n-i+r-2}{r-2}}
}{
\prod_{0\le i+j<n}\prod_{1\le k<l\le r}
\left(
e(t^ia_k,t^ja_l;p)
e(t^ia_k,t^ja_l;q)
\right)
^{\binom{n-i-j+r-3}{r-2}}}
\end{split}
\end{equation}
provides with a particular solution of the 
system of 
$q$- and $p$- difference equations. 
Since the two meromorphic 
functions $\det I(a;u,v)$ and 
$J_0(a;u,v)$ both satisfy the difference 
equations 
\eqref{eq:pqDEdet}, the ratio 
$\det I(a;u,v)/J_0(a;u,v)$ is invariant 
with respect to the $q$- and $p$-shifts 
in the $a$ parameters simultaneously. 
This implies that this ratio is a constant 
which does not depend on the $a$ parameters. 
\begin{thm}\label{thm:detcI} 
Suppose that $|p|<|t|^{2n-2}$ and $|q|<|t|^{2n-2}$. 
Under the condition 
$t^{2n-2}a_1\cdots a_m=pq$ with $m=2r+4$, 
let $I(a;u,v)=(I_{\mu,\nu}(a;u,v))_{\mu,\nu\in Z_{r,n}}$ 
be the square matrix defined by the integrals 
\eqref{eq:defcImunu}.
The determinant of $I(a;u,v)$ is expressed as 
\begin{equation*}
\begin{split}
\det\, I(a;u,v)
&=c_{r,n}\,
\prod_{i=0}^{n-1}
\prod_{1\le k<l\le m}
\Gamma(t^ia_ka_l;p,q)^{\binom{n-i+r-2}{r-1}}
\\
&\quad\cdot
\frac{
\prod_{i=1}^n
\prod_{k=1}^{r}
\prod_{l=1}^{r-1}
\left(
e(a_k,u_l;p)_{t,i}e(a_k,v_l;q)_{t,i}
\right)^{\binom{n-i+r-2}{r-2}}
}{
\prod_{0\le i+j<n}\prod_{1\le k<l\le r}
\left(
e(t^ia_k,t^ja_l;p)
e(t^ia_k,t^ja_l;q)
\right)
^{\binom{n-i-j+r-3}{r-2}}}, 
\end{split}
\end{equation*}
where $c_{r,n}$ is a constant which does not 
depend on the parameters $a=(a_1,\ldots,a_m)$.  
\end{thm}
\par
In the next section, we will give an explicit formula for 
$c_{r,n}$ as a function of $(p,q,t)$ as in Theorem \ref{thm:1A}. 
Since $c_{r,n}\ne 0$, 
$I(a;u,v)$ is in fact a fundamental 
matrix solution of the system of 
$q$-difference equations
\begin{equation*}
T_{q,a_k}T_{q,a_m}^{-1}I(a;u,v)=A^{k,m}(a;u;p,q)I(a;u,v)
\qquad(k=1,\ldots,m-1). 
\end{equation*}
\par\medskip
As in the previous subsection, we consider the integrals 
\begin{equation*}
K_{\mu,\nu}(a)=\la
E_{\mu}(a_{\pr{1,\ldots,r}};z;p), E_{\nu}(a_{\pr{1,\ldots,r}};z;q)
\ra_{\Phi}
\quad(\mu,\nu\in Z_{r,n})
\end{equation*}
defined by the interpolation bases for 
$\cH_{r-1,n}^{(p)}$, $\cH_{r-1,n}^{(q)}$, 
and set 
$K(a)=\big(K_{\mu,\nu}(a)\big)_{\mu,\nu\in Z_{r,n}}$. 
Then we have
\begin{equation*}
I_{\mu,\nu}(a;u,v)=F_{\mu}(a_{\pr{1,\ldots,r}};u;p)
F_{\nu}(a_{\pr{1,\ldots,r}};v;q)
K_{\mu,\nu}(a)\quad(\mu,\nu\in Z_{r,n}), 
\end{equation*}
which implies
\begin{equation*}
\begin{split}
\det I(a;u,v)
&=
\det K(a)\,
\prod_{\mu\in Z_{r,n}}
F_{\mu}(a_{\pr{1,\ldots,r}};u;p)
F_{\mu}(a_{\pr{1,\ldots,r}};v;q)
\\
&=
\det K(a)\,
\prod_{i=1}^{n}
\prod_{k=1}^{r}
\prod_{l=1}^{r-1}
\big(e(a_k,u_l;p)_{t,i}
e(a_k,v_l;q)_{t,i}
\big)^{\binom{n-i+r-2}{r-2}}. 
\end{split}
\end{equation*}
By Theorem \ref{thm:detcI} the determinant 
of the matrix $K(a)$ 
is expressed as follows. 
\begin{cor} \label{cor:detK(a)}
Under the condition 
$t^{2n-2}a_1\cdots a_m=pq$ with $m=2r+4$, 
we have 
\begin{equation*}
\det K(a)=c_{r,n}
\frac
{
\prod_{i=0}^{n-1}
\prod_{1\le k<l\le m}
\Gamma(t^ia_ka_l;p,q)^{\binom{n-i+r-2}{r-1}}
}
{\prod_{0\le i+j<n}\prod_{1\le k<l\le r}
\left(
e(t^ia_k,t^ja_l;p)
e(t^ia_k,t^ja_l;q)
\right)
^{\binom{n-i-j+r-3}{r-2}}},
\end{equation*}
where $c_{r,n}$ is a constant independent of the parameters $a=(a_1,\ldots,a_m)$. 
\end{cor}
\par
As in Theorem \ref{thm:1A} 
we now consider the integrals 
\begin{equation}
\label{eq:Kmunu(a;x,y)}
K_{\mu,\nu}(a;x,y)=\la
E_{\mu}(x;z;p), E_{\nu}(y;z;q)
\ra_{\Phi}
\quad(\mu,\nu\in Z_{r,n})
\end{equation}
defined by the interpolation bases for 
$\cH_{r-1,n}^{(p)}$, $\cH_{r-1,n}^{(q)}$ 
with respect to generic parameters 
$x=(x_1,\ldots,x_r)$, $y=(y_1,\ldots,y_r)$. 
Note that 
\begin{equation}
\label{eq:EmuEnu}
\begin{split}
E_{\mu}(x;z;p)&=\sum_{\alpha\in Z_{r,n}}
E_{\mu}(x;(a_{\pr{1,\ldots,r}})_{t,\alpha};p)
E_{\alpha}(a_{\pr{1,\ldots,r}};z;p), 
\\
E_{\nu}(y;z;q)&=\sum_{\beta\in Z_{r,n}}
E_{\nu}(y;(a_{\pr{1,\ldots,r}})_{t,\beta};q)
E_{\beta}(a_{\pr{1,\ldots,r}};z;q)
\end{split}
\end{equation}
by the property of interpolation functions. 
Also, by \cite[Theorem 4.1]{INSlaterBC} 
the determinants of these transition matrices are given by 
\begin{equation*}
\begin{split}
\det \big(E_{\mu}(x;(a_{\pr{1,\ldots,r}})_{t,\nu};p)\big)_{\mu,\nu\in Z_{r,n}}
&=
\prod_{0\le i+j<n}\prod_{1\le k<l\le r}
\left(
\frac{e(t^ia_k,t^ja_l;p)}
{e(t^ix_k,t^jx_l;p)}
\right)
^{\binom{n-i-j+r-3}{r-2}}, 
\\
\det \big(E_{\mu}(y;(a_{\pr{1,\ldots,r}})_{t,\nu};q)\big)_{\mu,\nu\in Z_{r,n}}
&=
\prod_{0\le i+j<n}\prod_{1\le k<l\le r}
\left(
\frac{e(t^ia_k,t^ja_l;q)}
{e(t^iy_k,t^jy_l;q)}
\right)
^{\binom{n-i-j+r-3}{r-2}}. 
\end{split}
\end{equation*}
Hence the determinant of the matrix 
$K(a;x,y)=\big(K_{\mu,\nu}(a;x,y)\big)_{\mu,\nu\in Z_{r,n}}$
is computed as 
\begin{equation*}
\det K(a;x;y)=\det K(a) 
\prod_{0\le i+j<n}\prod_{1\le k<l\le r}
\left(
\frac{e(t^ia_k,t^ja_l;p)}
{e(t^ix_k,t^jx_l;p)}
\frac{e(t^ia_k,t^ja_l;q)}
{e(t^iy_k,t^jy_l;q)}
\right)
^{\binom{n-i-j+r-3}{r-2}}. 
\end{equation*}
Then, by Corollary \ref{cor:detK(a)} we obtain the following
expression for $\det K(a;x,y)$.
\begin{cor} 
Under the condition 
$t^{2n-2}a_1\cdots a_m=pq$ with $m=2r+4$, we have
\label{cor:det K(axy)}
\begin{equation*}
\det K(a;x,y)=c_{r,n}
\frac
{
\prod_{i=0}^{n-1}
\prod_{1\le k<l\le m}
\Gamma(t^ia_ka_l;p,q)^{\binom{n-i+r-2}{r-1}}
}
{\prod_{0\le i+j<n}\prod_{1\le k<l\le r}
\left(
e(t^ix_k,t^jx_l;p)
e(t^iy_k,t^jy_l;q)
\right)
^{\binom{n-i-j+r-3}{r-2}}}, 
\end{equation*}
where $c_{r,n}$ is a constant independent of the parameters $a=(a_1,\ldots,a_m)$. 
\end{cor}
\begin{rem}{\rm 
We compute the constant $c_{r,n}$ later in Section \ref{section:5}, and eventually see that 
\begin{equation}
\label{eq:explicitcrn}
c_{r,n}=
\left(
\frac{2^n n!}
{(p;p)_{\infty}^n (q;q)_\infty^n}
\right)^{\binom{n+r-1}{r-1}}
\prod_{i=0}^{n-1}
\left(\frac{\Gamma(t^{i+1};p,q)}
{\Gamma(t;p,q)}\right)^{r\binom{n-i+r-2}{r-1}}.
\end{equation}
As a consequence, Corollary \ref{cor:det K(axy)} with the explicit formula \eqref{eq:explicitcrn} of $c_{r,n}$ 
implies Theorem \ref{thm:1A}.
Once the constant $c_{r,n}$ has been determined, 
we see that 
Theorem \ref{thm:detcI} and its corollaries 
are valid for $|p|<1$ and $|q|<1$
without any particular restriction.
}
\end{rem}
\par
The system of $q$- and $p$-difference equations for the matrix $K(a;x,y)$ as stated in Theorem \ref{thm:1B}
can be derived from the system \eqref{eq:TTI=AI} 
for $I(a;u,v)$. 
From \eqref{eq:deffmu}  and \eqref{eq:EmuEnu} the transition between $E_\mu(x;z;p)$ and $f_\alpha(z;u;p)$ is given by 
\begin{equation*}
E_{\mu}(x;z;p)=\sum_{\nu\in Z_{r,n}}
G_{\mu\nu}(a;x,u;p)
f_{\nu}(z;u;p),
\end{equation*}
where 
\begin{equation*}
G_{\mu\nu}(a;x,u;p)=\frac{E_{\mu}(x;(a_{\pr{1,\ldots,r}})_{t,\nu};p)}
{F_\nu(a_{\pr{1,\ldots,r}};u;p)}.
\end{equation*}
From \eqref{eq:defcImunu} and \eqref{eq:Kmunu(a;x,y)} we have 
\begin{equation*}
K(a;x,y)=G(a;x,u;p)I(a;u,v)G(a;y,v;q)^{\mathrm{t}}, 
\end{equation*}
where 
\begin{equation*}
G(a;x,u;p)=(G_{\mu\nu}(a;x,u;p))_{\mu,\nu\in Z_{r,n}}.
\end{equation*}
Since $G(a;y,v;q)$ is invariant under the $q$-shifts in $a$ parameters, by \eqref{eq:TTI=AI}, for $k=1,\ldots,m-1$, we have
\begin{equation*}
T_{q,a_k}T_{q,a_m}^{-1}K(a;x,y)=\big(T_{q,a_k}T_{q,a_m}^{-1}G(a;x,u;p) \big)A^{k,m}(a;u;p,q)I(a;u,v)
G(a;y,v;q)^{\mathrm{t}}, 
\end{equation*}
and hence, 
\begin{equation}
\label{eq:TTK=MK}
T_{q,a_k}T_{q,a_m}^{-1}K(a;x,y)=M^{k,m}(a;x;p,q)K(a;x,y),
\end{equation}
where 
\begin{equation}
\label{eq:Mkm(a;x;p,q)}
M^{k,m}(a;x;p,q)=\big(T_{q,a_k}T_{q,a_m}^{-1}G(a;x,u;p) \big)A^{k,m}(a;u;p,q)
G(a;x,u;p)^{-1}. 
\end{equation}
Note that these matrices are actually independent of $u$ as can be seen from \eqref{eq:TTK=MK}, provided $\det K(a;x,y)\ne 0$. 
Since $K(a;x,y)$ is invariant under the permutation of $a_1,\ldots,a_m$,
the $q$-difference equations \eqref{eq:1Bq} are obtained from \eqref{eq:Mkm(a;x;p,q)} by symmetry. 
Also, the $p$-difference equations \eqref{eq:1Bp} follow from the symmetry of $K(a;x,y)$ with respect to $p$ and $q$.
This completes the proof of Theorem \ref{thm:1B} under the assumption $c_{r,n}\ne 0$. 
\section{\boldmath Computation of the constants $c_{r,n}$} 
\label{section:5}
\subsection{Determinant of the bilinear form}
In this section, 
we use the notation $K^{(r,n)}_{\mu,\nu}(a)$ and 
$\Phi_n(z;a)$ for 
$K_{\mu,\nu}(a)$ and $\Phi(z)=\Phi(z;a;p,q)$ 
respectively, 
in order to make explicit the dependence on 
$(r,n)$ and $a=(a_1,\ldots,a_m)$. 
Namely, 
\begin{equation*}
\begin{split}
K^{(r,n)}_{\mu,\nu}(a)
&=\la E_{\mu}(a_{\pr{1,\ldots,r}};z;p),E_{\nu}(a_{\pr{1,\ldots,r}};z;q)
\ra_{\Phi}
\\
&=\int_{\bT^n}
E_{\mu}(a_{\pr{1,\ldots,r}};z;p),E_{\nu}(a_{\pr{1,\ldots,r}};z;q)
\Phi_n(z;a)\omega_n(z)
\qquad(\mu,\nu\in Z_{r,n}). 
\end{split}
\end{equation*}
Under the conditions $|p|<|t|^{2n-2}$, $|q|<|t|^{2n-2}$ and 
$t^{2n-2}a_1\cdots a_m=pq$ $(m=2r+4)$, 
by Corollary \ref{cor:detK(a)}
the determinant of the matrix 
$K^{(r,n)}(a)=\big(K^{(r,n)}_{\mu,\nu}(a)\big)_{\mu,\nu\in Z_{r,n}}$ 
is expressed as 
\begin{equation*}
\det K^{(r,n)}(a)=
c_{r,n}\,L_{r,n}(a),
\end{equation*}
with a constant $c_{r,n}$ which does not depend on $a$, 
where 
\begin{equation*}
L_{r,n}(a)=
\ds
\frac{
\prod_{i=0}^{n-1}
\prod_{1\le k<l\le m}
\Gamma(t^ia_ka_l;p,q)^{\binom{n-i+r-2}{r-1}}
}{
\prod_{0\le i+j<n}
\prod_{1\le k<l\le r}
\big(
e(t^ia_k,t^ja_l;p)
e(t^ia_k,t^ja_l;q)
\big)^{\binom{n-i-j+r-3}{r-2}}
}. 
\end{equation*}
In the following, we determine the unknown constant 
$c_{r,n}$ by comparing the asymptotic behavior of the 
two meromorphic functions 
$\det K^{(r,n)}(a)$ and $L_{r,n}(a)$ 
around their poles. 
\subsection{Asymptotic behavior of $L_{r,n}(a)$}
\label{subsection:5.2}
Among the parameters $a_1,\ldots,a_m$ ($m=2r+4$), 
we choose two parameters $a_1$ and $a_{r+1}$ 
and analyze the singularity of $L_{r,n}(a)$ 
along the pole 
$1-a_1a_{r+1}=0$. 
Since $L_{r,n}(a)$ has the factor 
\begin{equation*}
\Gamma(a_1a_{r+1};p,q)^{\binom{n+r-2}{r-1}}
=
\left(
\frac{(pq/a_1a_{r+1};p,q)_\infty}{(a_1a_{r+1};p,q)_\infty}
\right)^{\binom{n+r-2}{r-1}}, 
\end{equation*}
it has a pole 
of multiplicity $\binom{n+r-2}{r-1}$  
along the hypersurface $1-a_1a_{r+1}=0$. 
We compute the limit 
\begin{equation}\label{eq:wL}
\widetilde{L}_{r,n}(\widetilde{a})=
\lim_{a_{r+1}\to a_1^{-1}}(1-a_1a_{r+1})^{\binom{n+r-2}{r-1}}
L_{r,n}(a)
\end{equation}
as $a_{r+1}\to a_1^{-1}$. 
In this procedure, we regard 
$(a_1,\ldots,a_{m-1})$ as independent variables and 
$a_m=pq/t^{2n-2}a_1\cdots a_{m-1}$ as a function of 
$(a_1,\ldots,a_{m-1})$.  
Note that as $a_{r+1}\to a_1^{-1}$, $a_m$ has the limit 
\begin{equation*}
\widetilde{a}_m=
\lim_{a_{r+1}\to a_1^{-1}}a_m=pq/t^{2n-2}a_2\cdots 
a_ra_{r+2}\cdots a_{m-1}. 
\end{equation*}
Also, in the notation 
$\widetilde{L}_{r,n}(\widetilde{a})$, 
$\widetilde{a}$ stands for 
\begin{equation*}
\widetilde{a}=(a_1,\ldots,a_r,a_{r+2},\ldots,
a_{m-1},\widetilde{a}_m). 
\end{equation*}
The limit 
$\widetilde{L}_{r,n}(\widetilde{a})$ is computed explicitly 
as follows.  We first rewrite $L_{r,n}(a)$ as 
\begin{equation*}
\begin{split}
L_{r,n}(a)
&=
\Gamma(a_1a_{r+1};p,q)^{\binom{n+r-2}{r-1}}
\prod_{i=1}^{n-1}
\Gamma(t^ia_1a_{r+1};p,q)^{\binom{n-i+r-2}{r-1}}
\\
&\quad\cdot
\prod_{i=0}^{n-1}
\prod_{\substack{1\le k\le m\\ k\ne 1,r+1}}
\Gamma(t^{i}a_1a_k,t^{i}a_{r+1}a_k;p,q)^{\binom{n-i+r-2}{r-1}}
\\
&\quad\cdot
\frac{
\prod_{i=0}^{n-1}
\prod_{\substack{1\le k<l\le m\\ k,l\ne 1,r+1}}
\Gamma(t^ia_ka_l;p,q)^{\binom{n-i+r-2}{r-1}}
}{
\prod_{0\le i+j<n}
\prod_{1\le k<l\le r}
\big(
e(t^ia_k,t^ja_l;p)
e(t^ia_k,t^ja_l;q)
\big)^{\binom{n-i-j+r-3}{r-2}}
}.  
\end{split}
\end{equation*}
Since
\begin{equation*}
\begin{split}
&\lim_{a_{r+1}\to a_1^{-1}}(1-a_1a_{r+1})
\Gamma(a_1a_{r+1};p,q)
\\
&
=
\lim_{a_{r+1}\to a_1}\frac{
(pq/a_1a_{r+1};pq)_\infty}
{
(pa_1a_{r+1};p)_\infty
(qa_1a_{r+1};q)_\infty
(pqa_1a_{r+1};pq)_\infty
}
\\
&=\frac{1}{(p;p)_\infty(q;q)_\infty}, 
\end{split}
\end{equation*}
we obtain 
\begin{equation*}
\begin{split}
\widetilde{L}_{r,n}(\widetilde{a})
&=\ds
\frac{
\prod_{i=1}^{n-1}
\Gamma(t^i;p,q)^{\binom{n-i+r-2}{r-1}}
}{
\left((p;p)_\infty(q;q)_\infty\right)^{\binom{n+r-2}{r-1}}
}
\,
\prod_{i=0}^{n-1}
\prod_{\substack{1\le k\le m\\ k\ne 1,r+1}}
\Gamma(t^{i}a_1^{\pm1}a_k;p,q)^{\binom{n-i+r-2}{r-1}}
\\
&\quad\cdot
\frac{
\prod_{i=0}^{n-1}
\prod_{\substack{1\le k<l\le m\\ k,l\ne 1,r+1}}
\Gamma(t^ia_ka_l;p,q)^{\binom{n-i+r-2}{r-1}}
}{
\prod_{0\le i+j<n}
\prod_{1\le k<l\le r}
\big(
e(t^ia_k,t^ja_l;p)
e(t^ia_k,t^ja_l;q)
\big)^{\binom{n-i-j+r-3}{r-2}}
}. 
\end{split}
\end{equation*}
Here $a_m$ in the right-hand side should be understood as its limit 
$\widetilde{a}_m=pq/t^{2n-2}a_2\cdots a_ra_{r+2}\cdots a_{m-1}$.
\subsection{Remark on analytic continuation}
Before proceeding to asymptotic analysis of $\det K^{(r,n)}(a)$, 
we give a general remark on analytic continuation of the integral 
\begin{equation}
\label{eq:la f(z;a) ra}
\la f(z;a) \ra_{\Phi}=\int_{\bTR^n}
f(z;a)\Phi_n(z;a)\omega_n(z)
\end{equation}
for a holomorphic function $f(z;a)$ on $\bTC^n\times \bTC^m$, 
which defines a holomorphic function on the domain 
\begin{equation*}
U=\big\{a=(a_1,\ldots,a_m)\in(\mathbb{C}^\ast)^m \big\vert\ 
|a_k|<1\ (k=1,\ldots,m) \big\}\subset(\mathbb{C}^\ast)^m.   
\end{equation*}  
This function can be continued to a 
holomorphic function on a larger domain by 
replacing $\mathbb{T}^n$ with an appropriate 
$n$-cycle depending on the parameters 
$(a_1,\ldots,a_m)$.  
\par\medskip
Notice that $\Phi_n(z;a)$ has poles possibly along the divisors
\begin{equation*}
\begin{split}
z_i^{\pm 1}&=a_k p^\mu q^\nu \quad
\ (1\le i\le n;\ k=1,\ldots,m;\ 
\mu,\nu\in\mathbb{N}
),\\
z_i^{\pm1}z_j^{\pm1}&=t\,p^\mu q^{\nu}
\qquad(1\le i<j\le n;\ 
\mu,\nu\in\mathbb{N}
).
\end{split}
\end{equation*}
Also, 
regarded as a function of $z_i$ ($i=1,\ldots,n$), 
it has poles possibly at 
\begin{equation*}
p^{\mu}q^{\nu}a_k,\quad 
p^{-\mu}q^{-\nu}a_k^{-1},\quad 
p^{\mu}q^{\nu}t z_j^{\pm1},\quad
p^{-\mu}q^{-\nu}t^{-1}z_j^{\pm1},
\end{equation*}
where $1\le k\le m$, $1\le j\le n$, $j\ne i$ and 
$\mu,\nu\in\mathbb{N}$. 
In view of this fact, for each $a=(a_1,\ldots,a_m)\in(\mathbb{C}^\ast)^m$, 
we define two subsets $S_0$, $S_\infty$ 
of $\mathbb{C}^\ast$ by
\begin{equation*}
\begin{split}
S_0&=\big\{\ p^{\mu}q^{\nu}a_k\ \big\vert\  
1\le k\le m;\ 
\mu,\nu\in\mathbb{N}
\ \big\},
\\ 
S_\infty&=\big\{\ p^{-\mu}q^{-\nu}a_k^{-1}\ \big\vert\  
1\le k\le m;\ 
\mu,\nu\in\mathbb{N}
\ \big\}, 
\end{split}
\end{equation*}
and suppose that $S_0\cap S_\infty=\phi$, namely 
$a_ka_l\notin p^{-\mathbb{N}}q^{-\mathbb{N}}
$ $(1\le k,l\le m)$.  
Assuming that $|t|<\rho^2$ for some $\rho\in(0,1]$, 
we choose a circle 
\begin{equation*}
C_\delta(0)=\big\{ u\in\mathbb{C}^\ast\ \big\vert\ 
|u|=\delta\ \big\},\quad \delta\in[\rho,\rho^{-1}], 
\end{equation*}
which does not intersect with $S_0\cup S_\infty$.  
Then we define a cycle $C$ in $\mathbb{C}^\ast$
by
\begin{equation*}
C=C_\delta(0)+
\sum_{c\in S_0; |c|>\delta}\,C_\varepsilon(c)-
\sum_{c\in S_\infty; |c|<\delta}\,C_\varepsilon(c),
\end{equation*}
where $C_\varepsilon(c)$ denotes a sufficiently small circle 
around $c$.  
Note that, if $|a_k|<1$ ($k=1,\ldots,m$), then $C$ 
is homologous to the unit circle. 
\par
We now assume that  $|a_k|<\rho^{-1}$ ($k=1,\ldots,m$).   
Then such a cycle $C$ can be taken inside the annulus 
$\{ u\in\mathbb{C}^\ast\ \vert\ \rho\le |u|\le \rho^{-1}\}$.  
Since $|t|<\rho^2$, the meromorphic function 
$\Phi_n(z;a)$ is holomorphic in an neighborhood of the 
$n$-cycle $C^n=C\times\cdots\times C$.  
Hence, the integral 
\begin{equation}
\label{eq:int_Cn f(z;a)..}
\int_{C^n}f(z;a)\Phi_n(z;a)\omega_n(z)
\end{equation}
is well defined, and does not depend on the choice 
of $\delta\in [\rho,\rho^{-1}]$.  This implies the following 
lemma on analytic continuation.
\begin{lem}\label{lem:analcont}
Suppose that $|t|<\rho^2$ for some 
$\rho\in(0,1]$. Then the holomorphic function 
\eqref{eq:la f(z;a) ra} 
on the domain $U$ can be continued by the integral \eqref{eq:int_Cn f(z;a)..}
to a holomorphic function on 
\begin{equation}\label{eq:defUext}
\big\{a=(a_1,\ldots,a_m)\in(\mathbb{C}^\ast)^m
\big\vert\ |a_k|<\rho^{-1}\ (1\le k\le m),\ 
a_ka_l\notin p^{-\mathbb{N}}q^{-\mathbb{N}}
\ (1\le k,l\le m) 
\ \big\}.
\end{equation}
\hfill$\square$
\end{lem}
We remark that, when $f(z;a)$ depends meromorphically on $a$, 
the integral \eqref{eq:la f(z;a) ra} is continued similarly to a meromorphic function on the domain \eqref{eq:defUext}.
\subsection{Asymptotic behavior of $\det K^{(r,n)}(a)$}
Applying the same procedure as in Subsection \ref{subsection:5.2}
to $\det K^{(r,n)}(a)$, 
we compute the limit 
\begin{equation}\label{eq:limKrn}
\widetilde{K}_{r,n}(\widetilde{a})=
\lim_{a_{r+1}\to a_1^{-1}}(1-a_1a_{r+1})^{\binom{n+r-2}{r-1}}
\det 
K^{(r,n)}(a). 
\end{equation}
Note here that the power $\binom{n+r-2}{r-1}$ is the cardinality 
of $Z_{r,n-1}$. 
The indexing set $Z_{r,n}$ for the matrix $K^{(r,n)}(a)$ 
is divided into two parts as 
\begin{equation*}
Z_{r,n}=Z_{r,n}^{0}\sqcup Z_{r,n}^{+};\quad
Z_{r,n}^{0}=\prm{\mu\in Z_{r,n}}{\mu_1=0},
\ \ 
Z_{r,n}^{+}=\prm{\mu\in Z_{r,n}}{\mu_1>0}, 
\end{equation*} 
according 
as $\mu_1=0$ or $\mu_{1}>0$. 
Since 
$Z_{r-1,n}\isom Z_{r,n}^{0}$ and 
$Z_{r,n-1}\isom Z_{r,n}^{+}=Z_{r,n-1}+\ep_1$, 
\begin{equation*}
\ts
\#Z_{r,n}^0=\binom{n+r-2}{r-2},\quad 
\#Z_{r,n}^+=\binom{n+r-2}{r-1}. 
\end{equation*}
and 
the above decomposition of $Z_{r,n}$ 
corresponds the identity 
$\binom{n+r-1}{r-1}=
\binom{n+r-2}{r-2}+\binom{n+r-2}{r-1}$
of binomial coefficients. 
\par
In order to compute the limit \eqref{eq:limKrn}, 
we analyze the asymptotic behavior of each 
matrix element $K^{(r,n)}_{\mu,\nu}(a)$ 
along the hypersurface $1-a_1a_{r+1}=0$, 
by the same method of pinching 
as we used in \cite{INIntBC}. 
As we remarked in Lemma \ref{lem:analcont}, 
in the region \eqref{eq:defUext} the integral $K^{(r,n)}_{\mu,\nu}(a)$ is expressed as 
\begin{equation*}
K^{(r,n)}_{\mu,\nu}(a)=\int_{C^n}
E_{\mu}(a_{\pr{1,\ldots,r}};z;p)
E_{\nu}(a_{\pr{1,\ldots,r}};z;q)\Phi_n(z;a)\omega_n(z)
\end{equation*}
over a certain $n$-cycle $C^n$, provided that $|t|<\rho^2$. 
Assuming that $\rho\in (0,1)$ satisfies 
\begin{equation}
\label{eq:region-a}
|p|<\rho,\quad |q|<\rho,\quad |pq/t^{2n-2}|<\rho^{m-2},
\end{equation}
we consider the situation where 
\begin{equation*}
1<|a_1|<\rho^{-1};\quad \rho<|a_k|<1\ \ (k=2,\ldots,m-1). 
\end{equation*}
In this case we can choose the cycle $C$ as 
\begin{equation*}
C=C_0+C_\varepsilon(a_1)-C_\varepsilon(a_1^{-1}); \quad C_0=C_1(0),
\end{equation*}
with sufficiently small $\varepsilon>0$,  
and analyze the effect of pinching about the cycles $C_\varepsilon(a_1)$, $C_\varepsilon(a_1^{-1})$ 
as $a_{r+1}\to a_1^{-1}$
in the region \eqref{eq:region-a}.
\par
We first consider the integral with respect to the variable $z_1$.  
When 
$a_{r+1}$ approaches to $a_1^{-1}$, 
the contour $C$ is pinched by the two pairs 
of poles $(a_{r+1},a_1^{-1})$ and 
$(a_{r+1}^{-1},a_1)$. 
Taking this into account we decompose 
$\Phi_n(z;a)$ as 
\begin{equation*}
\Phi_n(z;a)
=
\frac{
\Gamma(a_1z_1^{\pm1};p,q)
\prod_{k=2}^{m}\Gamma(a_kz_1^{\pm1};p,q)}
{\Gamma(z_1^{\pm2};p,q)}
\prod_{j=2}^{n}
\frac
{\Gamma(tz_1^{\pm1}z_j^{\pm1};p,q)}
{\Gamma(z_1^{\pm1}z_j^{\pm1};p,q)}
\,\Phi_{n-1}(z_{\widehat{1}};a),
\end{equation*}
where $z_{\widehat{1}}=(z_2,\ldots,z_n)$, and compute the residues 
at the poles $z_1=a_1^{\pm1}$. 
Then we obtain 
\begin{equation*}
\underset{z_1=a_1^{\ep}}
{\mathrm{Res}\ }\Big(\Phi_n(z;a)
\frac{dz_1}{z_1}
\Big)
=\ep 
\frac{\prod_{k=2}^{m}\Gamma(a_1^{\pm1}a_k;p,q)
}{(p;p)_\infty(q;q)_\infty\Gamma(a_1^{-2};p,q)}
\,\widehat{\Phi}_{n-1}(z_{\widehat{1}};a)
\end{equation*}
for 
$\ep=\pm1$, where 
\begin{equation*}
\widehat{\Phi}_{n-1}(z_{\widehat{1}};a)=
\prod_{j=2}^{n}
\frac{\Gamma(ta_1^{\pm}z_j^{\pm1};p,q)}
{\Gamma(a_1^{\pm}z_j^{\pm1};p,q)}
\,\Phi_{n-1}(z_{\widehat{1}};a). 
\end{equation*}
Setting 
\begin{equation*}
\Psi_{\mu,\nu}(z;a)=E_{\mu}(a_{\pr{1,\ldots,r}};z;p)
E_{\nu}(a_{\pr{1,\ldots,r}};z;q)\Phi_n(z;a)
\qquad(\mu,\nu\in Z_{r,n}), 
\end{equation*}
we compute
\begin{equation*}
\begin{split}
&
\frac{1}{2\pi\sqrt{-1}}
\int_{C}\Psi_{\mu,\nu}(z;a)\frac{dz_1}{z_1}
=\frac{1}{2\pi\sqrt{-1}}
\int_{C_0}\Psi_{\mu,\nu}(z;a)\frac{dz_1}{z_1}
\\
&\qquad\quad+\frac{2\prod_{k=2}^{m}\Gamma(a_1^{\pm1}a_k;p,q)
}{(p;p)_\infty(q;q)_\infty\Gamma(a_1^{-2};p,q)}
E_{\mu}(a_{\pr{1,\ldots,r}};a_1,z_{\widehat{1}};p)
E_{\mu}(a_{\pr{1,\ldots,r}};a_1,z_{\widehat{1}};q)
\widehat{\Phi}_{n-1}(z_{\widehat{1}};a).
\end{split}
\end{equation*}
By the same argument as in 
\cite{INIntBC}, 
we repeat this computation for  
$z_2,\ldots,z_n$.  
As a result we obtain 
\begin{equation}\label{eq:Kmunu=}
\begin{split}
&K^{(r,n)}_{\mu,\nu}(a)=
\int_{C^n}\Psi_{\mu,\nu}(z;a)\omega_n(z)
\\
&=\int_{C_0^n}\Psi_{\mu,\nu}(z;a)\omega_n(z)+\frac{2n\prod_{k=2}^{m}\Gamma(a_1^{\pm1}a_k;p,q)
}{(p;p)_\infty(q;q)_\infty\Gamma(a_1^{-2};p,q)}
\\
&\qquad\qquad\qquad\quad
\cdot
\int_{C^{n-1}}
E_{\mu}(a_{\pr{1,\ldots,r}};a_1,z_{\widehat{1}};p)
E_{\nu}(a_{\pr{1,\ldots,r}};a_1,z_{\widehat{1}};q)
\widehat{\Phi}_{n-1}(z_{\widehat{1}};a)
\omega_{n-1}(z_{\widehat{1}}).
\end{split}
\end{equation}
We remark that the first 
term of the right-hand side 
is regular along $1-a_1a_{r+1}=0$, and 
has a finite limit 
in the limit 
as $a_{r+1}\to a_1^{-1}$. 
\par
If $\mu_1=0$ or $\nu_1=0$, then the first term 
of the right-hand side of \eqref{eq:Kmunu=} is 0. 
In fact, when 
$\mu_1=0$, we have 
\begin{equation}\label{eq:Emu0}
E_{\mu}(a_{\pr{1,\ldots,r}};z;p)=\frac{\prod_{i=1}^{n}e(z_i,a_1;p)}
{\prod_{k=2}^{r}e(a_k,a_1;p)_{t,\mu_k}}
E_{(\mu_2,\ldots,\mu_r)}(a_{\pr{2,\ldots,r}};z;p), 
\end{equation}
and hence $E_{\mu}(a_{\pr{1,\ldots,r}};a_1,z_{\widehat{1}};p)=0$. 
Similarly, when $\nu_1=0$, 
we have 
$E_{\nu}(a_{\pr{1,\ldots,r}};a_1,z_{\widehat{1}};q)=0$.
Therefore, when 
$\mu_1=0$ or $\nu_1=0$, we obtain
\begin{equation*}
K^{(r,n)}_{\mu,\nu}(a)=
\int_{C^n}\Psi_{\mu,\nu}(z;a)\omega_n(z)
=
\int_{C_0^n}\Psi_{\mu,\nu}(z;a)\omega_n(z). 
\end{equation*}
Since the integral over 
$C_0^{n}$ is regular along
$1-a_1a_{r+1}=0$, we obtain 
\begin{equation*}
\begin{split}
&\lim_{a_{r+1}\to a_1^{-1}}(1-a_1a_{r+1})\int_{C^n}
K^{(r,n)}_{\mu,\nu}(a)=0,
\\
&
\lim_{a_{r+1}\to a_1^{-1}}
K^{(r,n)}_{\mu,\nu}(a)
=\int_{C_0^n}
\Psi_{\mu,\nu}(z;a)\Big|_{a_{r+1}=a_1^{-1}}\omega_n(z)
=K^{(r,n)}_{\mu,\nu}(a)\Big|_{a_{r+1}=a_1^{-1}}. 
\end{split}
\end{equation*}
\par\medskip
We now decompose the matrix 
$K_{r,n}(a)$ into four blocks as 
\begin{equation}\label{eq:KK}
\begin{split}
\begin{picture}(220,50)
\put(-10,30){
$
K^{(r,n)}(a)
=
\begin{bmatrix}\ 
K^{(r,n)}_{\mu,\nu}(a)\ 
&
K^{(r,n)}_{\mu,\nu}(a)\ 
\\[8pt]
\ 
K^{(r,n)}_{\mu,\nu}(a)\ 
&
K^{(r,n)}_{\mu,\nu}(a)\ 
\end{bmatrix}$
}
\put(180,40){\small$(\mu_1=0)$}
\put(180,20){\small$(\mu_1>0)$}
\put(65,0){\small$(\nu_1=0)$}
\put(120,0){\small$(\nu_1>0)$}
\end{picture}
\end{split}
\end{equation}
according 
to the partition
$Z_{r,n}=Z_{r,n}^{0}\sqcup Z_{r,n}^{+}$ 
of the indexing set. 
Note that $\#Z_{r,n}^+=\binom{n+r-2}{r-1}$ 
and that 
$\lim_{a_{r+1}\to a_1^{-1}} 
(1-a_1a_{r+1})
K^{(r,n)}_{\mu,\nu}(a)=0$ 
when $\nu_1=0$. 
Hence we compute 
\begin{equation*}
\begin{split}
&
\lim_{a_{r+1}\to a_1^{-1}}(1-a_1a_{r+1})^{\binom{n+r-2}{r-1}}
\det K^{(r,n)}(a)
\\
&
=\lim_{a_{r+1}\to a_1^{-1}}
\det
\begin{bmatrix}\ 
K^{(r,n)}_{\mu,\nu}(a)\ 
&
K^{(r,n)}_{\mu,\nu}(a)\ 
\\[8pt]
\ 
(1-a_1a_{r+1})K^{(r,n)}_{\mu,\nu}(a)\ 
&
(1-a_1a_{r+1})K^{(r,n)}_{\mu,\nu}(a)\ 
\end{bmatrix}
\\
&=
\det
\begin{bmatrix}\ 
K^{(r,n)}_{\mu,\nu}(a)\big|_{a_{r+1}=a_1^{-1}}
\ 
&
K^{(r,n)}_{\mu,\nu}(a)\big|_{a_{r+1}=a_1^{-1}}
\\[8pt]
\ 
0
&
\ds\lim_{a_{r+1}\to a_1^{-1}}(1-a_1a_{r+1})
K^{(r,n)}_{\mu,\nu}(a)\ 
\end{bmatrix}
\\
&=
\det\big(K^{(r,n)}_{\mu,\nu}(a)
\big|_{a_{r+1}=a_1^{-1}}
\big)_{\mu,\nu\in Z_{r,n}^0}
\det\big(
\lim_{a_{r+1}\to a_1^{-1}}(1-a_1a_{r+1})
K^{(r,n)}_{\mu,\nu}(a)\big)_{\mu,\nu\in Z_{r,n}^+}. 
\end{split}
\end{equation*}
\par
When 
$\mu_1=0$, $\nu_1=0$, with the notation 
$\mu=(0,\mu')$, $\nu=(0,\nu')$ 
we have 
\begin{equation*}
\begin{split}
E_{\mu}(a_{\pr{1,\ldots,r}};z;p)&=
\frac{\prod_{i=1}^{n}e(z_i,a_1;p)}
{\prod_{k=2}^{r}e(a_k,a_1;p)_{t,\mu_k}}
E_{\mu'}(a_{\pr{2,\ldots,r}};z;p),\quad
\\
E_{\nu}(a_{\pr{1,\ldots,r}};z;q)&=
\frac{\prod_{i=1}^{n}e(z_i,a_1;q)}
{\prod_{k=2}^{r}e(a_k,a_1;q)_{t,\nu_k}}
E_{\nu'}(a_{\pr{2,\ldots,r}};z;q).
\end{split}
\end{equation*}
Since 
\begin{equation*}
\Phi_n(z;a)\big|_{a_{r+1}=a_1^{-1}}
=\frac{
\Phi_n(z;a_{\widehat{1},\widehat{r+1}})
}{\prod_{i=1}^{n}e(z_i,a_1;p)e(z_i,a_1;q)},
\quad a_{\widehat{1},\widehat{r+1}}
=(a_2,\ldots,a_r,a_{r+2},\ldots,a_m), 
\end{equation*}
we obtain
\begin{equation*}
\begin{split}
K^{(r,n)}_{\mu,\nu}(a)\big|_{a_{r+1}=a_1^{-1}}
&=
\frac{1}{\prod_{k=2}^{r}
e(a_k,a_1;p)_{t,\mu_k}e(a_k,a_1;q)_{t,\nu_k}}
\\[6pt]
&\quad\cdot
\int_{\bT^n}
E_{\mu'}(a_{\pr{2,\ldots,r}};z;p)
E_{\mu'}(a_{\pr{2,\ldots,r}};z;q)
\Phi(z;a_{\widehat{1},\widehat{r+1}})
\omega_n(z)
\\
&=
\frac{
K^{(r-1,n)}_{\mu',\nu'}(a_{\widehat{1},\widehat{r+1}})
}{\prod_{k=2}^{r}
e(a_k,a_1;p)_{t,\mu_k}e(a_k,a_1;q)_{t,\nu_k}}
\end{split}
\end{equation*}
and hence 
\begin{equation*}
\begin{split}
&
\det\big(K^{(r,n)}_{\mu,\nu}(a)\big|_{a_{r+1}
=a_1^{-1}}\big)_{\mu,\nu\in Z_{r,n}^0}
\\
&=
\frac{\det\big(K^{(r-1,n)}_{\mu,\nu}(a_{\widehat{1},\widehat{r+1}})\big|_{a_{r+1}=a_1^{-1}}\big)_{\mu,\nu\in Z_{r-1,n}}
}{
\prod_{\mu\in Z_{r,n}^0}
\prod_{k=2}^{r}
e(a_k,a_1;p)_{t,\mu_k}e(a_k,a_1;q)_{t,\mu_k}}
\\
&=
\frac{K_{r-1,n}(a_{\widehat{1},\widehat{r+1}})
}{\prod_{j=0}^n
\prod_{l=2}^{r}
\big(e(a_1,t^ja_l;p)e(a_1,t^ja_l;q)\big)^{\binom{n-j+r-3}{r-2}}
}. 
\end{split}
\end{equation*}
(In the right-hand side, 
$a_m$ should be understood as 
$a_m=pq/t^{2n-2}a_2\cdots a_r{a}_{r+2}\cdots a_{m-1}$.)
\par\medskip
We next consider the case where 
$\mu_1>0$ and $\nu_1>0$. 
From \eqref{eq:Kmunu=} we compute
\begin{equation*}
\begin{split}
&\lim_{a_{r+1}\to a_1^{-1}}(1-a_1a_{r+1})K^{(r,n)}_{\mu,\nu}(a)
\\
&=
\frac{2n\,
\prod_{1\le k\le m;\,k\ne 1,r+1}
\Gamma(a_1^{\pm1}a_k;p,q)
}{(p;p)_\infty^2(q;q)_\infty^2}
\\
&\quad\cdot
\lim_{a_{r+1}\to a_1^{-1}}
\int_{C^{n-1}}
E_{\mu}(a_{\pr{1,\ldots,r}};a_1,z_{\widehat{1}};p)
E_{\nu}(a_{\pr{1,\ldots,r}};a_1,z_{\widehat{1}};q)
\widehat{\Phi}_{n-1}(z_{\widehat{1}};a)
\omega_{n-1}(z_{\widehat{1}}). 
\end{split}
\end{equation*}
Here, by the property of interpolation functions, we have 
\begin{equation*}
\begin{split}
E_{\mu}(a_{\pr{1,\ldots,r}};a_1,z_{\widehat{1}};p)&=
E_{\mu-\ep_1}(t^{\ep_1}a_{\pr{1,\ldots,r}};z_{\widehat{1}};p),
\\ 
E_{\nu}(a_{\pr{1,\ldots,r}};a_1,z_{\widehat{1}};q)&=
E_{\nu-\ep_1}(t^{\ep_1}a_{\pr{1,\ldots,r}};z_{\widehat{1}};q).  
\end{split}
\end{equation*}
Also, noting that 
\begin{equation*}
\begin{split}
\lim_{a_{r+1}\to a_1^{-1}}\widehat{\Phi}_{n-1}(z_{\widehat{1}};a)
&=
\Phi_{n-1}(z_{\widehat{1}};a)
\big|_{a_1\to ta_1, a_{r+1}\to ta_1^{-1}}
\\
&=\Phi_{n-1}(z_{\widehat{1}};
ta_1,a_2,\ldots,a_r,ta_1^{-1},a_{r+2},\ldots)
\end{split}
\end{equation*}
and that 
$a_m=pq/t^{2n-2}a_2\cdots a_r a_{r+2}\cdots a_{m-1}
=pq/t^{2n-4}(ta_1)a_2\cdots a_r(ta_{1}^{-1})a_{r+2} \cdots a_{m-1}$, 
we compute
\begin{equation*}
\begin{split}
&\lim_{a_{r+1}\to a_1^{-1}}(1-a_1a_{r+1})K^{(r,n)}_{\mu,\nu}(a)
\\
&=
\frac{2n\,
\prod_{1\le k\le m;\,k\ne 1,r+1}
\Gamma(a_1^{\pm1}a_k;p,q)
}{(p;p)_\infty^2(q;q)_\infty^2}
K^{(r,n-1)}_{\mu-\ep_1,\nu-\ep_1}(a)
\big|_{a_1\to ta_1, a_{r+1}\to ta_1^{-1}}. 
\end{split}
\end{equation*}
Passing to the determinant, we obtain
\begin{equation*}
\begin{split}
&\det\Big
(\lim_{a_{r+1}\to a_1^{-1}}(1-a_1a_{r+1})K^{(r,n)}_{\mu,\nu}(a)
\Big)_{\mu,\nu\in Z_{r,n}^+}
\\
&=
\bigg(\frac{2n\,
\prod_{1\le k\le m;\,k\ne 1,r+1}
\Gamma(a_1^{\pm1}a_k;p,q)
}{(p;p)_\infty^2(q;q)_\infty^2}\bigg)^{\binom{n+r-2}{r-1}}
\det K^{(r,n-1)}(a)\big|_{a_1\to ta_1, a_{r+1}\to ta_1^{-1}}. 
\end{split}
\end{equation*}
\par\medskip
Summarizing the arguments above, we obtain
\begin{equation}\label{eq:wKrec}
\begin{split}
\widetilde{K}_{r,n}(\widetilde{a})&=
\lim_{a_{r+1}\to a_1^{-1}}(1-a_1a_{r+1})^{\binom{n+r-2}{r-1}}
\det K^{(r,n)}(a)
\\
&=
\bigg(
\frac{2n}
{(p;p)_\infty^2(q;q)_\infty^2}\bigg)^{\binom{n+r-2}{r-1}}
\frac{\prod_{1\le k\le m;\,k\ne 1,r+1}
\Gamma(a_1^{\pm1}a_k;p,q)^{\binom{n+r-2}{r-1}}
}{\prod_{j=0}^{n-1}
\prod_{l=2}^{r}
\big(e(a_1,t^ja_l;p)e(a_1,t^j;q)\big)^{\binom{n-j+r-3}{r-2}}}
\\
&\quad\cdot
\det
K^{(r-1,n)}(a_{\widehat{1},\widehat{r+1}})\, 
\det K^{(r,n-1)}(a)
\big|_{a_1\to ta_1, a_{r+1}\to ta_1^{-1}}
\qquad(n\ge 1).  
\end{split}
\end{equation}
For $n=1$ we understand 
$\det K^{(r,0)}(a)=1$. 
This computation is valid also for $r=1$: 
\begin{equation*}
\widetilde{K}_{1,n}(\widetilde{a})=
\frac{2n\prod_{k=3}^{6}\Gamma(a_1^{\pm1}a_k;p,q)}{(p;p)_\infty^2(q;q)_\infty^2}
\det K^{(1,n-1)}(ta_1,ta_1^{-1},a_3,a_4,a_5,a_6)\quad(n\ge 1). 
\end{equation*}
Here we understand  $\det K^{(1,0)}(a)=1$.
\subsection{Determination of $c_{r,n}$}
In order to compare 
$\widetilde{K}_{r,n}(\widetilde{a})$ 
with 
$\widetilde{L}_{r,n}(\widetilde{a})$, 
we compute 
\begin{equation*}
\frac{\widetilde{L}_{r,n}(\widetilde{a})}
{L_{r-1,n}(a_{\widehat{1},\widehat{r+1}})\,
L_{r,n-1}(a)\big|_{a_1\to ta_1, a_{r+1}\to ta_1^{-1}}. 
}
\end{equation*}
The three factors in this expression are given as follows: 
\begin{equation*}
\begin{split}
\widetilde{L}_{r,n}(\widetilde{a})
&=
\frac{
\prod_{i=1}^{n-1}
\Gamma(t^i;p,q)^{\binom{n-i+r-2}{r-1}}
}{
\left((p;p)_\infty(q;q)_\infty\right)^{\binom{n+r-2}{r-1}}
}
\,
\prod_{i=0}^{n-1}
\prod_{\substack{1\le k\le m\\ k\ne 1,r+1}}
\Gamma(t^{i}a_1^{\pm1}a_k;p,q)^{\binom{n-i+r-2}{r-1}}
\\
&\quad\cdot
\frac{
\prod_{i=0}^{n-1}
\prod_{\substack{1\le k<l\le m\\ k,l\ne 1,r+1}}
\Gamma(t^ia_ka_l;p,q)^{\binom{n-i+r-2}{r-1}}
}{
\prod_{0\le i+j<n}
\prod_{1\le k<l\le r}
\big(
e(t^ia_k,t^ja_l;p)
e(t^ia_k,t^ja_l;q)
\big)^{\binom{n-i-j+r-3}{r-2}}
},
\end{split}
\end{equation*}
\begin{equation*}
\begin{split}
L_{r-1,n}(a_{\widehat{1},\widehat{r+1}})
&=
\frac{
\prod_{i=0}^{n-1}
\prod_{\substack{1\le k<l\le m\\ k,l\ne 1,r+1}}
\Gamma(t^ia_ka_l;p,q)^{\binom{n-i+r-3}{r-2}}
}{
\prod_{0\le i+j<n}
\prod_{2\le k<l\le r}
\big(
e(t^ia_k,t^ja_l;p)
e(t^ia_k,t^ja_l;q)
\big)^{\binom{n-i-j+r-4}{r-3}}
},
\end{split}
\end{equation*}
\begin{equation*}
\begin{split}
L_{r,n-1}(a)\Big|_{\substack{a_1\to ta_1\\ a_{r+1}\to ta_1^{-1}}}
&=
\frac{\prod_{i=2}^{n}
\Gamma(t^{i};p,q)^{\binom{n-i+r-1}{r-1}}
\prod_{i=1}^{n-1}
\prod_{\substack{1\le k\le m\\ k\ne 1,r+1}}
\Gamma(t^{i}a_1^{\pm1}a_k;p,q)^{\binom{n-i+r-2}{r-1}}
}{
\prod_{i\ge1, j\ge0, 0\le i+j<n}
\prod_{l=2}^{r}
\big(
e(t^{i}a_1,t^ja_l;p)
e(t^{i}a_1,t^ja_l;q)
\big)^{\binom{n-i-j+r-3}{r-2}}
}
\\
&\quad\cdot
\frac{
\prod_{i=0}^{n-2}
\prod_{\substack{1\le k<l\le m\\ k,l\ne 1,r+1}}
\Gamma(t^ia_ka_l;p,q)^{\binom{n-i+r-3}{r-1}}
}{
\prod_{0\le i+j<n-1}
\prod_{2\le k<l\le r}
\big(
e(t^ia_k,t^ja_l;p)
e(t^ia_k,t^ja_l;q)
\big)^{\binom{n-i-j+r-4}{r-2}}
}.
\end{split}
\end{equation*}
Combining these formulas, for 
$r\ge2$ we have 
\begin{equation}\label{eq:wLrec}
\begin{split}
\widetilde{L}_{r,n}(\widetilde{a})
&=
\frac{1}{
\left((p;p)_\infty(q;q)_\infty\right)^{\binom{n+r-2}{r-1}}
}
\frac{\Gamma(t;p,q)^{\binom{n+r-2}{r-1}}}
{\prod_{i=1}^{n}\Gamma(t^i;p,q)^{\binom{n-i+r-2}{r-2}}}
\\
&\quad\cdot
\frac{
\prod_{1\le k\le m;\, k\ne 1,r+1}
\Gamma(a_1^{\pm1}a_k;p,q)^{\binom{n+r-2}{r-1}}
}{
\prod_{0\le j<n}
\prod_{l=2}^{r}
\big(
e(a_k,t^ja_l;p)
e(a_k,t^ja_l;q)
\big)^{\binom{n-j+r-3}{r-2}}
}
\\
&\quad\cdot 
L_{r-1,n}(a_{\widehat{1},\widehat{r+1}})\,
L_{r,n-1}(a)\big|_{a_1\to ta_1, a_{r+1}\to ta_1^{-1}}
\quad(n\ge 1),
\end{split}
\end{equation}
where $L_{r,0}(a)=1$, and 
for 
$r=1$ we have 
\begin{equation}\label{eq:wLrec1}
\begin{split}
\widetilde{L}_{1,n}(\widetilde{a})
&=
\frac{
\prod_{k=3}^{6}
\Gamma(a_1^{\pm1}a_k;p,q)}
{(p;p)_\infty(q;q)_\infty}
\frac{\Gamma(t;p,q)}
{\Gamma(t^n;p,q)}\,
L_{1,n-1}(ta_1,ta_1^{-1},a_3,\ldots,a_6)
\quad(n\ge 1),
\end{split}
\end{equation}
where $L_{1,0}(a)=1$. 
\begin{rem}{\rm
When 
$r=1$ $(m=6)$, by the balancing condition
$t^{2n-2}a_1\cdots a_6=pq$, 
we have 
$t^{2n-2}a_3a_4a_5a_6=pq$ 
in the limit 
$a_2\to a_1^{-1}$. 
Since $(t^{n-1}a_ia_j)(t^{n-1}a_ka_l)=pq$ 
for $\pr{i,j,k,l}=\pr{3,4,5,6}$, 
$
\Gamma(t^{n-1}a_ia_j;p,q)
\Gamma(t^{n-1}a_ka_l;p,q)=1$.  
Hence 
we have 
$\prod_{3\le k<l\le 6}\Gamma(t^{n-1}a_ka_l;p,q)=1$. 
}
\end{rem}
\par
Recall that $\det K^{(r,n)}(a)$ and $L_{r,n}(a)$ are related 
through the formulas 
\begin{equation*}
\det K^{(r,n)}(a)
=c_{r,n}L_{r,n}(a),\quad
\widetilde{K}_{r,n}(\widetilde{a})
=c_{r,n}\widetilde{L}_{r,n}(\widetilde{a}).  
\end{equation*}
Hence, by 
combining 
\eqref{eq:wKrec} and \eqref{eq:wLrec}, 
we obtain the following recurrence formulas for 
$c_{r,n}$: 
\begin{equation*}
c_{1,n}=c_{1,n-1}\,\frac{2n}{(p;p)_\infty(q;q)_\infty}
\frac{\Gamma(t^n;p,q)}{\Gamma(t;p,q)}\quad
(n\ge 1),\quad c_{1,0}=1, 
\end{equation*}
for $r=1$, and 
\begin{equation*}
c_{r,n}=c_{r-1,n}\,c_{r,n-1}
\left(\frac{2n}{(p;p)_\infty(q;q)_\infty}\right)^{\binom{n+r-2}{r-1}}
\frac
{\prod_{i=1}^{n}\Gamma(t^i;p,q)^{\binom{n-i+r-2}{r-2}}}
{\Gamma(t;p,q)^{\binom{n+r-2}{r-1}}},\quad c_{r,0}=1
\end{equation*}
for $r\ge2$. 
Solving these recurrence formulas, we obtain 
\begin{equation*}
\begin{split}
c_{r,n}&=
\left(
\frac{2^n n!}
{(p;p)_{\infty}^n (q;q)_\infty^n}
\right)^{\binom{n+r-1}{r-1}}
\frac{
\prod_{i=1}^{n}\Gamma(t^i;p,q)^{r\binom{n-i+r-1}{r-1}}
}{\Gamma(t;p,q)^{r\binom{n+r-1}{r}}}
\\
&=
\left(
\frac{2^n n!}
{(p;p)_{\infty}^n (q;q)_\infty^n}
\right)^{\binom{n+r-1}{r-1}}
\prod_{i=1}^{n}
\left(\frac{\Gamma(t^i;p,q)}
{\Gamma(t;p,q)}\right)^{r\binom{n-i+r-1}{r-1}}
\end{split}
\end{equation*}
for $r\ge1$ and $n\ge1$. 
This completes the proof of 
Theorem \ref{thm:1A}, 
as is pointed out in the remark of Corollary \ref{cor:det K(axy)}.
\section{\boldmath Determinant formula for $q$-hypergeometric integrals of type $BC_n$}
\label{section:6}
In this section we derive a determinant formula for $q$-hypergeometric integrals of type $BC_n$ 
from Theorem \ref{thm:1A}. 
In view of the balancing condition $a_1\cdots a_m t^{2n-2} =pq$, 
we first replace $a_m$ with $pa_m$, and then take the limit $p\to 0$.
By this procedure, from $\Phi(z;a;p,q)$ of \eqref{eq:Phi} we obtain 
\begin{equation*}
\mathsf{\Phi}(z;a;q)=
\prod_{i=1}^{n}
\frac{
(z_i^{\pm 2};q)_\infty(qa_m^{-1}z_i^{\pm1};q)_\infty}
{\prod_{k=1}^{m-1}(a_kz_i^{\pm1};q)_\infty}
\prod_{1\le i<j\le n}
\frac{(z_i^{\pm1}z_j^{\pm1};q)_\infty}
{(tz_i^{\pm1}z_j^{\pm 1};q)_\infty}. 
\end{equation*}
\par
In place of $\mathcal{H}_{s-1,n}^{(p)}$ we use the $\mathbb{C}$-vector space of $W_n$-invariant Laurent polynomials 
in $z=(z_1,\ldots,z_n)$ of degree $\le s-1$ in each variable:
\begin{equation*}
\mathsf{H}_{s-1,n}=\{f(z)\in \mathbb{C}[z^{\pm}]^{W_n}\,|\, \deg_{z_i}f(z)\le s-1 \ (i=1,\ldots,n)\}.
\end{equation*}
For generic $c=(c_1,\ldots,c_s)\in (\mathbb{C}^*)^s$, 
there exists a unique basis $\{\mathsf{E}_\mu(c;z)\,|\, \mu\in Z_{s,n}\}$ of $\mathsf{H}_{s-1,n}$ 
satisfying the condition 
\begin{equation*}
\mathsf{E}_{\mu}(c;(c)_{t,\nu})=\delta_{\mu,\nu}\qquad(\mu,\nu\in Z_{s,n}). 
\end{equation*}
This interpolation basis is obtained from $\prmts{E_{\mu}(c;z;p)}{\mu\in Z_{s,n}}$ simply by taking the limit $p\to 0$:
\begin{equation*}
\mathsf{E}_\mu(c;z)=\lim_{p\to 0}E_{\mu}(c;z;p)
\qquad (\mu\in Z_{s,n}).
\end{equation*}
Otherwise this basis can be constructed by the method of Section \ref{section:2} by replacing $e(u,v;p)$ with its limit
\begin{equation*}
\mathsf{e}(u,v)=u^{-1}(1-uv)(1-uv^{-1})=u+u^{-1}-v-v^{-1}.
\end{equation*}
These polynomials $\mathsf{E}_\mu(c;z)$ $(\mu\in Z_{s,n})$ for $s=2$ are 
written as 
\begin{equation}
\label{eq:Es=2}
\mathsf{E}_{(r,n-r)}(c_1,c_2;z)=
\sum_{1\le i_1<\cdots<i_{r}\le n\atop 1\le j_1<\cdots<j_{n-r}\le n}
\prod_{k=1}^{r}\frac{\mathsf{e}(z_{i_k},c_{2} t^{i_k-k})}{\mathsf{e}(c_1t^{k-1},c_2 t^{i_k-k})}
\prod_{l=1}^{n-r}\frac{\mathsf{e}(z_{j_l},c_1t^{j_l-l})}{\mathsf{e}(c_{2}t^{l-1},c_1t^{j_l-l})}
\end{equation}
for $r=0,1,\ldots,n$, 
where the summation is taken over all pairs of sequences 
$1\le i_1<\cdots<i_r\le n$ and $1\le j_1<\cdots<j_{n-r}\le n$ such that 
$\{i_1,\ldots,i_r\}\cup\{j_1,\ldots,j_{n-r}\}=\{1,2,\ldots,n\}$.
The polynomials \eqref{eq:Es=2} are 
used in the study of Jackson integrals of type $BC_n$ \cite{Ito2}. 
The polynomials $\mathsf{E}_\mu(c;z)$ $(\mu\in Z_{s,n})$ for general $s$ were defined for the first time in the present paper.
\par
We assume $m=2r+4$ $(r=1,2,\ldots)$. 
Fixing generic 
parameters $x=(x_1,\ldots,x_r)$ and $y=(y_1,\ldots,y_r)$, 
we take the interpolation bases 
for the two vector spaces 
$\mathsf{H}_{s-1,n}$ and $\cH^{(q)}_{r-1,n}$
with respect to $x$ and $y$ 
respectively: 
\begin{equation*}
\mathsf{H}_{r-1,n}=\bigoplus_{\mu\in Z_{r,n}}\bC\,\mathsf{E}_{\mu}(x;z),
\quad
\cH^{(q)}_{r-1,n}=\bigoplus_{\mu\in Z_{r,n}}\bC\,E_{\mu}(y;z;q). 
\end{equation*}
For each pair $(\mu,\nu)\in Z_{r,n}\times Z_{r,n}$, 
we consider 
the $q$-hypergeometric integral
\begin{equation}
\label{tri-Kmunu(a;x,y)}
\begin{split}
\mathsf{K}_{\mu,\nu}(a;x,y)
&=
\la \mathsf{E}_{\mu}(x;z),E_{\nu}(y;z;q)\ra_{\mathsf{\Phi}}
\\
&=\int_{\bTR^n}
\mathsf{E}_{\mu}(x;z)\,E_{\nu}(y;z;q)\mathsf{\Phi}(z;a;q)\omega_{n}(z)
\qquad(\mu,\nu\in Z_{r,n}),  
\end{split}
\end{equation}
assuming that $|a_k|<1$ $(k=1,\ldots,m-1)$. 
By the limiting procedure as explained above, Theorem \ref{thm:1A} implies the following evaluation formula.
\begin{thm}\label{thm:1Aq}
Set $m=2r+4$ $(r=1,2,\ldots)$.  
Under the balancing condition 
$a_1\cdots a_{m-1}a_mt^{2n-2}=q$ for the parameters, 
the determinant of the 
$\binom{n+r-1}{r-1}\times \binom{n+r-1}{r-1}$ matrix 
$\mathsf{K}(a;x,y)=(\mathsf{K}_{\mu,\nu}(a;x,y))_{\mu,\nu\in Z_{r,n}}$ 
is given explicitly by 
\begin{equation*}
\begin{split}
\det \mathsf{K}(a;x,y)
&=
\bigg(
\frac{2^nn!}{(q;q)_\infty^n}
\bigg)^{\binom{n+r-1}{r-1}}
\prod_{i=0}^{n-1}
\bigg(\frac{(t;q)_\infty^r
\prod_{k=1}^{m-1}(q/t^i a_ka_m;q)_{\infty}
}
{(t^{i+1};q)_\infty^r
\prod_{1\le k<l\le m-1}
(t^ia_ka_l;q)_{\infty}
}
\bigg)^{\binom{n-i+r-2}{r-1}}
\\
&\qquad\cdot
\Bigg(
\prod_{0\le i+j<n}
\prod_{1\le k<l\le r}
\big(
\mathsf{e}(t^ix_k,t^jx_l)
e(t^iy_k,t^jy_l;q)
\big)^{\binom{n-i-j+r-3}{r-2}}
\Bigg)^{\!\!-1}.
\end{split}
\end{equation*}
\end{thm}
Note that the above formula with $r = 1$ (i.e., the integral \eqref{eq:detK(axy);r=1*} with $p\to 0$) is known as 
Gustafson's multivariate Nassrallah--Rahman integral \cite{G1}
\begin{equation}
\label{eq:GmNR}
\int_{\bTR^n}\mathsf{\Phi}(z;a;q)\omega_n(z)
=
\frac{2^nn!}{(q;q)_\infty^n}
\prod_{i=0}^{n-1}
\frac{(t;q)_\infty
\prod_{k=1}^{5}(q/t^i a_ka_6;q)_{\infty}
}
{(t^{i+1};q)_\infty
\prod_{1\le k<l\le 5}
(t^ia_ka_l;q)_{\infty}
},
\end{equation}
which recovers the Nassrallah--Rahman integral \cite{NR} when $n = 1$. 
In \cite{Ito2}, a proof of \eqref{eq:GmNR} is given by means of the polynomials \eqref{eq:Es=2}.

Furthermore, we can take the limit $a_{2r+3}\to 0$ in Theorem \ref{thm:1Aq}. 
Then, without balancing condition for the parameters $a_1,\ldots,a_{2r+2}$, we have the following.
\begin{cor}\label{cor:1Aq}
Let $\widetilde{\mathsf{K}}_{\mu,\nu}(a;x,y)$ be the $q$-hypergeometric integrals defined by \eqref{tri-Kmunu(a;x,y)} using 
\begin{equation*}
\widetilde{\mathsf{\Phi}}(z;a;q)=
\prod_{i=1}^{n}
\frac{
(z_i^{\pm 2};q)_\infty}
{\prod_{k=1}^{2r+2}(a_kz_i^{\pm1};q)_\infty}
\prod_{1\le i<j\le n}
\frac{(z_i^{\pm1}z_j^{\pm1};q)_\infty}
{(tz_i^{\pm1}z_j^{\pm 1};q)_\infty}
\quad(r=1,2,\ldots),
\end{equation*}
with $2r+2$ parameters, instead of $\mathsf{\Phi}(z;a;q)$.  
The determinant of the 
$\binom{n+r-1}{r-1}\times \binom{n+r-1}{r-1}$ matrix 
$\widetilde{\mathsf{K}}(a;x,y)=(\widetilde{\mathsf{K}}_{\mu,\nu}(a;x,y))_{\mu,\nu\in Z_{r,n}}$ 
is given by 
\begin{equation*}
\begin{split}
\det \widetilde{\mathsf{K}}(a;x,y)
&=
\bigg(
\frac{2^nn!}{(q;q)_\infty^n}
\bigg)^{\binom{n+r-1}{r-1}}
\prod_{i=0}^{n-1}
\bigg(\frac{(t;q)_\infty^r
(t^{2n-i-2} a_1a_2\cdots a_{2r+2};q)_{\infty}
}
{(t^{i+1};q)_\infty^r
\prod_{1\le k<l\le 2r+2}
(t^ia_ka_l;q)_{\infty}
}
\bigg)^{\binom{n-i+r-2}{r-1}}
\\
&\qquad\cdot
\Bigg(
\prod_{0\le i+j<n}
\prod_{1\le k<l\le r}
\big(
\mathsf{e}(t^ix_k,t^jx_l)
e(t^iy_k,t^jy_l;q)
\big)^{\binom{n-i-j+r-3}{r-2}}
\Bigg)^{\!\!-1}.
\end{split}
\end{equation*}
\end{cor}
The above formula with $r = 1$ is also known as 
Gustafson's multivariate Askey--Wilson integral \cite{G2}
\begin{equation}
\label{eq:GmAW}
\int_{\bTR^n}\widetilde{\mathsf{\Phi}}(z;a;q)\omega_n(z)
=
\frac{2^nn!}{(q;q)_\infty^n}
\prod_{i=0}^{n-1}
\frac{(t;q)_\infty
(t^{2n-i-2} a_1a_2a_3a_4;q)_{\infty}
}
{(t^{i+1};q)_\infty
\prod_{1\le k<l\le 4}
(t^ia_ka_l;q)_{\infty}
},
\end{equation}
which is the Askey--Wilson integral \cite{AW} when $n = 1$. 
In \cite{vD} and \cite{Ito1}, proofs of \eqref{eq:GmAW} are given by using degenerate cases of the polynomials \eqref{eq:Es=2}
\begin{equation}
\label{eq:Es=2d}
\sum_{1\le i_1<\cdots<i_{r}\le n}
\prod_{k=1}^{r}\mathsf{e}(z_{i_k},c_{1} t^{i_k-k})\quad (r=0,1,\ldots,n),
\end{equation}
which essentially coincide with the special cases of Okounkov's interpolation polynomials 
\cite[Theorem 5.2]{O1}, \cite{O2} {\em attached to single columns} of partitions. (For the polynomials \eqref{eq:Es=2d}, see also \cite[Introduction, p.~1073--p.1074]{AI2009} and \cite{KNS}.)
Corollary \ref{cor:1Aq} can be regarded as a $BC_n$ version of 
the determinant formula of Tarasov and Varchenko \cite{TV97} 
for $q$-hypergeometric integrals of type $A_n$. 
\par\medskip
As a basis of the vector space $\mathsf{H}_{r-1,n}$, we can also take the {\it symplectic Schur functions}  
\begin{equation*}
\chi_\lambda(z)=\frac{\det\big(z_j^{\lambda_k+n-k+1}-z_j^{-\lambda_k-n+k-1}\big)_{j,k=1}^n}{\det\big(z_j^{n-k+1}-z_j^{-n+k-1}\big)_{j,k=1}^n}
\end{equation*}
associated with the partitions 
$\lambda\in B_{r,n}=\{\lambda\in \mathbb{Z}^n\,|\, r-1\ge \lambda_1\ge \cdots \ge \lambda_n\ge 0\}$.
These functions $\chi_\lambda(z)$ 
are expanded in terms of 
our interpolation polynomials 
as
\begin{equation*}
\chi_\lambda(z)
=\sum_{\mu\in Z_{r,n}}c_{\lambda\mu}\mathsf{E}_\mu(x;z)
\quad (\lambda\in B_{r,n}),
\end{equation*}
where $c_{\lambda\mu}=\chi_\lambda((x)_{t,\mu})$. The determinant of the matrix 
$C=\big(c_{\lambda\mu}\big)_{\lambda\in B_{r,n},\mu\in Z_{r,n}}$ is given by 
\begin{equation*}
\det C=
\det \Big(\chi_\lambda((x)_{t,\mu})\Big)_{\lambda\in B_{r,n},\mu\in Z_{r,n}}
=
\prod_{0\le i+j<n}
\prod_{1\le k<l\le r}
\mathsf{e}(t^ix_k,t^jx_l)
^{\binom{n-i-j+r-3}{r-2}}, 
\end{equation*}
as is proved in \cite[Corollary 1.5]{AI2009} or \cite[Theorem 3.2 (3.6)]{IIO}, for instance. We define the matrix 
\begin{equation*}
\mathsf{X}(a;y)=\big(\la \chi_\lambda(z),E_{\nu}(y;z;q)\ra_{\mathsf{\Phi}}\big)_{\lambda\in B_{r,n},\nu\in Z_{r,n}}.
\end{equation*} 
Then we have 
$
\mathsf{X}(a;y)=C\,\mathsf{K}(a;x,y)
$, so that 
$
\det\mathsf{X}(a;y)=\det C \det\mathsf{K}(a;x,y)
$. 
Under the condition $a_1\cdots a_{2r+4}t^{2n-2}=q$, this implies  
\begin{equation*}
\begin{split}
\det \mathsf{X}(a;y)
&=
\bigg(
\frac{2^nn!}{(q;q)_\infty^n}
\bigg)^{\binom{n+r-1}{r-1}}
\prod_{i=0}^{n-1}
\bigg(\frac{(t;q)_\infty^r
\prod_{k=1}^{2r+3}(q/t^i a_ka_{2r+4};q)_{\infty}
}
{(t^{i+1};q)_\infty^r
\prod_{1\le k<l\le 2r+3}
(t^ia_ka_l;q)_{\infty}
}
\bigg)^{\binom{n-i+r-2}{r-1}}
\\
&\qquad\cdot
\Bigg(
\prod_{0\le i+j<n}
\prod_{1\le k<l\le r}
e(t^iy_k,t^jy_l;q)
^{\binom{n-i-j+r-3}{r-2}}
\Bigg)^{\!\!-1}.
\end{split}
\end{equation*}
Similarly we define 
\begin{equation*}
\widetilde{\mathsf{X}}(a;y)
=\big(\la \chi_\lambda(z),E_{\nu}(y;z;q)\ra_{\widetilde{\mathsf{\Phi}}}\big)_{\lambda\in B_{r,n},\nu\in Z_{r,n}},
\end{equation*} 
which satisfies $
\widetilde{\mathsf{X}}(a;y)=C\,\widetilde{\mathsf{K}}(a;x,y)
$, so that 
$
\det\widetilde{\mathsf{X}}(a;y)=\det C\,\det\widetilde{\mathsf{K}}(a;x,y)
$. 
Then we have 
\begin{equation*}
\begin{split}
\det \widetilde{\mathsf{X}}(a;y)
&=
\bigg(
\frac{2^nn!}{(q;q)_\infty^n}
\bigg)^{\binom{n+r-1}{r-1}}
\prod_{i=0}^{n-1}
\bigg(\frac{(t;q)_\infty^r
(t^{2n-i-2} a_1a_2\cdots a_{2r+2};q)_{\infty}
}
{(t^{i+1};q)_\infty^r
\prod_{1\le k<l\le 2r+2}
(t^ia_ka_l;q)_{\infty}
}
\bigg)^{\binom{n-i+r-2}{r-1}}
\\
&\qquad\cdot
\Bigg(
\prod_{0\le i+j<n}
\prod_{1\le k<l\le r}
e(t^iy_k,t^jy_l;q)
^{\binom{n-i-j+r-3}{r-2}}
\Bigg)^{\!\!-1}.
\end{split}
\end{equation*}
This determinant formula is a contour integral version of the formula 
(\cite[Theorem 1.2]{INSlaterBC}
or 
\cite[Theorem 1.3]{AI2009}) 
for Jackson integrals of type $BC_n$. 
\par
The elliptic version of the determinant formulas for Jackson integrals of type $BC_n$ 
has not been established yet. It would be an important problem to clarify 
the relationship between Jackson integrals and contour integrals in the context of elliptic hypergeometric pairings
as in this paper. 

\section*{Acknowledgment}
This work was supported by JSPS Kakenhi Grants (B)15H03626 and (C)18K03339. 


\end{document}